%% file: AAMain.tex
\documentclass[12pt]{amsart}

\pdfoutput=1
\def\inmain{1}
\def\usehyperref{1}

\input{"Preamble"}

\usepackage{standalone} 

\usepackage[title, titletoc]{appendix}

\title{Tannaka duality and $1$-affineness}

\author{G. Stefanich}

\date{}

\begin{document}

%%%%%%%%%%%%%%%%%%%%%%%%%%%%%%%%%%%%%%%%%%%%%%%%%%%%%%%%%%%%%%%%%%%%%%%%
%%%%%%%%%%%%%%%%%%%%%%%%%%%%%%%%%%%%%%%%%%%%%%%%%%%%%%%%%%%%%%%%%%%%%%%%
%%%%%%%%%%%%%%%%%%%%%%%%%%%%%%%%%%%%%%%%%%%%%%%%%%%%%%%%%%%%%%%%%%%%%%%%
%%%%%%%%%%%%%%%%%%%%%%%%%%%%%%%%%%%%%%%%%%%%%%%%%%%%%%%%%%%%%%%%%%%%%%%%
%%%%%%%%%%%%%%%%%%%%%%%%%%%%%%%%%%%%%%%%%%%%%%%%%%%%%%%%%%%%%%%%%%%%%%%%
%%%%%%%%%%%%%%%%%%%%%%%%%%%%%%%%%%%%%%%%%%%%%%%%%%%%%%%%%%%%%%%%%%%%%%%%

\begin{abstract}
We show that Lurie's results on Tannaka duality for geometric stacks hold without any tameness hypotheses. We deduce this as a consequence of an affineness theorem in the theory of sheaves of categories. This affineness result is also applied to the study of tensor product and integral transform formulas for categories of quasicoherent sheaves.
\end{abstract}
 
\maketitle

\tableofcontents

\newpage

\input{Introduction}

\input{Preliminaries}

\input{Affineness}

\input{ConsequencesAffineness}

\input{Classical}

\appendix

\input{LimitsAndColimits}

\bibliographystyle{myamsalpha2}
\bibliography{References}
 
\end{document}

%% file: Preamble.tex
\usepackage{amssymb, amsmath, amsthm}
\usepackage{calligra, mathrsfs}
\usepackage{upgreek}
\usepackage{graphicx}
\usepackage{url}
\usepackage{mathtools}
\usepackage{enumerate}
\usepackage{verbatim}
\usepackage[retainorgcmds]{IEEEtrantools}
\usepackage{tikz-cd}
\usetikzlibrary{positioning}
\usepackage{microtype}

\usepackage[margin=1in,marginparwidth=0.8in, marginparsep=0.1in]{geometry}

\ifx\usehyperref\undefined
\usepackage{xr}
\newcommand\texorpdfstring[2]{#1}
\newcommand\nolinkurl[1]{\url{#1}}
\else
\usepackage{xr-hyper}
\usepackage[linktocpage=true, unicode]{hyperref}
\usepackage{xcolor}
\hypersetup{
    colorlinks,
    linkcolor={red!65!black},
    citecolor={blue!78!black}
    }
\fi

\newcommand{\Acal}{{\mathcal A}}
\newcommand{\Bcal}{{\mathcal B}}
\newcommand{\Ccal}{{\mathcal C}}
\newcommand{\Dcal}{{\mathcal D}}
\newcommand{\Ecal}{{\mathcal E}}
\newcommand{\Fcal}{{\mathcal F}}
\newcommand{\Gcal}{{\mathcal G}}

\newcommand{\Ical}{{\mathcal I}}
\newcommand{\Jcal}{{\mathcal J}}

\newcommand{\Mcal}{{\mathcal M}}
\newcommand{\Ncal}{{\mathcal N}}
\newcommand{\Ocal}{{\mathcal O}}
\newcommand{\Pcal}{{\mathcal P}}

\newcommand{\acal}{\Acal}
\newcommand{\ccal}{\Ccal}
\newcommand{\dcal}{\Dcal}
\newcommand{\ecal}{\Ecal}
\newcommand{\mcal}{\Mcal}

\newcommand{\ZZ}{{\mathbb Z}}

\usepackage{scalerel,stackengine}

\makeatletter
\g@addto@macro\bfseries{\boldmath}
\makeatother

\newcommand{\kr}{\kern -2pt}

\DeclareMathOperator{\id}{id}
\DeclareMathOperator{\Hom}{Hom}

\DeclareMathOperator{\Fun}{Fun}
\DeclareMathOperator{\Arr}{Arr}

\DeclareMathOperator{\der}{D}

\newcommand{\enh}{{\text{\normalfont enh}}}
\newcommand{\op}{{\text{\normalfont op}}}

\DeclareMathOperator{\Spc}{Spc}

\DeclareMathOperator{\Cat}{Cat}

\newcommand{\cathat}{\smash{{\vstretch{.8}{\widehat{\vstretch{1.25}{\Cat}}}}}}

\DeclareMathOperator{\colim}{colim}
\DeclareMathOperator{\Tot}{Tot}
\newcommand{\lex}{{\text{\normalfont lex}}}

\DeclareMathOperator{\Ind}{Ind}

\let\Pr\relax
\DeclareMathOperator{\Pr}{Pr}

\newcommand{\radj}{{\text{\normalfont radj}}}

\DeclareMathOperator{\Sp}{Sp}
\DeclareMathOperator{\Ab}{Ab}
\newcommand{\cn}{{\normalfont\text{cn}}}

\DeclareMathOperator{\CAlg}{CAlg}

\DeclareMathOperator{\Mod}{Mod}
\DeclareMathOperator{\LMod}{LMod}

\let \Bar \relax
\DeclareMathOperator{\Bar}{Bar}
\let\ss\relax
\newcommand{\ss}{{\normalfont \text{ss}}}

\DeclareMathOperator{\Groth}{Groth}

\newcommand{\comp}{{\normalfont\text{comp}}}
\newcommand{\otimeshat}{\mathbin{\smash{\widehat{\otimes}}}}
\newcommand{\Barhat}{\smash{{\vstretch{.8}{\widehat{\vstretch{1.25}{\Bar}}}}}}

\let \sc \relax
\newcommand{\sc}{{\normalfont \text{sc}}}
\newcommand{\cnorm}{{\normalfont\text{c}}}
\newcommand{\acnorm}{{\normalfont\text{ac}}}

\DeclareMathOperator{\Core}{Core}

\DeclareMathOperator{\Aff}{Aff}
\DeclareMathOperator{\QAff}{QAff}
\DeclareMathOperator{\Spec}{Spec}

\DeclareMathOperator{\Stk}{Stk}
\DeclareMathOperator{\PreStk}{PreStk}

\DeclareMathOperator{\QCoh}{QCoh}
\newcommand{\twoQCoh}{{2\kr\QCoh}}
\DeclareMathOperator{\Rep}{Rep}

\newcommand{\ab}{{\normalfont \text{Ab}}}

\newcommand{\pst}{{\normalfont\text{PSt}}}
\newcommand{\St}{{\text{\normalfont St}}}
\newcommand{\st}{{\text{\normalfont st}}}
 
\newtheorem{proposition}[subsubsection]{Proposition}
\newtheorem{lemma}[subsubsection]{Lemma}
\newtheorem{theorem}[subsubsection]{Theorem}
\newtheorem{corollary}[subsubsection]{Corollary}

\newtheoremstyle{note}{8.0pt plus 2.0pt minus 4.0pt}{8.0pt plus 2.0pt minus 4.0pt}{}{}{\bfseries}{.}{.5em}{} 
\theoremstyle{note}
\newtheorem{example}[subsubsection]{Example}
\newtheorem{remark}[subsubsection]{Remark}
\newtheorem{notation}[subsubsection]{Notation}
\newtheorem{construction}[subsubsection]{Construction}

\newtheorem{definition}[subsubsection]{Definition}

\AtBeginDocument{\def\MR#1{}}

%% file: Introduction.tex
%%%%%%%%%%%%%%%%%%%%%%%%%%%%%%%%%%%%%%%%%%%%%%%%%%%%%%%%%%%%%%%%%%%%%%%%
%%%%%%%%%%%%%%%%%%%%%%%%%%%%%%%%%%%%%%%%%%%%%%%%%%%%%%%%%%%%%%%%%%%%%%%%
%%%%%%%%%%%%%%%%%%%%%%%%%%%%%%%%%%%%%%%%%%%%%%%%%%%%%%%%%%%%%%%%%%%%%%%%
%%%%%%%%%%%%%%%%%%%%%%%%%%%%%%%%%%%%%%%%%%%%%%%%%%%%%%%%%%%%%%%%%%%%%%%%
%%%%%%%%%%%%%%%%%%%%%%%%%%%%%%%%%%%%%%%%%%%%%%%%%%%%%%%%%%%%%%%%%%%%%%%%
%%%%%%%%%%%%%%%%%%%%%%%%%%%%%%%%%%%%%%%%%%%%%%%%%%%%%%%%%%%%%%%%%%%%%%%%
 
\section{Introduction}

Let $G$ be an affine group scheme over a commutative ring $k$. The Tannakian formalism is a web of ideas that relate $G$ and its symmetric monoidal category $\Rep(G)^\heartsuit$ of representations. In its most basic form, it states that the group of points $G(k)$ is isomorphic to the group of symmetric monoidal automorphisms of the forgetful functor from $\Rep(G)^\heartsuit$ to the category of $k$-modules. 

In geometric terms, the symmetric monoidal category $\Rep(G)^\heartsuit$ may be identified with the category $\QCoh(BG)^\heartsuit$ of quasicoherent sheaves on the classifying stack of $G$. From this point of view, the Tannakian formalism attempts to recover $BG$ from $\QCoh(BG)^\heartsuit$. The stack $BG$ is an example of a geometric stack, by which we will mean an algebraic stack modelled fpqc-locally on affine schemes. A natural question is then to what extent a geometric stack $X$ may be recovered from the symmetric monoidal category $\QCoh(X)^\heartsuit$. This is addressed by the following fundamental result of Lurie:

\begin{theorem}[\cite{LurieTannaka, SAG}] \label{theorem tannaka clasico lurie}
Let $X, Y$ be a pair of geometric stacks, and assume that $X$ is quasi-compact and has affine diagonal. Then the assignment $f \mapsto f^*$ provides an equivalence between the groupoid of maps $Y \rightarrow X$ and the groupoid of colimit preserving symmetric monoidal functors $F: \QCoh(X)^\heartsuit \rightarrow \QCoh(Y)^\heartsuit$ satisfying the following two conditions:
\begin{enumerate}[\normalfont (i)]
\item $F$ sends flat sheaves to flat sheaves.
\item Let $0 \rightarrow \Fcal' \rightarrow \Fcal \rightarrow \Fcal'' \rightarrow 0$ be an exact sequence in $\QCoh(X)^\heartsuit$ such that $\Fcal''$ is flat. Then the sequence $0 \rightarrow F(\Fcal') \rightarrow F(\Fcal) \rightarrow F(\Fcal'') \rightarrow 0$ is exact.
\end{enumerate}
\end{theorem}

Theorem \ref{theorem tannaka clasico lurie} was extended in \cite{DAGVIII} to the context of spectral geometric stacks, that is, stacks modelled fpqc-locally on spectra of connective $E_\infty$-rings. In this setting, the symmetric monoidal  category $\QCoh(X)^\heartsuit$ is replaced by the symmetric monoidal $\infty$-category $\QCoh(X)^\cn$ of connective quasicoherent sheaves on $X$. Even for classical geometric stacks, replacing $\QCoh(X)^\heartsuit$ with $\QCoh(X)^\cn$ removes the need for condition (ii) in theorem \ref{theorem tannaka clasico lurie}:

\begin{theorem}[\cite{DAGVIII}]\label{theorem tannaka spectral lurie}
Let $X, Y$ be a pair of spectral geometric stacks, and assume that $X$ is quasi-compact and has affine diagonal. Then the assignment $f \mapsto f^*$ provides an equivalence between the space of maps $Y \rightarrow X$ and the space of colimit preserving symmetric monoidal functors $F: \QCoh(X)^\cn \rightarrow \QCoh(Y)^\cn$ satisfying the following condition:
\begin{enumerate}[\normalfont (i)]
\item $F$ sends flat sheaves to flat sheaves. 
\end{enumerate}
\end{theorem}

Conditions (i) and (ii) in the above theorems are called tameness conditions. There has been interest over the years in removing these conditions from the statements. With regard to theorem \ref{theorem tannaka clasico lurie}, tameness may be removed in the case of schemes by work of Brandenburg and Chirvasitu \cite{BrandenburgTensorial, BCTensor}, and in the case of geometric stacks with the resolution property by work of Sch\"{a}ppi \cite{SchappiCharacterization}.  In the setting of Noetherian Artin stacks, Hall and Rydh \cite{HallRydh} showed that tameness may be replaced with the condition that $F$ preserves coherent sheaves.

Some progress has also been made in removing tameness from theorem \ref{theorem tannaka spectral lurie}. This has been accomplished in the case of algebraic spaces by Bhatt \cite{BhattAlgebraization}, and in the case of spectral geometric stacks with the resolution property in work of Lurie \cite{SAG}. In the context of Noetherian spectral geometric stacks with quasi-affine diagonal, Bhatt and Halpern-Leistner \cite{BhattHalpernLeistner} show that tameness may be replaced with the condition that $F$ preserves pseudo-coherent sheaves, a condition that may be also be removed if the diagonal is affine \cite{SAG}.

The goal of this paper is to show that theorems \ref{theorem tannaka clasico lurie} and \ref{theorem tannaka spectral lurie} hold in all generality  without any tameness hypotheses. Moreover, we show that the spectral version continues to hold when the diagonal is only assumed to be quasi-affine:

\begin{theorem}\label{theorem intro tannaka}
Let $X, Y$ be a pair of spectral geometric stacks, and assume that $X$ is quasi-compact and has quasi-affine diagonal. 
\begin{enumerate}[\normalfont (1)]
\item The assignment $f \mapsto f^*$ provides an equivalence between the space of maps $Y \rightarrow X$ and the space of colimit preserving symmetric monoidal functors $F: \QCoh(X)^\cn \rightarrow \QCoh(Y)^\cn$. 
\item If $X$ and $Y$ are classical and $X$ has affine diagonal then the above are also equivalent to the  space of colimit preserving symmetric monoidal functors $F: \QCoh(X)^\heartsuit \rightarrow \QCoh(Y)^\heartsuit$.
\end{enumerate}
\end{theorem}

Our proof of theorem \ref{theorem intro tannaka} will make use of the theory of sheaves of categories. For simplicity, let us first focus on the classical setting. Given a geometric stack $X$, there is a notion of quasicoherent sheaf of Grothendieck abelian categories on $X$: roughly speaking, this consists of a compatible assignment of an $R$-linear Grothendieck abelian category to each affine chart $\Spec(R) \rightarrow X$. The totality of such sheaves assemble into a symmetric monoidal $2$-category denoted $\twoQCoh^{\Ab}(X)$, which we think of as a categorification of the symmetric monoidal category $\QCoh(X)^\heartsuit$.

Just like ordinary quasicoherent sheaves, the categorical version admits pullback and pushforward functoriality, and in particular there is a global sections functor $\Gamma(X, -)$ from $\twoQCoh^{\Ab}(X)$ into the $2$-category $\Groth_1$ of Grothendieck abelian categories. For each object $\ccal$ in $\twoQCoh^{\Ab}(X)$, the category of global sections $\Gamma(X, \ccal)$ has a canonical action of the symmetric monoidal category $\QCoh(X)^\heartsuit$, giving rise to a functor
\[
\Gamma^\enh(X,-): \twoQCoh^{\Ab}(X) \rightarrow \Mod_{\QCoh(X)^\heartsuit}(\Groth_1).
\]

From the point of view of categorical sheaf theory, the simplest stacks are those for which the above functor is an equivalence. This is a categorification of the condition that the category $\QCoh(X)^\heartsuit$ be equivalent to the category of modules over the ring of functions on $X$, which in the case when $X$ is a quasi-compact scheme holds if and only if $X$ is affine. By analogy, we say that a geometric stack $X$ is $\twoQCoh^{\Ab}$-affine if $\Gamma^\enh(X, -)$ is an equivalence.

In the context of quasicoherent sheaves of presentable stable $\infty$-categories, the question of affineness was studied by Gaitsgory \cite{G}, where a number of positive results were established. The basic philosophy is that many geometric objects become affine after categorification; this observation is supported by \cite{Thesis} where we showed that many Artin $n$-stacks are affine with respect to the theory of quasicoherent sheaves of  presentable stable $(\infty,n)$-categories.

The notion of quasicoherent sheaf of Grothendieck abelian categories was introduced in \cite{SAG}, where the functor $\Gamma^\enh(X,-)$ is shown to be fully faithful whenever $X$ is quasi-compact and has affine diagonal. Its image was characterized by a pair of tameness conditions, similar to those in theorem \ref{theorem tannaka clasico lurie}. Our next result removes these conditions:

\begin{theorem}\label{theorem affineness abelian intro}
Let $X$ be a quasi-compact geometric stack with affine diagonal. Then $\Gamma^\enh(X, -)$ is an equivalence.
\end{theorem}

To deduce part (1) of theorem \ref{theorem intro tannaka} one needs to replace Grothendieck abelian categories with complete Grothendieck prestable $\infty$-categories. This is a notion introduced by Lurie in \cite{SAG} which provides a convenient framework for working with presentable stable $\infty$-categories with complete t-structures. In the same way that for any geometric stack $X$ the category $\QCoh(X)^\heartsuit$ is Grothendieck abelian, for any spectral geometric stack $X$ the $\infty$-category $\QCoh(X)^\cn$ is complete Grothendieck prestable. 

To each spectral geometric stack $X$ one may attach a symmetric monoidal $\infty$-category $\twoQCoh^{\pst}_\comp(X)$ of quasicoherent sheaves of complete Grothendieck prestable $\infty$-categories on $X$, and there is a corresponding global sections functor
\[
\Gamma^\enh_\comp(X,-): \twoQCoh^{\pst}_\comp(X) \rightarrow \Mod_{\QCoh(X)^\cn}(\Groth_\comp).
\]
In \cite{SAG}, Lurie shows that $\Gamma^\enh_\comp(X, -)$ is fully faithful whenever $X$ is quasi-compact with affine diagonal, and provides a description of its image in terms of a tameness condition. Our next result generalizes this to the case of quasi-affine diagonals, and removes all tameness hypotheses:

\begin{theorem}\label{theorem affineness prestable intro}
Let $X$ be a quasi-compact spectral geometric stack with quasi-affine diagonal. Then $\Gamma^\enh_\comp(X, -)$ is an equivalence.
\end{theorem}

In addition to Tannaka duality, theorems \ref{theorem affineness abelian intro} and \ref{theorem affineness prestable intro} may be used to deduce tensor product formulas for categories of quasicoherent sheaves (a particular case of which is in fact also needed for their proof). The basic setup is the following: given maps of geometric stacks $X \rightarrow Y \leftarrow Z$, one wishes to give a description of the category of quasicoherent sheaves on $X \times_Y Z$ as a relative tensor product of the categories of quasicoherent sheaves on $X$ and $Z$. A prototypical result of this kind was established by Ben-Zvi, Francis and Nadler in \cite{BZFN}, where they show that at the level of quasicoherent sheaves of spectra one has an equivalence $\QCoh(X \times_Y Z) = \QCoh(X) \otimes_{\QCoh(Y)} \QCoh(Z)$ whenever $X, Y, Z$ belong to the class of so-called perfect stacks. In the classical setting, a similar formula was established by  Brandenburg \cite{Brandenburg} in the case of quasi-compact quasi-separated schemes, and by Sch\"{a}ppi \cite{SchappiInd, SchappiWhich} in the case of quasi-compact geometric stacks with affine diagonal having the resolution property. Our main theorem on this topic is the following:

\begin{theorem}\label{theorem tensor products introduction}
Let $X \rightarrow Y \leftarrow Z$ be maps of quasi-compact spectral geometric stacks with quasi-affine diagonal. 
\begin{enumerate}[\normalfont (1)]
\item Equip the presentable stable $\infty$-category $\QCoh(X) \otimes_{\QCoh(Y)} \QCoh(Z)$ with the t-structure whose connective half is generated under colimits and extensions by the objects of the form $\Fcal \otimes \Gcal$ with $\Fcal$ and $\Gcal$ connective. Then $\QCoh(X \times_Y Z)$ is equivalent to the left completion of $\QCoh(X) \otimes_{\QCoh(Y)} \QCoh(Z)$.
\item If the structure sheaves $\Ocal_X, \Ocal_Y$ are compact and truncated, then 
\[
\QCoh(X \times_Y Z) = \QCoh(X) \otimes_{\QCoh(Y)} \QCoh(Z).
\]
\item If $Y$ has affine diagonal then 
\[
\QCoh(X \times_Y Z)^\heartsuit = \QCoh(X)^\heartsuit \otimes_{\QCoh(Y)^\heartsuit} \QCoh(Z)^\heartsuit.
\]
\end{enumerate}
\end{theorem}

These tensor product formulas may be used to deduce integral transform formulas. In this case, one wishes to identify the category of quasicoherent sheaves on $X \times_Y Z$ with a category of functors between the categories of quasicoherent sheaves on $X$ and $Z$. Ben-Zvi, Francis and Nadler in \cite{BZFN} address the case of perfect stacks, in which they construct an equivalence between $\QCoh(X \times_Y Z)$  and the $\infty$-category $\Fun^L_{\QCoh(Y)}(\QCoh(X), \QCoh(Z))$ of $\QCoh(Y)$-linear colimit preserving functors from $\QCoh(X)$ to $\QCoh(Z)$. Theorem \ref{theorem tensor products introduction} allows us to deduce the following variant:

\begin{corollary}\label{corollary integral transforms introduction}
Let $X \rightarrow Y \leftarrow Z$ be maps of quasi-compact spectral geometric stacks with quasi-affine diagonal. Assume that $\Ocal_X$ and $\Ocal_Y$ are compact and truncated. Then there is an equivalence
\[
\QCoh(X \times_Y Z) = \Fun^L_{\QCoh(Y)}(\QCoh(X), \QCoh(Z)).
\]
\end{corollary}

Tensor product formulas are also relevant to the computation of higher traces. As discussed in \cite{BZFN}, if $X$ is a perfect stack and $n \geq 0$, one may identify the $E_n$-Hochschild homology of $\QCoh(X)$ with $\QCoh(X^{S^n})$. More generally, a similar formula computes the factorization homology of $\QCoh(X)$ over an arbitrary compact manifold. Theorem \ref{theorem tensor products introduction} may be used to obtain a variant of these computations that hold beyond the perfect setting:

\begin{corollary}\label{corollary factorization homology introduction}
Let $X$ be a quasi-compact spectral geometric stack with quasi-affine diagonal and let $M$ be a compact manifold. Then there is a canonical t-structure on the factorization homology $\int_M \QCoh(X)$ of the $E_\infty$ presentable stable $\infty$-category $\QCoh(X)$, whose left completion is equivalent to $\QCoh(X^M)$.
\end{corollary}

%%%%%%%%%%%%%%%%%%%%%%%%%%%%%%%%%%%%%%%%%%%%%%%%%%%%%%%%%%%%%%%%%%%%%%%%
%%%%%%%%%%%%%%%%%%%%%%%%%%%%%%%%%%%%%%%%%%%%%%%%%%%%%%%%%%%%%%%%%%%%%%%%
%%%%%%%%%%%%%%%%%%%%%%%%%%%%%%%%%%%%%%%%%%%%%%%%%%%%%%%%%%%%%%%%%%%%%%%%
%%%%%%%%%%%%%%%%%%%%%%%%%%%%%%%%%%%%%%%%%%%%%%%%%%%%%%%%%%%%%%%%%%%%%%%%
%%%%%%%%%%%%%%%%%%%%%%%%%%%%%%%%%%%%%%%%%%%%%%%%%%%%%%%%%%%%%%%%%%%%%%%%
%%%%%%%%%%%%%%%%%%%%%%%%%%%%%%%%%%%%%%%%%%%%%%%%%%%%%%%%%%%%%%%%%%%%%%%%

\addtocontents{toc}{\protect\setcounter{tocdepth}{2}}
\subsection{Conventions and notation}

Throughout the paper we use the convention where the word category stands for $\infty$-category, and use the term $(1,1)$-category or classical category if we wish to refer to the classical notion. Similarly, we will use the term geometric stack to refer to spectral stacks and say that a geometric stack is classical if it admits an atlas by spectra of $0$-truncated connective $E_\infty$-rings.

We denote by $\CAlg^{\cn}$ the category of connective ring spectra, and by $\Aff = (\CAlg^{\cn})^\op$ the category of affine schemes. We let $\PreStk$ be the category of accessible presheaves on $\Aff$, and we let $\Stk$ be the full subcategory of $\PreStk$ on the fpqc sheaves. Objects of $\PreStk$ and $\Stk$ are called prestacks and stacks, respectively.  The class of geometric stacks used in this paper is the closure of $\Aff$ inside $\Stk$ under small coproducts and geometric realizations of flat groupoids,  as in \cite{Simpson, TVgeometric}. Our main results are stated for quasi-compact geometric stacks with (quasi-)affine diagonal, in which case this notion agrees with the notion of (quasi-)geometric stack from \cite{SAG} sections 9.1 and 9.3.

Each category $\ccal$ has a Hom bifunctor which we denote by $\Hom_\ccal(-,-)$, whose target is the category $\Spc$ of homotopy types.  For each category $\ccal$ and each $n \geq -2$ we denote by $\ccal_{\leq n}$ the full subcategory of $\ccal$ on the $n$-truncated objects. If the inclusion $\ccal_{\leq n} \rightarrow \ccal$ admits a left adjoint, this will be denoted by $\tau_{\leq n}$.  In cases when $\ccal$ is a stable category equipped with a t-structure we will repurpose this notation by letting $\ccal_{\leq n}$ (resp.  $\ccal_{\geq n}$) be the full subcategory of $\ccal$ on the $n$-truncated (resp. $n$-connective) objects with respect to the t-structure, and $\tau_{\leq n}$ (resp. $\tau_{\geq n}$) be the corresponding truncation functor. We  denote by $\ccal^\cn = \ccal_{\geq 0}$ the full subcategory of $\ccal$ on the connective objects, and by $\ccal^\heartsuit$ the heart of the t-structure. 

We denote by $\Ab$ the category of abelian groups, and $\Sp$ the category of spectra. For each commutative ring spectrum $R$ we denote by $\Mod_R$ the category of $R$-module spectra. If $R$ is connective we will denote by $\Mod_R^\cn$ the full subcategory of $\Mod_R$ on the connective $R$-module spectra, and by $\Mod_R^\heartsuit$ the full subcategory of $\Mod_R^\cn$ on the $0$-truncated objects. This applies in particular to the case when $R$ is a (classical) commutative ring: in this case $\Mod_R^\heartsuit$ is the category of $R$-modules in abelian groups, while $\Mod_R$ is its derived category. Given a geometric stack $X$, we denote by $\QCoh(X)$ the category of quasicoherent sheaves of spectra on $X$, and by $\QCoh(X)^\cn$ (resp. $\QCoh(X)^\heartsuit$) the full subcategory on the connective (resp. connective and $0$-truncated) objects.

We fix a sequence of nested universes. Objects belonging to the first two universes are called small and large, respectively. We let $\Cat$ be the category of small categories and  $\cathat$ be the category of large categories. We denote by $\Pr^L$ the   subcategory of $\cathat$ on the presentable categories and colimit preserving functors, and by $\Pr^L_{\St}$ the full subcategory of $\Pr^L$ on the presentable stable categories.  We will frequently consider $\Pr^L$ and $\Pr^L_{\St}$ as symmetric monoidal categories as in \cite{HA} chapter 4.8, where for each pair of objects $\Ccal, \Dcal$ the tensor product $\Ccal \otimes \Dcal$ is the universal recipient of a  functor from $\Ccal \times \Dcal$ which preserves colimits in each coordinate. 

For each pair of categories $\ccal, \dcal$ we denote by $\Fun(\ccal, \dcal)$ the category of functors from $\ccal$ to $\dcal$. If $\ccal$ and $\dcal$ admit small colimits  we denote by $\Fun^L(\ccal, \dcal)$  the full subcategory of $\Fun(\ccal, \dcal)$ on those functors which preserve small colimits.

We use the term $2$-category to refer to $(\infty,2)$-categories. A commutative square
\[
\begin{tikzcd}
X' \arrow{d}{f'} \arrow{r}{g'} & X \arrow{d}{f} \\
Y' \arrow{r}{g} & Y
\end{tikzcd}
\]
in a $2$-category $\ccal$ is said to be vertically right adjointable if the arrows $f$ and $f'$ admit right adjoints $f^R, f'^R$, and the induced $2$-cell $g' \circ f'^R \rightarrow f^R \circ g$ is an isomorphism. One similarly defines the notions of vertical left adjointability  and horizontal left/right adjointability. This will frequently be applied in the particular case when $\ccal$ is the $2$-category of categories.

%%%%%%%%%%%%%%%%%%%%%%%%%%%%%%%%%%%%%%%%%%%%%%%%%%%%%%%%%%%%%%%%%%%%%%%%
%%%%%%%%%%%%%%%%%%%%%%%%%%%%%%%%%%%%%%%%%%%%%%%%%%%%%%%%%%%%%%%%%%%%%%%%
%%%%%%%%%%%%%%%%%%%%%%%%%%%%%%%%%%%%%%%%%%%%%%%%%%%%%%%%%%%%%%%%%%%%%%%%
%%%%%%%%%%%%%%%%%%%%%%%%%%%%%%%%%%%%%%%%%%%%%%%%%%%%%%%%%%%%%%%%%%%%%%%%
%%%%%%%%%%%%%%%%%%%%%%%%%%%%%%%%%%%%%%%%%%%%%%%%%%%%%%%%%%%%%%%%%%%%%%%%
%%%%%%%%%%%%%%%%%%%%%%%%%%%%%%%%%%%%%%%%%%%%%%%%%%%%%%%%%%%%%%%%%%%%%%%%
 
\subsection{Acknowledgments}

I am grateful to Jacob Lurie for an in-depth conversation on the results of this paper and related topics. I am also thankful to David Ben-Zvi for several discussions on Tannaka duality; to Bhargav Bhatt and Jacob Lurie for a conversation regarding weakenings of the quasi-affine diagonal assumption; and to David Nadler for suggesting the application to factorization homology.

\ifx\inmain\undefined
\bibliographystyle{myamsalpha}
\bibliography{References}
\fi

%% file: Preliminaries.tex
%%%%%%%%%%%%%%%%%%%%%%%%%%%%%%%%%%%%%%%%%%%%%%%%%%%%%%%%%%%%%%%%%%%%%%%%
%%%%%%%%%%%%%%%%%%%%%%%%%%%%%%%%%%%%%%%%%%%%%%%%%%%%%%%%%%%%%%%%%%%%%%%%
%%%%%%%%%%%%%%%%%%%%%%%%%%%%%%%%%%%%%%%%%%%%%%%%%%%%%%%%%%%%%%%%%%%%%%%%
%%%%%%%%%%%%%%%%%%%%%%%%%%%%%%%%%%%%%%%%%%%%%%%%%%%%%%%%%%%%%%%%%%%%%%%%
%%%%%%%%%%%%%%%%%%%%%%%%%%%%%%%%%%%%%%%%%%%%%%%%%%%%%%%%%%%%%%%%%%%%%%%%
%%%%%%%%%%%%%%%%%%%%%%%%%%%%%%%%%%%%%%%%%%%%%%%%%%%%%%%%%%%%%%%%%%%%%%%%

\section{Grothendieck prestable categories}\label{section prestable}

In \cite{SAG} appendix C, Lurie introduced the notion of Grothendieck prestable category, as an $\infty$-categorical version of the notion of Grothendieck abelian category. Roughly speaking, a Grothendieck prestable category is a category $\ccal$ which satisfies certain exactness properties similar to those present in the categories of connective modules over connective ring spectra. 

The theory of Grothendieck prestable categories provides a convenient language to study presentable stable categories equipped with t-structures: any Grothendieck prestable category $\Ccal$ is equivalent to the connective half of a t-structure on the category of spectrum objects $\Sp(\ccal)$, and furthermore every presentable stable category with a right complete t-structure compatible with filtered colimits is of this form. Throughout the paper we will often work with Grothendieck prestable categories satisfying a certain completeness condition which corresponds under this dictionary to the left completeness of t-structures. This is satisfied for instance whenever $\ccal =\QCoh(X)^\cn$ is the category of connective quasicoherent sheaves on a geometric stack $X$.

We begin this section in \ref{subsection groth} by reviewing various general features of the category $\Groth$ of Grothendieck prestable categories and colimit preserving functors, and of its localization $\Groth_\comp$ consisting of complete Grothendieck prestable categories. For the most part, this is a review of portions of \cite{SAG} appendix C; the new material consists of the notion of almost compactness, together with a result on the commutation of limits and colimits in $\Groth$.

The categories $\Groth$ and $\Groth_\comp$ admit canonical symmetric monoidal structures. To accomplish the goals of this paper we will need a robust theory of relative tensor products inside these. This presents some difficulty since the class of colimits at our disposal is limited. Indeed, we will only allow ourselves to form relative tensor products over algebras satisfying certain conditions that guarantee good behavior of Bar constructions. These conditions are $2$-categorical in nature, and are based in the notion of almost rigid algebra in a  monoidal $2$-category which we explore in \ref{subsection almost rigid}.

Although $\Groth_\comp$ has the structure of a symmetric monoidal $2$-category, and thus it makes sense to consider almost rigid algebras inside it, we will in fact need to work with a slight variant of $\Groth_\comp$ denoted by $\Groth^{\st}_\comp$, which we define in \ref{subsection grothstcomp}. In terms of t-structures, morphisms in $\Groth_\comp$ correspond to right t-exact functors, while  $\Groth^{\st}_\comp$ allows for functors which are only right t-exact up to a shift. This added generality will be necessary to establish our main result, since the functor of pushforward along the diagonal map of a geometric stack with quasi-affine diagonal is only right t-exact up to a shift.

Finally, in  \ref{subsection completed relative} we introduce the notion of admissible commutative algebra in $\Groth_\comp$, by which we will mean a commutative algebra which is almost rigid when regarded as an algebra in $\Groth^{\st}_\comp$. We discuss here various pleasant properties of admissible commutative algebras which follow from the general properties of almost rigid algebras established in \ref{subsection almost rigid}. In particular, we show that admissible commutative algebras support a good theory of completed relative tensor products.

%%%%%%%%%%%%%%%%%%%%%%%%%%%%%%%%%%%%%%%%%%%%%%%%%%%%%%%%%%%%%%%%%%%%%%%%
%%%%%%%%%%%%%%%%%%%%%%%%%%%%%%%%%%%%%%%%%%%%%%%%%%%%%%%%%%%%%%%%%%%%%%%%
%%%%%%%%%%%%%%%%%%%%%%%%%%%%%%%%%%%%%%%%%%%%%%%%%%%%%%%%%%%%%%%%%%%%%%%%
%%%%%%%%%%%%%%%%%%%%%%%%%%%%%%%%%%%%%%%%%%%%%%%%%%%%%%%%%%%%%%%%%%%%%%%%
%%%%%%%%%%%%%%%%%%%%%%%%%%%%%%%%%%%%%%%%%%%%%%%%%%%%%%%%%%%%%%%%%%%%%%%%
%%%%%%%%%%%%%%%%%%%%%%%%%%%%%%%%%%%%%%%%%%%%%%%%%%%%%%%%%%%%%%%%%%%%%%%%

\subsection{The category of Grothendieck prestable categories}\label{subsection groth}

We begin by recalling the notion of Grothendieck prestable category introduced in \cite{SAG} appendix C.

\begin{definition}\label{def prestable}
A Grothendieck prestable category is a presentable category $\Ccal$ satisfying the following properties:
\begin{enumerate}[\normalfont (a)]
\item The initial and final objects of $\Ccal$ agree (that is, $\Ccal$ is pointed).
\item Every cofiber sequence in $\Ccal$ is also a fiber sequence.
\item Every map in $\Ccal$ of the form $f: X \rightarrow \Sigma(Y)$ is the cofiber of its fiber.
\item Filtered colimits and finite limits commute in $\Ccal$.
\end{enumerate}

We denote by $\Groth$ the full subcategory of $\Pr^L$ on the Grothendieck prestable categories.
\end{definition}

For each Grothendieck prestable category $\Ccal$ the canonical functor from $\ccal$ into the category $\ccal \otimes \Sp = \Sp(\ccal)$  of spectrum objects in $\ccal$ is fully faithful, and identifies $\Ccal$ with the connective half of a t-structure on $\Sp(\Ccal)$. This t-structure is right complete (in other words, $\Sp(\ccal) = \lim_{n \to -\infty} \Sp(\ccal)_{\geq n}$) and compatible with filtered colimits (in other words, the truncation functors preserve filtered colimits). It turns out that the assignment $\Ccal \mapsto \Sp(\Ccal)$ provides an equivalence between $\Groth$ and the category whose objects are presentable stable categories equipped with a  right complete t-structure compatible with filtered colimits and whose morphisms are colimit preserving right t-exact functors (see \cite{SAG} corollary C.3.1.4 and remark C.3.1.5).
 
We now discuss a few properties of morphisms in $\Groth$ that will be of fundamental importance for us. The first one is left exactness, which admits several equivalent characterizations:
 
\begin{proposition}[\cite{SAG} proposition C.3.2.1]\label{proposition characteriz left exact}
Let $f: \ccal \rightarrow \dcal$ be a colimit preserving functor between Grothendieck prestable categories. Then the following are equivalent:
\begin{enumerate}[\normalfont (1)]
\item $f$ is left exact.
\item $f$ maps $0$-truncated objects of $\ccal$ to $0$-truncated objects of $\dcal$.
\item The functor $\Sp(f) : \Sp(\ccal) \rightarrow \Sp(\dcal)$ is left t-exact.
\end{enumerate}
\end{proposition}

The second property of morphisms of $\Groth$ that we will need is compactness, which also admits multiple descriptions:

\begin{proposition}[\cite{SAG} proposition C.3.4.1]\label{proposition equivalence compactness}
Let $f: \ccal \rightarrow \dcal$ be a colimit preserving functor between Grothendieck prestable categories. Then the following are equivalent:
\begin{enumerate}[\normalfont (1)]
\item $f$ is compact: in other words, the right adjoint to $f$ preserves small filtered colimits.
\item The right adjoint to $\Sp(f) : \Sp(\ccal) \rightarrow \Sp(\dcal)$ preserves small colimits.
\end{enumerate}
\end{proposition}

The last property that we will need is a weakening of compactness:

\begin{definition}
Let $f: \ccal \rightarrow \dcal$ be a colimit preserving functor between Grothendieck prestable categories. We say that $f$ is almost compact if the right adjoint to $f$ preserves small filtered colimits when restricted to $\dcal_{\leq n}$ for all $n \geq 0$.
\end{definition}

\begin{proposition}
Let $f: \ccal \rightarrow \dcal$ be a colimit preserving functor between Grothendieck prestable categories. Then the following are equivalent:
\begin{enumerate}[\normalfont (1)]
\item $f$ is almost compact.
\item The right adjoint to $\Sp(f)$ preserves filtered colimits when restricted to $\Sp(\dcal)_{\leq 0}$.
\item The right adjoint to $\Sp(f)$ preserves filtered colimits when restricted to $\Sp(\dcal)_{\leq n}$ for all integers $n$.
\end{enumerate}
\end{proposition}
\begin{proof}
Let $g$ be the right adjoint to $f$, and $G$ be the right adjoint to $\Sp(f)$. We first prove that (1) implies (2).  Since the t-structure on $\Sp(\ccal)$ is right complete, to prove (2) it will suffice to show that $\tau_{\geq -n} \circ G$ preserves filtered colimits when restricted to $\Sp(\dcal)_{\leq 0}$, for all $n \geq 0$. This is equivalent to $\tau_{\geq 0} \circ G$ preserving filtered colimits when restricted to $\Sp(\dcal)_{\leq n}$. Using the left t-exactness of $G$ we may identify $\tau_{\geq 0} \circ G$ with $g \circ \tau_{\geq 0}$. The fact that (2) holds now follows from (1) together with the fact that truncations on $\Sp(\dcal)$ preserve filtered colimits.

Assertion (2) is equivalent to (3), since $G$ is an exact functor. We finish the proof by showing that (3) implies (1). Since $g$ may be identified with $\tau_{\geq 0} \circ G|_{\dcal}$, to show that $f$ is almost compact it will suffice to show that $\tau_{\geq 0} \circ G$ preserves filtered colimits when restricted to $\dcal_{\leq n}$ for all $n \geq 0$. This follows from (3), since truncations on $\Sp(\ccal)$ preserve filtered colimits.
\end{proof}

The category $\Groth$ does not have all small limits nor colimits, see appendix \ref{appendix limits and colimits}. Nevertheless, it does have small limits (resp. colimits) of diagrams with left exact (resp. almost compact) transition functors.

\begin{notation}
We denote by $\Groth^\lex$ (resp. $\Groth^{\text{c}}$, resp. $\Groth^{\text{ac}}$) the wide subcategory of $\Groth$ on the left exact (resp. compact, resp. almost compact) functors.
\end{notation}

\begin{proposition} \label{prop colimits y limits} \hfill
\begin{enumerate}[\normalfont (1)]
\item The category $\Groth^\lex$ admits small limits, and these are preserved by the inclusions $\Groth^\lex \rightarrow \Groth \rightarrow \cathat$. 
\item The categories $\Groth^\cnorm$ and $\Groth^\acnorm$ admit small colimits, and these are preserved by the inclusions into $\Groth$.
\end{enumerate}
\end{proposition}
\begin{proof}
Item (1) and the compactness half of item (2) are \cite{SAG} propositions  C.3.2.4,  C.3.5.1 and C.3.5.3. The almost compactness half of item (2) is proven in an analogous way to the compactness half.
\end{proof}

\begin{remark}\label{remark limit along lex}
Let $(\ccal_\alpha)$ be a diagram in $\Groth^\lex$, and let $\ccal$ be its limit in $\Groth$. Then $\Sp(\ccal) = \lim \Sp(\Ccal_\alpha)$, and an object in $\Sp(\ccal)$ is connective (resp. coconnective) if and only if its image in $\Sp(\ccal_\alpha)$ is connective (resp. coconnective) for all $\alpha$. 
\end{remark}

\begin{remark}\label{remark colimit along compact}
Let $(\ccal_\alpha)$ be a diagram in $\Groth^\acnorm$, and let $\ccal$ be its colimit in $\Groth$. Then it follows from the proof of (the almost compact version of) \cite{SAG} proposition C.3.5.1 that $\Sp(\ccal)$ is the colimit in $\Pr^L$ of the diagram $(\Sp(\ccal_\alpha))$, so that $\Sp(\ccal)$ may be identified with the limit in $\cathat$ of the diagram induced from this by passage to right adjoints. Furthermore, an object in  $\Sp(\ccal)$ is coconnective if and only if its image in $\Sp(\ccal_\alpha)$  is coconnective for all $\alpha$.
\end{remark}

\begin{remark}
There is yet another property of morphisms in $\Groth$ which is sometimes relevant. We say that a morphism $f: \ccal \rightarrow \dcal$ is strongly compact if it admits a colimit preserving right adjoint; in terms of t-structures, this corresponds to the condition that $\Sp(f)$ admits a t-exact colimit preserving right adjoint. It turns out that colimits of diagrams with strongly compact transitions are preserved by the inclusion $\Groth \rightarrow \Pr^L$ (\cite{SAG} remark C.3.5.4). Working with compactness and almost compactness as opposed to strong compactness will be necessary  to establish our main results in the case of quasi-affine (as opposed to affine) diagonals.
\end{remark}

For later purposes we record the following commutation property of limits and colimits in $\Groth$:
\begin{proposition}\label{proposition commute limits and colimits}
Let $\Ical, \Jcal$ be small categories and let $F: \Ical \times \Jcal \rightarrow \Groth$ be a diagram. Assume the following:
\begin{enumerate}[\normalfont (a)]
\item For every $\alpha \rightarrow \alpha'$ in $\Ical$ and every $\beta$ in $\Jcal$ the map $F(\alpha, \beta) \rightarrow F(\alpha', \beta)$ is left exact.
\item For every $\alpha$ in $\Ical$ and every $\beta \rightarrow \beta'$ in $\Jcal$ the map $F(\alpha, \beta) \rightarrow F(\alpha, \beta')$ is almost compact.
\item For every $\alpha \rightarrow \alpha'$ in $\Ical$ and every $\beta \rightarrow \beta'$ in $\Jcal$ the commutative square
\[
\begin{tikzcd}
F(\alpha, \beta) \arrow{r}{} \arrow{d}{} & F(\alpha, \beta') \arrow{d}{} \\
F(\alpha', \beta) \arrow{r}{} & F(\alpha', \beta')
\end{tikzcd}
\]
is horizontally right adjointable.
\end{enumerate}
 For each $\alpha$ in $\Ical$ let $F(\alpha, \ast) =\colim_{\Jcal} F(\alpha, \beta)$, and for each $\beta$ in $\Jcal$ let $F(\ast, \beta) = \lim_{\Ical} F(\alpha, \beta)$. Then:
\begin{enumerate}[\normalfont (1)]
\item For each $\alpha \rightarrow \alpha'$ in $\Ical$ the map $F(\alpha, \ast) \rightarrow F(\alpha', \ast)$ is left exact.
\item For each $\beta \rightarrow \beta'$ in $\Jcal$ the map $F(\ast, \beta) \rightarrow F(\ast, \beta')$ is almost compact.  Furthermore, this map is compact if in assumption {\normalfont (b)} we replace almost compactness with compactness.
\item The canonical functor $\colim_{\Jcal} F(\ast, \beta) \rightarrow \lim_{\Ical} F(\alpha, \ast)$ is an equivalence.
\end{enumerate}
\end{proposition}

The proof of proposition \ref{proposition commute limits and colimits} requires the following:

\begin{lemma}\label{lemma vertical right adj and sp}
Let \[
\sigma: \begin{tikzcd}
\ccal' \arrow{d}{f'} \arrow{r}{g'} & \ccal \arrow{d}{f} \\
\dcal' \arrow{r}{g} & \dcal 
\end{tikzcd}
\]
be a commutative square of Grothendieck prestable categories and colimit preserving functors. Assume that $f'$ is left exact. Then $\sigma$ is horizontally right adjointable if and only if the square
\[
\Sp(\sigma): 
\begin{tikzcd}
\Sp(\ccal') \arrow{d}{\Sp(f')} \arrow{r}{\Sp(g')} & \Sp(\ccal )\arrow{d}{\Sp(f)} \\
\Sp(\dcal') \arrow{r}{\Sp(g)} & \Sp(\dcal) 
\end{tikzcd}
\]
is horizontally right adjointable.
\end{lemma}
\begin{proof}
Let $h$ and $h'$ be the right adjoints to $\Sp(g)$ and $\Sp(g')$, and note that the right adjoints to $g$ and $g'$ are given by (the restrictions to connective objects of) $\tau_{\geq 0} \circ h$ and $\tau_{\geq 0} \circ h'$, respectively. If $\Sp(\sigma)$ is horizontally right adjointable we have 
\[
\Sp(f') \circ \tau_{\geq 0} \circ h' = \tau_{\geq 0 } \circ \Sp(f') \circ h' = \tau_{\geq 0} \circ h \circ \Sp(f)
\]
which implies that $\sigma$ is horizontally right adjointable. Conversely, assume that $\sigma$ is horizontally right adjointable. Then for every integer $n$ we have
\[
\tau_{\geq n} \circ \Sp(f') \circ h' =  \Sp(f') \circ \tau_{\geq n} \circ h' = \tau_{\geq n} \circ h \circ \Sp(f).
\]
The fact that $\Sp(\sigma)$ is horizontally right adjointable follows from the fact that the t-structure on $\Sp(\Dcal')$ is right complete.
\end{proof}

\begin{proof}[Proof of proposition \ref{proposition commute limits and colimits}]
We begin with a proof of (1). By proposition \ref{proposition characteriz left exact} it suffices to show that the map $\Sp(F(\alpha, \ast)) \rightarrow \Sp(F(\alpha', \ast))$ is left t-exact. Using remark \ref{remark colimit along compact} we may (after passing to right adjoints in the second variable) identify $\Sp(F(\alpha, \ast)) = \lim \Sp(F(\alpha, \beta))$  and $\Sp(F(\alpha', \ast)) =\lim \Sp(F(\alpha', \beta)) $. Furthermore, these remain limits after restricting to coconnective objects.  Combining (c) with lemma \ref{lemma vertical right adj and sp} we see that for every map $\beta \rightarrow \beta'$ in $\Jcal$ the square
\[
\begin{tikzcd}
\Sp(F(\alpha, \beta)) \arrow{r}{} \arrow{d}{} & \Sp(F(\alpha, \beta')) \arrow{d}{} \\
\Sp(F(\alpha', \beta)) \arrow{r}{} & \Sp(F(\alpha', \beta'))
\end{tikzcd}
\] 
is horizontally right adjointable. It  follows  that  the map $\Sp(F(\alpha, \ast)) \rightarrow \Sp(F(\alpha', \ast))$ is the limit of the maps $\Sp(F(\alpha, \beta)) \rightarrow \Sp(F(\alpha', \beta))$. Item (1) is now a consequence of the fact that each of these functors is left t-exact.

Item (2) follows from the fact that the right adjoint to the map $F(\ast, \beta) \rightarrow F(\ast, \beta')$ is the limit of right adjoints to  the maps $F(\alpha, \beta) \rightarrow F(\alpha, \beta')$ (by (c)). It remains to prove (3).  Let $F(\ast, \ast) = \colim_{\Jcal} F(\ast, \beta)$. We wish to show that $F(\ast, \ast) = \lim_{\Ical} F(\alpha, \ast)$. As before, we may identify $\Sp(F(\ast,\ast))$ with the limit of the categories $\Sp(F(\ast, \beta))$. By remark   \ref{remark limit along lex} we may reduce to showing that the map $\Sp(F(\ast,\ast)) \rightarrow \lim \Sp(F(\alpha,\ast))$ is a t-exact equivalence. 

It follows from (c) that for every $\alpha$ in $\Ical$ and every map $\beta \rightarrow \beta'$ in $\Jcal$ the square
\[
\begin{tikzcd}
F(\ast, \beta) \arrow{r}{} \arrow{d}{} & F(\ast, \beta') \arrow{d}{} \\
F(\alpha , \beta) \arrow{r}{} & F(\alpha , \beta')
\end{tikzcd}
\]
is horizontally right adjointable. This implies, by virtue of lemma \ref{lemma vertical right adj and sp} that the square
\[
\begin{tikzcd}
\Sp(F(\ast, \beta)) \arrow{r}{} \arrow{d}{} & \Sp(F(\ast, \beta')) \arrow{d}{} \\
\Sp(F(\alpha , \beta)) \arrow{r}{} & \Sp(F(\alpha , \beta'))
\end{tikzcd}
\]
is horizontally right adjointable, so that we may define a functor $\Ical^\lhd \times (\Jcal^\rhd)^\op \rightarrow \cathat$ by passage to right adjoints of $\Sp(F(-, -))$ in the second variable. We now have
\[
\Sp(F(\ast,\ast)) = \lim_{\Jcal^\op}  \Sp(F(\ast, \beta)) = \lim_{\Jcal^\op} \lim_{\Ical} \Sp(F(\alpha, \beta)) = \lim_{\Ical} \lim_{\Jcal^\op} \Sp(F(\alpha, \beta)) =  \lim_{\Ical} \Sp(F(\alpha, \ast)).
\]
Our claim now follows from the fact that all the above limits remain limits after restricting to coconnective objects.
\end{proof}

A fundamental feature of Grothendieck prestable categories is that they are closed under tensor products in $\Pr^L$: 

\begin{theorem}[\cite{SAG} theorem C.4.2.1]\label{theorem tensor product in Groth} 
 For every pair of Grothendieck prestable categories $\ccal, \dcal$ the tensor product $\ccal \otimes \dcal$ (computed in $\Pr^L$) is Grothendieck prestable.
 \end{theorem}
 
 \begin{remark}
Let $\ccal$ and $\dcal$ be Grothendieck prestable categories. We have $\Sp(\ccal \otimes \dcal) = \Sp(\ccal) \otimes \Sp(\dcal)$, so applying theorem  \ref{theorem tensor product in Groth} we obtain a canonical t-structure on $\Sp(\ccal) \otimes \Sp(\dcal)$. This t-structure is characterized by the property that the category of connective objects is generated under colimits by those objects of the form $X \otimes Y$ with $X, Y$ connective objects of $\Sp(\ccal)$ and $\Sp(\dcal)$, respectively.
\end{remark}

It follows from  theorem \ref{theorem tensor product in Groth} that $\Groth$ inherits a symmetric monoidal structure from the category $\Mod_{\Sp^\cn}(\Pr^L)$ of additive categories. This in fact restricts to a symmetric monoidal structure on  $\Groth^\lex$ and  $\Groth^\cnorm$:

\begin{proposition}[\cite{SAG} propositions C.4.4.1, C.4.4.3]\label{proposition tensor y compact y lex}
 Let $f: \ccal \rightarrow \ccal'$ be a left exact (resp. compact) morphism in $\Groth$.   Then for every Grothendieck prestable category $\dcal$ the functor $f \otimes \id: \ccal \otimes \dcal \rightarrow \ccal' \otimes \dcal$ is left exact (resp. compact).
\end{proposition}

The symmetric monoidal structure on $\Groth$ is compatible with colimits of diagrams with almost compact transitions:

\begin{proposition} \label{proposition tensor products and colimits}
Let $\ccal$ be a Grothendieck prestable category. Then the composite functor
\[
\Groth^\acnorm \hookrightarrow \Groth \xrightarrow{- \otimes \ccal} \Groth
\]
preserves small colimits.
\end{proposition}
\begin{proof}
The analogous assertion with compactness instead of almost compactness is \cite{SAG} proposition C.4.5.1. The same proof works in the case of almost compactness.
\end{proof}

We will throughout the paper work with Grothendieck prestable categories subject to an additional completeness condition, which corresponds under the assignment $\Ccal \mapsto \Sp(\ccal)$ to the left completeness of the t-structure on $\Sp(\ccal)$.

\begin{notation}
Let $\Ccal$ be a Grothendieck prestable category. We denote by $\widehat{\ccal}$ the limit of the tower $\ccal_{\leq 0} \leftarrow \ccal_{\leq 1} \leftarrow \ccal_{\leq 2}\leftarrow \ldots$. We note that there is a canonical functor $\ccal \rightarrow \widehat{\ccal}$.
\end{notation}

\begin{definition}
Let $\Ccal$ be a Grothendieck prestable category. We say that $\ccal$ is complete if the projection  $\ccal \rightarrow \widehat{\ccal}$ is an equivalence.
\end{definition}

\begin{proposition}[\cite{SAG} proposition C.3.6.3]\label{proposition properties comp}
Let $\Ccal$ be a Grothendieck prestable category.
\begin{enumerate}[\normalfont (1)]
\item The category $\widehat{\Ccal}$ is a complete Grothendieck prestable category.
\item The projection $\ccal \rightarrow \widehat{\ccal}$ is colimit preserving and left exact.
\item For every complete Grothendieck prestable category $\Dcal$, restriction along the projection $\ccal \rightarrow \widehat{\ccal}$ provides an equivalence between the category of colimit preserving functors $\widehat{\ccal} \rightarrow \dcal$ and the category of colimit preserving functors $\ccal \rightarrow \dcal$.
\end{enumerate}
\end{proposition}

\begin{notation}
We denote by $\Groth_\comp$ the full subcategory of $\Groth$  on the complete Grothendieck prestable categories. 
\end{notation}

It follows from part (3) of proposition \ref{proposition properties comp} that $\Groth_\comp$ is a localization of $\Groth$, where the localization functor $\Groth \rightarrow \Groth_\comp$ sends each Grothendieck prestable category $\ccal$ to $\widehat{\ccal}$. The following proposition deals with the interaction of this localization functor with left exactness and almost compactness conditions on morphisms.

\begin{proposition}\label{proposition lex y almost compact completion}
Let $f: \ccal \rightarrow \dcal$ be a colimit preserving functor between Grothendieck prestable categories and let $\widehat{f}: \widehat{\ccal} \rightarrow \widehat{\dcal}$  be the induced functor.
\begin{enumerate}[\normalfont (1)]
\item If $f$ is left exact then $\widehat{f}$ is left exact.
\item $f$ is almost compact if and only if $\widehat{f}$ is almost compact.
\end{enumerate}
\end{proposition}
\begin{proof}
This follows directly from the fact that the projections $\ccal \rightarrow \widehat{\ccal}$ and $\dcal \rightarrow \widehat{\dcal}$, and their right adjoints, restrict to equivalences on the full subcategories of $n$-truncated objects for all $n \geq 0$.
\end{proof}

\begin{notation}
We denote by $\Groth_\comp^\lex$ (resp. $\Groth_\comp^\acnorm$) the wide subcategory of $\Groth_\comp$ on the left exact (resp. almost compact) arrows.
\end{notation}

\begin{proposition}\hfill
\begin{enumerate}[\normalfont (1)]
\item The category $ \Groth_\comp^\lex$ admits small limits, and these are preserved by the inclusions $\Groth_\comp^\lex \rightarrow \Groth_\comp \rightarrow \Groth \rightarrow \cathat$.
\item The category $\Groth_\comp^\acnorm$ admits small colimits, and these are preserved by the inclusion into $\Groth_\comp$.
\end{enumerate}
\end{proposition}
\begin{proof}
We first prove (1). By proposition \ref{prop colimits y limits}, it will be enough to show that if $\ccal_\alpha$ is a diagram in $\Groth^\lex$ with limit $\ccal$ such that $\ccal_\alpha$ is complete for all $\alpha$, then $\ccal$ is complete. Consider the commutative square
\[
\begin{tikzcd}
\ccal \arrow{d}{} \arrow{r}{} & \lim_n \ccal_{\leq n} \arrow{d}{} \\
\lim_\alpha \ccal_{\alpha} \arrow{r}{} & \lim_n \lim_\alpha (\ccal_\alpha)_{\leq n}.
\end{tikzcd}
\]
The fact that $\ccal$ is complete follows from the fact that the vertical arrows and the bottom horizontal arrow are equivalences.

We now establish (2). It follows from proposition \ref{proposition lex y almost compact completion} that the commutative square of inclusions
\[
\begin{tikzcd}
\Groth^\acnorm \arrow{r}{} & \Groth \\
\Groth^\acnorm_\comp \arrow{u}{} \arrow{r}{} & \Groth_\comp \arrow{u}{}
\end{tikzcd}
\]
is vertically left adjointable. The desired claim is now a consequence of proposition \ref{prop colimits y limits}.
\end{proof}

We also have the following compatibility of completion with tensor products:

\begin{proposition}[\cite{SAG} proposition C.4.6.1]\label{proposition localization complete sym mon}
The localization $\Groth \rightarrow \Groth_\comp$ is compatible with the symmetric monoidal structure on $\Groth$.
\end{proposition}

\begin{notation}
Let $\ccal, \dcal$ be objects of $\Groth_\comp$. We denote by $\ccal \otimeshat \dcal$ their tensor product in $\Groth_\comp$.
\end{notation}

This completed tensor product is compatible with colimits of diagrams with almost compact transitions:

\begin{proposition}\label{proposition completed tensor y colimits}
Let $\ccal$ be a complete Grothendieck prestable category. Then the composite functor
\[
\Groth^\acnorm_\comp \hookrightarrow \Groth_\comp \xrightarrow{- \otimeshat \ccal} \Groth_\comp
\]
preserves small colimits.
\end{proposition}
\begin{proof}
We have a commutative diagram
\[
\begin{tikzcd}
\Groth^\acnorm  \arrow{r}{}\arrow{d}{} &  \Groth  \arrow{d}{} \arrow{r}{- \otimes \ccal} & \Groth \arrow{d}{} \\
\Groth^\acnorm_\comp \arrow{r}{}  &  \Groth_\comp \arrow{r}{- \otimeshat \ccal} & \Groth_\comp
\end{tikzcd}
\]
where the vertical arrows are given by completion. Since the leftmost vertical arrow is a localization, to prove the proposition it will suffice to show that the induced functor $\Groth^\acnorm \rightarrow \Groth_\comp$ preserves small colimits. This follows directly from proposition \ref{proposition tensor products and colimits}, since the rightmost vertical arrow is a localization.
\end{proof}

%%%%%%%%%%%%%%%%%%%%%%%%%%%%%%%%%%%%%%%%%%%%%%%%%%%%%%%%%%%%%%%%%%%%%%%%
%%%%%%%%%%%%%%%%%%%%%%%%%%%%%%%%%%%%%%%%%%%%%%%%%%%%%%%%%%%%%%%%%%%%%%%%
%%%%%%%%%%%%%%%%%%%%%%%%%%%%%%%%%%%%%%%%%%%%%%%%%%%%%%%%%%%%%%%%%%%%%%%%
%%%%%%%%%%%%%%%%%%%%%%%%%%%%%%%%%%%%%%%%%%%%%%%%%%%%%%%%%%%%%%%%%%%%%%%%
%%%%%%%%%%%%%%%%%%%%%%%%%%%%%%%%%%%%%%%%%%%%%%%%%%%%%%%%%%%%%%%%%%%%%%%%
%%%%%%%%%%%%%%%%%%%%%%%%%%%%%%%%%%%%%%%%%%%%%%%%%%%%%%%%%%%%%%%%%%%%%%%%

\subsection{Almost rigidity}\label{subsection almost rigid} 

As discussed in \ref{subsection groth}, the  category $\Groth$ admits colimits of diagrams with compact transitions, and these are compatible with the symmetric monoidal structure. In particular, if $\Acal$ is a commutative algebra in $\Groth$ and $\Mcal, \Ncal$ are a pair of  $\Acal$-modules, then the Bar construction $\Bar_{\acal}(\Mcal, \Ncal)_\bullet$ admits a geometric realization in $\Groth$ as long as all its face maps are compact. This is in particular guaranteed whenever $\Acal$ has the property that for every $\Sp(\acal)$-module $\Mcal'$ in $\Pr^L$  the action map  $\Sp(\Acal) \otimes \Mcal' \rightarrow \Mcal'$ admits a colimit preserving right adjoint. This is a purely $2$-categorical property that may be formulated more generally when $\Sp(\acal)$ and $\Pr^L$ are replaced by an algebra $A$ in a monoidal $2$-category $\ccal$. Our next goal is to single out a class of algebras which do satisfy this property.

\begin{definition}\label{definition almost rigid}
Let $\ccal$ be a   monoidal $2$-category. We say that an algebra object $A$ of $\ccal$ is almost rigid if the multiplication map $\mu: A \otimes A \rightarrow A$ admits a right adjoint and the commutative square
\[
\begin{tikzcd}
A \otimes A \otimes A \arrow{d}{\mu \otimes \id} \arrow{r}{\id \otimes \mu} & A \otimes A \arrow{d}{\mu} \\
A \otimes A \arrow{r}{\mu} & A
\end{tikzcd}
\]
is both horizontally and vertically right adjointable. We say that $A$ is rigid if it is almost rigid and the unit map $1_{\ccal} \rightarrow \Acal$ admits a right adjoint.
\end{definition}

\begin{remark}
When specialized to the symmetric monoidal $2$-category of presentable stable categories and colimit preserving functors, definition \ref{definition almost rigid} recovers the notion of rigid presentable stable monoidal category from \cite{GR} section 1.9. 
\end{remark}

\begin{remark}\label{remark opposite almost rigid}
Let $\ccal$ be a symmetric monoidal $2$-category and let $A$ be an algebra in $\ccal$. Then $A$ is almost rigid if and only if its opposite algebra is almost rigid.
\end{remark}

\begin{remark}\label{remark tensor almost rigid}
Let $\ccal$ be a symmetric monoidal $2$-category and let $A, B$ be algebra objects of $\ccal$. Then the algebra $A \otimes B$ is the image of the pair $(A, B)$ along the symmetric monoidal functor of symmetric monoidal $2$-categories $\otimes: \ccal \times \ccal \rightarrow \ccal$. It follows that if $A$ and $B$ are almost rigid (resp. rigid) then $A \otimes B$ is almost rigid (resp. rigid).
\end{remark}

The following is our main result concerning the notion of almost rigidity:

\begin{theorem}\label{theo modules over almost rigid algebras}
Let $\ccal$ be a   monoidal $2$-category and let $\dcal$ be a $2$-category with a left action from $\ccal$. Let $A$ be an almost rigid algebra object of $\ccal$ and let $M$ be a left $A$-module in $\dcal$. Then: 
\begin{enumerate}[\normalfont (1)]
\item The action map $\mu_M : A \otimes M \rightarrow M$ admits a right adjoint.
\item Let $f: M \rightarrow N$ be a morphism of left $A$-modules in $\dcal$. Assume that the morphism of $\dcal$ underlying $f$ admits a right adjoint. Then the commutative square
\[
\begin{tikzcd}
A \otimes M \arrow{r}{\mu_M}\arrow{d}{\id \otimes f} & M \arrow{d}{f} \\
A \otimes N \arrow{r}{\mu_N} & N 
\end{tikzcd}
\]
is vertically right adjointable.
\end{enumerate}
\end{theorem}

Before giving the proof of theorem \ref{theo modules over almost rigid algebras} we note the following consequence:

\begin{corollary}\label{coro apply in arrow cat}
Let $\ccal$ be a   monoidal $2$-category and let $\dcal$ be a $2$-category with a left action from $\ccal$. Let $A$ be an almost rigid algebra object of $\ccal$ and let $f: M \rightarrow N$ be a morphism of left $A$-modules in $\dcal$. Then the commutative square
\[
\begin{tikzcd}
A \otimes M \arrow{r}{\mu_M}\arrow{d}{\id \otimes f} & M \arrow{d}{f} \\
A \otimes N \arrow{r}{\mu_N} & N 
\end{tikzcd}
\]
is horizontally right adjointable.
\end{corollary}
\begin{proof}
Consider the $2$-category $\Fun([1], \dcal)$ of arrows in $\dcal$. This admits a canonical left action from $\Fun([1], \ccal)$, which after restriction of scalars along the diagonal map $\ccal = \Fun([0], \ccal) \rightarrow \Fun([1], \ccal)$ gives rise to a left action from $\ccal$. We may identify the commutative square in the statement with the action map $\mu_f: A \otimes f \rightarrow f$. To prove the corollary it will suffice to show that $\mu_f$ admits a right adjoint. This follows from theorem \ref{theo modules over almost rigid algebras}.
\end{proof}

The proof of theorem \ref{theo modules over almost rigid algebras} requires some preliminary lemmas.

\begin{lemma}\label{lemma how to check vert adjointable}
Let $\ccal$ be a $2$-category and let $f: X \rightarrow Y$ be a morphism in $\ccal$.  Suppose that $f $ is the colimit (in the $2$-category $\operatorname{Arr}(\ccal)$ of arrows in $\ccal$) of a diagram of maps $f_\alpha : X_\alpha \rightarrow Y_\alpha$ and that this colimit is preserved by the source and target functors $\operatorname{Arr}(\ccal) \rightarrow \ccal$.  Assume that for every transition $\alpha \rightarrow \beta$ the square
\[
\begin{tikzcd}
X_\alpha \arrow{r}{} \arrow{d}{f_\alpha} & X_\beta \arrow{d}{f_\beta} \\
Y_\alpha \arrow{r}{} & Y_\beta
\end{tikzcd}
\]
is vertically right adjointable. Then:
\begin{enumerate}[\normalfont (1)]
\item $f$ is right adjointable, and for every $\alpha$ the   square
\[
\begin{tikzcd}
X_\alpha \arrow{r}{} \arrow{d}{f_\alpha} & X \arrow{d}{f}\\
Y_\alpha \arrow{r}{} & Y
\end{tikzcd}
\]
is vertically right adjointable.
\item Let
\[
\sigma: 
\begin{tikzcd}
X  \arrow{r}{ } \arrow{d}{f } & X' \arrow{d}{f'} \\
Y  \arrow{r}{ } & Y'
\end{tikzcd} \hspace{0.5cm}
\]
be a commutative square in $\ccal$  and assume that  for every $\alpha$ the square
\[
\begin{tikzcd}
X_\alpha \arrow{r}{} \arrow{d}{f_\alpha} & X' \arrow{d}{f'}\\
Y_\alpha \arrow{r}{} & Y'
\end{tikzcd}
\]
is vertically right adjointable. Then $\sigma$ is vertically right adjointable.
\end{enumerate}
\end{lemma}
\begin{proof}
This is a consequence of \cite{Pres} proposition 4.3.17 (applied to the epimorphism of $2$-categories $[1] \rightarrow \operatorname{Adj}$ given by the inclusion of the universal left adjoint inside the walking adjunction).
\end{proof}

\begin{lemma}\label{lemma colimit adjointable squares}
Let $\ccal$ be a $2$-category and let
\[
\sigma: 
\begin{tikzcd}
X  \arrow{r}{g } \arrow{d}{f } & X' \arrow{d}{f'} \\
Y  \arrow{r}{g'} & Y'
\end{tikzcd} \hspace{0.5cm}
\]
be a commutative square in $\ccal$. Suppose that $\sigma$ is the colimit (in the category $\Fun([1] \times [1], \ccal)$ of commutative squares in $\ccal$) of a diagram of squares
\[
\sigma_\alpha: 
\begin{tikzcd}
X_\alpha  \arrow{r}{g_\alpha} \arrow{d}{f_\alpha} & X' \arrow{d}{f'_\alpha} \\
Y_\alpha  \arrow{r}{g'_\alpha} & Y'_\alpha
\end{tikzcd} \hspace{0.5cm}
\]
and that this colimit is preserved by the evaluation functors $\Fun([1] \times [1], \ccal) \rightarrow \ccal$. Assume the following:
\begin{enumerate}[\normalfont (a)]
\item The square $\sigma_\alpha$ is vertically right adjointable for all $\alpha$.
\item For all transitions $\alpha \rightarrow \beta$ the square
\[
\begin{tikzcd}
X_\alpha  \arrow{r}{} \arrow{d}{f_\alpha} & X_\beta \arrow{d}{f_\beta} \\
Y_\alpha  \arrow{r}{} & Y_\beta
\end{tikzcd}
\]
is vertically right adjointable.
\item For all transitions $\alpha \rightarrow \beta$ the square
\[
\begin{tikzcd}
X'_\alpha  \arrow{r}{} \arrow{d}{f'_\alpha} & X'_\beta \arrow{d}{f'_\beta} \\
Y'_\alpha  \arrow{r}{} & Y'_\beta
\end{tikzcd}
\]
is vertically right adjointable.
\end{enumerate}
Then $\sigma$ is vertically right adjointable.
\end{lemma}
\begin{proof}
We regard $\sigma$ as a morphism between $g$ and $g'$ in the arrow $2$-category $\operatorname{Arr}(\ccal)$. From this point of view, our goal is to show that $\sigma$ admits a right adjoint. Similarly regard $\sigma_\alpha$ as a morphism from $g_\alpha$ to $g'_\alpha$ for all $\alpha$. Note that $\sigma$ is the colimit in $\operatorname{Arr}(\ccal)$ of the diagram $\sigma_\alpha$, and that this colimit is preserved by the source and target functors. By lemma \ref{lemma how to check vert adjointable} we may reduce to showing that for every transition $\alpha \rightarrow \beta$ the commutative square
\[
\begin{tikzcd}
g_\alpha \arrow{d}{\sigma_\alpha} \arrow{r}{} & g_\beta \arrow{d}{\sigma_\beta}\\
g'_\alpha \arrow{r}{} & g'_\beta
\end{tikzcd}
\]
is vertically right adjointable. The vertical arrows admit right adjoints by (a). Applying \cite{HSTI} lemma 3.4.12 we may reduce to showing that the image of the above square under the source and target functors $\operatorname{Arr}(\ccal) \rightarrow \ccal$ is vertically right adjointable. This is the content of (b) and (c).
\end{proof}

\begin{lemma}\label{lemma horizontal adjointability and split geom realiz}
Let $\ccal$ be a $2$-category and let
\[
\sigma: 
\begin{tikzcd}
X  \arrow{r}{g } \arrow{d}{f } & X' \arrow{d}{f'} \\
Y  \arrow{r}{g'} & Y'
\end{tikzcd} \hspace{0.5cm}
\]
be a commutative square in $\ccal$. Suppose that $f$ is the geometric realization  (in the $2$-category $\operatorname{Arr}(\ccal)$ of arrows in $\ccal$) of a simplicial object $f_\bullet: X_\bullet \rightarrow Y_\bullet$ whose image under the evaluation functors $\Arr(\ccal) \rightarrow \ccal$ is split.  Assume the following:
\begin{enumerate}[\normalfont (a)]
\item The maps $g$ and $g'$ admit right adjoints.
\item For every $n \geq 0$ the square
\[
\begin{tikzcd}
X_n  \arrow{r}{} \arrow{d}{f_n} & X  \arrow{d}{f } \\
Y_n  \arrow{r}{} & Y 
\end{tikzcd}
\]
is horizontally right adjointable.
\item For every $n \geq 0$ the square
\[
\begin{tikzcd}
X_n  \arrow{r}{} \arrow{d}{f_n} & X' \arrow{d}{f'} \\
Y_n  \arrow{r}{} & Y'
\end{tikzcd}
\]
is horizontally right adjointable.
\end{enumerate}
Then $\sigma$ is horizontally right adjointable.
\end{lemma}
\begin{proof}
Denote by $g^R, g'^R$ the right adjoints to $g$ and $g'$. For each $n \geq 0$ let $h_n: Y_n \rightarrow Y$ be the canonical map, and $h_n^R$ be its right adjoint. Let $\eta: fg^R \rightarrow g'^Rf'$ be the canonical natural transformation. Our goal is to show that $\eta$ is an isomorphism.

For every $n \geq 0$ the square $\sigma$ fits in a commutative diagram
\[
\begin{tikzcd}
X_n \arrow{r}{} \arrow{d}{f_n} & X  \arrow{r}{g} \arrow{d}{f } & X' \arrow{d}{f'} \\
Y_n \arrow{r}{h_n} & Y \arrow{r}{g'} & Y'
\end{tikzcd}
\]
Our assumptions guarantee that the outer commutative square and the left commutative square are horizontally right adjointable. It follows that $\eta$ becomes an isomorphism when composed with $h^R_n$, and in particular it becomes an isomorphism when composed with $h_nh_n^R$. The fact that the objects $Y$ and $Y_n$ fit in a split augmented simplicial object in $\ccal$ implies that the endofunctors $h_n h_n^R: Y \rightarrow Y$ sit in a split augmented simplicial object with augmentation $\id_Y$. It follows from this that $\eta$ is an isomorphism when composed with $\id_Y$, and hence $\eta$ is an isomorphism, as desired.
\end{proof}

\begin{proof}[Proof of theorem \ref{theo modules over almost rigid algebras}]
We begin with a proof of (1). Denote by $\Bar_A(A, M)_\bullet$ the Bar resolution of $M$, so that $\Bar_A(A, M)_n = A^{\otimes n+1} \otimes M$ for every $n \geq 0$. Let $\Bar^\ss_A(A, M)$ be the  semisimplicial object underlying $\Bar_A(A,M)_\bullet$. We have that $\Bar_A(A, M)_\bullet$ is a split  simplicial resolution of $M$, and in particular we see that $A \otimes M$  is the geometric realization of $A \otimes \Bar_A(A,M)_\bullet$.  It follows that $\mu_M$ is the  colimit of the map of semisimplicial objects
\[
A \otimes \Bar^\ss_A(A, M)_\bullet \rightarrow \Bar^\ss_A(A,M)_\bullet
\] 
and that this colimit is preserved by the evaluation functors. By lemma \ref{lemma how to check vert adjointable}  we may reduce to showing  that for each face map $\varphi: [n] \rightarrow [n+1]$ the induced commutative square
\[
\begin{tikzcd}
A \otimes \Bar_A(A, M)_{n+1} \arrow{d}{} \arrow{r}{} & A \otimes  \Bar_A(A, M)_{n } \arrow{d}{} \\\Bar_A(A, M)_{n+1} \arrow{r}{} & \Bar_A(A, M)_{n}
\end{tikzcd}
\]
is vertically right adjointable. Denote the above square by $\sigma$. We consider two cases:
\begin{itemize}
\item If $\varphi$ is different from the $0$-th face then $\sigma$ is obtained by tensoring the square
\[
\begin{tikzcd}
A \otimes A \arrow{r}{\id} \arrow{d}{\mu} & A \otimes A \arrow{d}{\mu} \\ 
A \arrow{r}{\id} & A
\end{tikzcd}
\]
with a square of the form
\[
\begin{tikzcd}
A^{\otimes n+1} \otimes M \arrow{d}{\id} \arrow{r}{} &  A^{\otimes n} \otimes M \arrow{d}{\id} \\
A^{\otimes n+1} \otimes M \arrow{r}{} &  A^{\otimes n} \otimes M.
\end{tikzcd}
\]
The fact that $\sigma$ is vertically right adjointable follows from the fact that the above two squares are vertically right adjointable (since $\mu: A \otimes A \rightarrow A$ admits a right adjoint).
\item If $\varphi$ is the $0$-th face then $\sigma$ is obtained by tensoring the square in definition \ref{definition almost rigid} with $A^{\otimes n} \otimes M$. The fact that $\sigma$ is vertically right adjointable then follows from the fact that $A$ is almost rigid.
\end{itemize}

We now prove (2).  We begin by addressing the following particular case:
\begin{enumerate}[$(\ast)$]
\item The commutative square
\[
\begin{tikzcd}[column sep = large]
A \otimes A \otimes M \arrow{d}{\id \otimes \mu_M} \arrow{r}{\mu \otimes \id} & A \otimes M \arrow{d}{\mu_{M}} \\
A \otimes M \arrow{r}{\mu_M} & M
\end{tikzcd}
\]
is vertically right adjointable.
\end{enumerate}

As in (1), we may write the commutative square in the statement as the colimit of the square
\[
\begin{tikzcd}[column sep = 3cm]
A \otimes A \otimes \Bar^\ss_A(A,M)_\bullet  \arrow{r}{ \mu \otimes \id } \arrow{d}{\id \otimes \mu_{\Bar^\ss_A(A,M)_\bullet} } & A \otimes \Bar^\ss_A(A,M)_\bullet \arrow{d}{\mu_{ \Bar^\ss_A(A,M)_\bullet}} \\
A \otimes  \Bar^\ss_A(A,M)_\bullet \arrow{r}{\mu_{ \Bar^\ss_A(A,M)_\bullet}} &  \Bar^\ss_A(A,M)_\bullet.
\end{tikzcd}
\]
To prove $(\ast)$ we will verify conditions (a), (b) and (c) of lemma \ref{lemma colimit adjointable squares}. Condition (c) was already established during our proof of (1). The squares arising from condition (b) are obtained by tensoring those squares arising from condition (c) with $A$, so we see that (b) also holds. It remains to show (a). This amounts to establishing $(\ast)$ in the case when $M$ is replaced by $\Bar_A(A,M)_n$ for some $n \geq 0$. In other words, we may reduce to the case $M = A \otimes M'$ is a free left $A$-module. In this case the commutative square from $(\ast)$    is obtained by tensoring the transpose of the square from definition \ref{definition almost rigid} with $M'$, and is therefore vertically right adjointable since $A$ is almost rigid.

We now address the general case of (2).  By lemma \ref{lemma horizontal adjointability and split geom realiz} it suffices to prove the following:
\begin{enumerate}[\normalfont (i)]
\item For every $n \geq 0$ the commutative square
\[
\begin{tikzcd}
A \otimes \Bar_A(A, M)_n \arrow{d}{} \arrow{r}{} &  \Bar_A(A, M)_n \arrow{d}{} \\
A \otimes M \arrow{r}{} & M
\end{tikzcd}
\]
is vertically right adjointable.
\item For every $n \geq 0$ the commutative square
\[
\begin{tikzcd}
A \otimes \Bar_A(A, M)_n \arrow{d}{} \arrow{r}{}  & \Bar_A(A, M)_n \arrow{d}{} \\
A \otimes N \arrow{r}{} & N
\end{tikzcd}
\]
is vertically right adjointable.
\end{enumerate}
We first prove (i). We argue by induction on $n$. The case $n = 0$ is the content of $(\ast)$. Assume now that $n > 0$ and that the assertion is known for $n - 1$. Then our square sits as the outer commutative square in the commutative diagram
\[
\begin{tikzcd}
A \otimes \Bar_A(A, M)_n \arrow{d}{} \arrow{r}{} &  \Bar_A(A, M)_n \arrow{d}{} \\
A \otimes \Bar_A(A, M)_{n-1} \arrow{d}{} \arrow{r}{} &  \Bar_A(A, M)_{n-1} \arrow{d}{} \\
A \otimes M \arrow{r}{} & M
\end{tikzcd}
\]
where here the top square is induced from the $0$-th face map. The inductive hypothesis guarantees that the bottom square is vertically right adjointable, so we may reduce to showing that the top square is vertically right adjointable. This is a consequence of $(\ast)$ in the case when $M$ is replaced with $ \Bar_A(A, M)_{n-1}$.

We now prove (ii). Our square sits as the outer commutative square in the following commutative diagram:
\[
\begin{tikzcd}
A \otimes \Bar_A(A, M)_n \arrow{d}{} \arrow{r}{}  & \Bar_A(A, M)_n \arrow{d}{} \\
A \otimes \Bar_A(A, N)_n \arrow{d}{} \arrow{r}{}  & \Bar_A(A, N)_n \arrow{d}{} \\
A \otimes N \arrow{r}{} & N.
\end{tikzcd}
\]
It follows from (i) (in the case when $M$ is replaced by $N$) that the bottom   square is vertically right adjointable. We may thus reduce to showing that the top square is vertically right adjointable. Replacing $M$ and $N$ with $\Bar_A(A, M)_n$ and $\Bar_A(A,N)_n$ we may now reduce to proving (2) in the case when $M = A \otimes M'$ and $N= A \otimes N'$ are free modules, and $f$ is of the form $\id \otimes f'$ for some map $f' : M' \rightarrow N'$. In this case the square in the statement is the tensor product of the square
\[
\begin{tikzcd}
A \otimes A    \arrow{r}{\mu } \arrow{d}{\id  } & A \arrow{d}{\id } \\
A \otimes A    \arrow{r}{\mu} & A
\end{tikzcd}
\]
and the square
\[
\begin{tikzcd}
M' \arrow{r}{\id} \arrow{d}{f'} & M' \arrow{d}{f'} \\
N'  \arrow{r}{\id} & N'.
\end{tikzcd}
\]
The result now follows from the fact that both of the above squares are vertically right adjointable.
\end{proof}

%%%%%%%%%%%%%%%%%%%%%%%%%%%%%%%%%%%%%%%%%%%%%%%%%%%%%%%%%%%%%%%%%%%%%%%%
%%%%%%%%%%%%%%%%%%%%%%%%%%%%%%%%%%%%%%%%%%%%%%%%%%%%%%%%%%%%%%%%%%%%%%%%
%%%%%%%%%%%%%%%%%%%%%%%%%%%%%%%%%%%%%%%%%%%%%%%%%%%%%%%%%%%%%%%%%%%%%%%%
%%%%%%%%%%%%%%%%%%%%%%%%%%%%%%%%%%%%%%%%%%%%%%%%%%%%%%%%%%%%%%%%%%%%%%%%
%%%%%%%%%%%%%%%%%%%%%%%%%%%%%%%%%%%%%%%%%%%%%%%%%%%%%%%%%%%%%%%%%%%%%%%%
%%%%%%%%%%%%%%%%%%%%%%%%%%%%%%%%%%%%%%%%%%%%%%%%%%%%%%%%%%%%%%%%%%%%%%%%

\subsection{The \texorpdfstring{$2$}{2}-category \texorpdfstring{$\Groth^{\st}_\comp$}{Grothstcomp}}\label{subsection grothstcomp}

The symmetric monoidal categories  $\Groth$ and $\Groth_\comp$ admit natural $2$-categorical enhancements, so it makes sense to specialize the notion of almost rigidity from \ref{subsection almost rigid} to the case $\ccal = \Groth$ or $\ccal = \Groth_\comp$. For the purposes of this paper we will in fact need to work with algebras which are almost rigid in a slight variant of $\Groth_\comp$, which we introduce in this section. We begin by recalling the following notion from \cite{SAG} remark C.3.1.3:
\begin{definition}
Let $\ccal$ be a presentable stable category. A core for $\ccal$ is a full subcategory of $\ccal$ which is closed under small colimits and extensions. 
\end{definition}

In \cite{SAG} remark C.3.1.3, Lurie defines a category $\Groth^+$ with the following properties:
\begin{itemize}
\item Objects of $\Groth^+$ are pairs $(\ccal, \ccal')$ of a presentable stable category $\ccal$ and a core $\ccal'$ for $\ccal$.
\item A morphism $(\ccal, \ccal') \rightarrow (\dcal, \dcal')$ in $\Groth^+$ is a colimit preserving functor $f: \ccal \rightarrow \dcal$ such that $f(\ccal') \subseteq \dcal'$.
\end{itemize}
As discussed in \cite{SAG} remark C.4.2.3, $\Groth^+$ admits a symmetric monoidal structure, where the tensor product of a pair of objects $(\ccal, \ccal')$ and $(\dcal, \dcal')$ is given by the presentable stable category $\ccal \otimes \dcal$ equipped with the smallest core containing the object $X \otimes Y$ for all $X$ in $\ccal'$ and $Y$ in $\dcal'$. Furthermore, the assignment $\ccal \mapsto (\Sp(\ccal), \ccal)$ provides a fully faithful symmetric monoidal embedding $\Groth \rightarrow \Groth^+$.

Our next goal is to study a variant of $\Groth^+$ that accommodates for functors that are compatible with cores only up to a shift.

\begin{notation}
For each presentable stable category $\ccal$ we denote by $\Core(\ccal)$ the set of cores of $\ccal$. We equip $\Core(\ccal)$ with the preorder where $\ccal' \leq \ccal''$ if there exists an integer $n$ such that $\Sigma^n \ccal' \subseteq \ccal'$. We denote by $\Core(\ccal)/{\simeq}$ the underlying poset. In other words, elements of $\Core(\ccal)/{\simeq}$ consist of equivalence classes of cores, where $\ccal'$ and $\ccal''$ are equivalent if $\ccal' \leq \ccal'' \leq \ccal'$. For each core $\ccal'$ for $\ccal$, we denote by $[\ccal']$ its corresponding equivalence class.
\end{notation}

\begin{construction}\label{construction Grothoverline}
Let $\rho: \Pr^L_{\St} \rightarrow \Cat$ be the functor defined as follows:
\begin{itemize}
\item For each presentable stable category $\ccal$ we let $\rho(\ccal) = \Core(\ccal)/{\simeq}$, regarded as a category in such a way that there is a unique arrow $[\ccal'] \rightarrow [\ccal'']$ if and only if $\ccal' \leq \ccal''$.
\item For each morphism $f: \ccal \rightarrow \dcal$ in $\Pr^L_{\St}$ we let $\rho(f): \Core(\ccal)/{\simeq}\rightarrow \Core(\dcal)/{\simeq}$ be the functor that sends the equivalence class of a core $\ccal'$ on $\ccal$ to the equivalence class of the smallest core on $\dcal$ containing $f(\ccal')$.
\end{itemize}
We make $\rho$ into a morphism of operads, where the action on operations is given as follows:
\begin{itemize}
\item For each finite family of presentable stable categories $\ccal_1, \ldots, \ccal_n, \dcal$ and every functor $f: \ccal_1 \times \ldots \times \ccal_n \rightarrow \dcal$ which is colimit preserving in each variable, we let 
\[
\rho(f): \Core(\ccal_1)/{\simeq} \times \ldots \times \Core(\ccal_n)/{\simeq} \rightarrow \Core(\dcal)/{\simeq}
\]
be the map that sends a family of equivalence classes of cores $\ccal'_1, \ldots, \ccal'_n$ to the equivalence class of the smallest core on $\dcal$ containing $f(\ccal'_1 \times \ldots \times \ccal'_n)$.
\end{itemize}
We let $  \overline{\Groth} \rightarrow \Pr^L_{\St}$ be the cocartesian fibration of operads classified by $\rho$.
\end{construction}

\begin{remark}
 The symmetric monoidal category $\overline{\Groth}$ from construction \ref{construction Grothoverline} admits the following informal description:
\begin{itemize}
\item Objects of $\overline{\Groth}$ are pairs $(\ccal, [\ccal'])$ of a presentable stable category $\ccal$ and an equivalence class of cores for $\ccal$. 
\item A morphism $(\ccal, [\ccal']) \rightarrow (\dcal, [\dcal'])$ in $\overline{\Groth}$ consists of a colimit preserving functor $f: \ccal \rightarrow \dcal$ such that there exists an integer $n$ with the property that $\Sigma^n f(\ccal') \subseteq \dcal'$.
\item The tensor product of a pair of objects $(\ccal, [\ccal']),(\dcal, [\dcal'])$ of $\overline{\Groth}$ is given by the presentable stable category $\ccal \otimes \dcal$ equipped with the equivalence class of the smallest core containing the objects $X \otimes Y$ for $X$ in $\ccal'$ and $Y$ in $\dcal'$.
\end{itemize}
\end{remark}

\begin{remark}
Let $\rho$ be as in construction \ref{construction Grothoverline}. For each presentable stable category $\ccal$ the category $\rho(\ccal)$ admits finite  colimits. Furthermore, if $f: \ccal_1 \times \ldots \times \ccal_n \rightarrow \dcal$ is an operation in $\Pr^L_\St$, the operation $\rho(f)$ preserves finite colimits in each variable. It follows from this that $\overline{\Groth}$ admits finite colimits, which are preserved by the forgetful functor to $\Pr^L_{\St}$. Furthermore, the symmetric monoidal structure on $\overline{\Groth}$ is compatible with finite colimits. 
\end{remark}

\begin{remark}
Let $\rho$ be as in construction \ref{construction Grothoverline}. Then $\rho$ receives a lax symmetric monoidal natural transformation from the functor $\Pr^L_{\St} \rightarrow \Cat$ that sends each presentable stable category $\ccal$ to the poset of cores on $\ccal$ ordered by inclusion. This transformation classifies a strictly symmetric monoidal functor $\Groth^+ \rightarrow \overline{\Groth}$, which at the level of objects sends a pair of a presentable stable category $\ccal$ and a core $\ccal'$ for $\ccal$ to $(\ccal, [\ccal'])$. In particular, we obtain a symmetric monoidal functor $\Groth \rightarrow \overline{\Groth}$ that sends each Grothendieck prestable category $\ccal$ to $(\Sp(\ccal), [\ccal])$.
\end{remark}

\begin{notation}
We denote by $\Groth^\st$  the full subcategory of $\overline{\Groth}$ on those objects of the form $(\Sp(\ccal), [\ccal])$ with $\ccal$ in $\Groth$. We equip $\Groth^\st$ with the induced symmetric monoidal structure.
\end{notation}

Our next goal is to study a full subcategory of $\Groth^\st$ on objects that we will call complete.

\begin{proposition}\label{proposition well defined complete}
Let $\ccal, \dcal$ be Grothendieck prestable categories such that $(\Sp(\ccal), [\ccal])$ is equivalent to $(\Sp(\dcal), [\dcal])$ (as objects of $\Groth^\st$). Then $\ccal$ is complete if and only if $\dcal$ is complete.
\end{proposition} 
\begin{proof}
Assume that $\ccal$ is complete; we will show that $\dcal$ is complete. Fix an equivalence $(\Sp(\ccal), [\ccal]) = (\Sp(\dcal), [\dcal])$ in $\Groth^\st$. In other words, this is an equivalence $f: \Sp(\ccal) \rightarrow \Sp(\dcal)$ which is both left and right t-exact up to shifts. Denote by $g$ the inverse to $f$. To prove that $\dcal$ is complete it will suffice to show the following:
\begin{enumerate}[\normalfont (i)]
\item Assume that $\ldots X_2 \rightarrow X_1 \rightarrow X_0$ is an inverse system of objects of $\dcal$ such that $\tau_{\leq n}(X_{n+1}) = X_n$ for all $n \geq 0$, and let $X$ be its limit in $\Sp(\Dcal)$. Then $X$ is connective and $\tau_{\leq n}(X) = X_n$ for all $n \geq 0$.
\item If $X$ is an object of $\dcal$ such that $\tau_{\leq n}(X) = 0$ for all $n\geq 0$ then $X = 0$.
\end{enumerate}

We first address (i). Since $g$ is right t-exact up to a shift we have that the inverse system $\tau_{\leq m}(g(X_n))$ is eventually constant for all integers $m$. Using the fact that $\ccal$ is complete  we deduce that the limit $g(X)$ of the system $g(X_n)$ has the property that if $m$ is an integer then $\tau_{\leq m}(g(X)) = \tau_{\leq m}(g(X_n))$ for $n$ sufficiently large. Since $f$ is right t-exact up to a shift we have that if $m'$ is an integer then $\tau_{\leq m'}(f(g(X))) = \tau_{\leq m'}(f(g(X_n)))$ for $n$ sufficiently large. Consequently $\tau_{\leq m'}(X) = \tau_{\leq m'}(X_n)$ for $n$ sufficiently large. Assume now that $m' \geq 0$. Then this implies $\tau_{\leq m'}(X) = X_{m'}$, as desired.

We now prove (ii). We will do so by showing that $g(X) = 0$. Since $\ccal$ is complete, it will be enough to show that $\tau_{\leq m}(g(X)) = 0$ for all integers $m$. Indeed,  using the fact that $g$ is right t-exact up to a shift we have $\tau_{\leq m}((g(X)) = \tau_{\leq m}(g(\tau_{\leq n}(X))) = 0$ for $n$ sufficiently large.
\end{proof}
 
\begin{notation}
We denote by $\Groth^\st_{\comp}$ the full subcategory of $\Groth^\st$ on the objects of the form $(\Sp(\ccal), [\ccal])$ with $\ccal$ a complete Grothendieck prestable category.
\end{notation}
 
\begin{proposition}\label{proposition univ prop completion stable}
Let $\ccal$ be a Grothendieck prestable category, and let $\Dcal$ be a complete Grothendieck prestable category. Let $p: \Ccal \rightarrow \widehat{\Ccal}$ be the canonical map. Then precomposition with $\Sp(p)$ induces an equivalence
\[
\Hom_{\Groth^\st}((\Sp(\widehat{\ccal}),[\widehat{\ccal}]), (\dcal, [\dcal])) = \Hom_{\Groth^\st}((\Sp( \ccal),[\ccal]), (\dcal, [\dcal])) .
\]  
\end{proposition}
\begin{proof}
It   suffices to show that for every integer $n$, precomposition with $\Sp(p)$ induces an equivalence between the space of colimit preserving functors $\Sp(\widehat{\Ccal}) \rightarrow \Sp(\dcal)$ which are right t-exact up to a shift by $n$, and the space of colimit preserving functors $\Sp(\ccal) \rightarrow \Sp(\dcal)$ which are right t-exact up to a shift by $n$. Shifting the t-structure on $\Sp(\dcal)$, we may reduce to the case $n = 0$, which follows directly from the universal property of $\widehat{\ccal}$.
\end{proof}

\begin{corollary}
The inclusion $\Groth^\st_{\comp} \rightarrow \Groth^\st$ admits a left adjoint which is compatible with the symmetric monoidal structure on $\Groth^\st$.
\end{corollary}
\begin{proof}
The existence of a left adjoint $L^\st : \Groth^\st \rightarrow \Groth^\st_\comp$ to the inclusion follows directly from proposition \ref{proposition univ prop completion stable}. It remains to show that  the class of $L^\st$-local maps in $\Groth^\st$ is stable under tensor products. Let $f: (\Sp(\ccal), [\ccal]) \rightarrow (\Sp(\dcal), [\dcal])$ be an $L^\st$-local map. This fits into a commutative square
\[
\begin{tikzcd}
(\Sp(\ccal), [\ccal]) \arrow{r}{} \arrow{d}{\Sp(p_\ccal)} & (\Sp(\dcal), [\dcal])  \arrow{d}{\Sp(p_\dcal)} \\
(\Sp(\widehat{\ccal}),[\widehat{\ccal}]) \arrow{r}{} & (\Sp(\widehat{\dcal}),[\widehat{\dcal}])
\end{tikzcd}
\]
where the bottom arrow is an equivalence, and the vertical arrows are the unit maps for $L^\st$ described by proposition \ref{proposition univ prop completion stable}. We may thus reduce to analyzing the case $f = \Sp(p_\ccal)$. Let $\ecal$ be a Grothendieck prestable category. We wish to show that the map
\[
\Sp(p_\ccal) \otimes \id: (\Sp(\ccal), [\ccal]) \otimes (\Sp(\ecal), [\ecal]) \rightarrow (\Sp(\widehat{\ccal}), \widehat{\ccal}) \otimes (\Sp(\ecal), [\ecal])
\]
is $L^\st$-local. We may identify the above with the map
\[
\Sp(p_\ccal \otimes \id): (\Sp(\ccal \otimes \ecal), [\ccal \otimes \ecal]) \rightarrow (\Sp(\widehat{\ccal} \otimes \ecal), \widehat{\ccal} \otimes \ecal)
\]
which fits into a commutative square
\[
\begin{tikzcd}[column sep = large]
(\Sp(\ccal \otimes \ecal), [\ccal \otimes \ecal]) \arrow{r}{\Sp(p_\ccal \otimes \id)} \arrow{d}{} & (\Sp(\widehat{\ccal} \otimes \ecal), \widehat{\ccal} \otimes \ecal) \arrow{d}{} \\
(\Sp(\widehat{\ccal \otimes \ecal}), [\widehat{\ccal \otimes \ecal}])  \arrow{r}{\Sp(\widehat{p_\ccal \otimes \id})} & (\Sp(\widehat{\widehat{\ccal} \otimes \ecal}), \widehat{\widehat{\ccal} \otimes \ecal}) .
\end{tikzcd}
\]
Here the bottom horizontal arrow is an equivalence by proposition \ref{proposition localization complete sym mon}, and the vertical arrows are the units maps for $L^\st$. It follows that the top horizontal arrow is $L^\st$-local, as desired.
\end{proof}

\begin{corollary}
The projection $\Groth \rightarrow \Groth^\st$ induces a symmetric monoidal functor $\Groth_\comp \rightarrow \Groth^\st_\comp$.
\end{corollary}
\begin{proof}
It suffices to show that this projection maps morphisms which are local for the localization $L: \Groth \rightarrow \Groth_\comp$ to morphisms which are local for the localization $L^\st: \Groth^\st \rightarrow \Groth^\st_{\comp}$. Let $f: \ccal \rightarrow \dcal$ be an $L$-local map. Then $f$ fits into a commutative square
\[
\begin{tikzcd}
\ccal \arrow{r}{f} \arrow{d}{} & \dcal \arrow{d}{} \\
\widehat{\ccal} \arrow{r}{\widehat{f}} & \widehat{\dcal}.
\end{tikzcd}
\]
Here the bottom horizontal arrow is an isomorphism, so it suffice to show that the projection $\Groth \rightarrow \Groth^\st$ maps the vertical arrows to $L^\st$-local maps. This follows from proposition \ref{proposition univ prop completion stable}.
\end{proof}

We now discuss the natural $2$-categorical structure on $\Groth^\st_\comp$.

\begin{construction}
Let $\Cat^\omega$ be the full subcategory of $\Cat$ on the compact objects. For each category $\Ical$ the presentable additive category $\Ical \otimes \Sp^\cn = \Fun(\Ical^\op, \Sp^\cn)$ is Grothendieck prestable, so we have a symmetric monoidal functor 
\[
\Cat^\omega \hookrightarrow \Cat \xrightarrow{- \otimes \Sp^\cn} \Groth
\]
which we will denote by $\varphi$. The composition of $\varphi$ with the inclusion $\Groth \rightarrow \Groth^+$ preserves finite colimits, and consequently for  every object $\ccal$ in $\Groth$ the composite functor
\[
\Cat^\omega \xrightarrow{\varphi} \Groth \xrightarrow{ - \otimes \ccal} \Groth
\]
preserves finite colimits. We may thus give $\Groth$ the structure of a symmetric monoidal $2$-category, where the Hom category between a pair of objects $\ccal, \dcal$ is given by the ind-object $\Ical \mapsto \Hom_{\Groth}(\varphi(\Ical) \otimes \ccal, \dcal)$ on $\Cat^\omega$.

Let $\varphi^\st: \Cat^\omega \rightarrow \Groth^\st$ be the composition of $\varphi$ with the projection $\Groth \rightarrow \Groth^\st$. As above,  the composition of $\varphi^\st$ with the inclusion $\Groth^\st \rightarrow \overline{\Groth}$ preserves finite colimits, so for every object $(\ccal, [\ccal])$ in $\Groth^\st$ the composite functor
\[
\Cat^\omega \xrightarrow{\varphi^\st} \Groth \xrightarrow{ - \otimes (\ccal, [\ccal])} \Groth^\st
\]
preserves finite colimits. It follows that $\Groth^\st$ has the structure of a symmetric monoidal $2$-category as well.

Let $L: \Groth \rightarrow \Groth_\comp$ and $L^\st: \Groth^\st \rightarrow \Groth^\st_\comp$ be the localization functors. Then the functors $L \circ \varphi$ and $L^\st \circ \varphi^\st$ induce $2$-categorical structures on $\Groth_\comp$ and $\Groth^\st_\comp$. We note that the commutative square of symmetric monoidal categories
\[
\begin{tikzcd}
\Groth \arrow{r}{} \arrow{d}{L} & \Groth^\st \arrow{d}{L^\st} \\
\Groth_\comp \arrow{r}{} & \Groth_\comp^\st
\end{tikzcd}
\] 
admits an enhancement to a commutative square of symmetric monoidal $2$-categories and symmetric monoidal functors, where the vertical arrows are localization functors.
\end{construction}

\begin{remark}\label{remark adjointability in groth overline}
The projection $\Groth^\st \rightarrow \Pr^L_{\St}$ admits the structure of symmetric monoidal functor of symmetric monoidal $2$-categories. This induces fully faithful functors at the level of Homs: the Hom category in $\Groth^\st$ from $(\Sp(\ccal), [\ccal])$ to $(\Sp(\dcal), [\dcal])$ may be identified with the category of colimit preserving functors $f: \Sp(\ccal), \Sp(\dcal)$ which are right t-exact up to a shift. In particular, we have the following:
\begin{itemize}
\item A morphism $(\Sp(\ccal), [\ccal]) \rightarrow (\Sp(\dcal), [\dcal])$ in $\Groth^\st$  is right adjointable if and only if the underlying map of presentable stable categories $f: \ccal \rightarrow \dcal$ admits a colimit preserving right adjoint  which is right t-exact up to a shift.
\item A commutative square in $\Groth^\st$ is vertically right adjointable if and only if its vertical arrows are right adjointable, and the underlying commutative square of presentable stable categories is vertically right adjointable.
\end{itemize}
\end{remark}

%%%%%%%%%%%%%%%%%%%%%%%%%%%%%%%%%%%%%%%%%%%%%%%%%%%%%%%%%%%%%%%%%%%%%%%%
%%%%%%%%%%%%%%%%%%%%%%%%%%%%%%%%%%%%%%%%%%%%%%%%%%%%%%%%%%%%%%%%%%%%%%%%
%%%%%%%%%%%%%%%%%%%%%%%%%%%%%%%%%%%%%%%%%%%%%%%%%%%%%%%%%%%%%%%%%%%%%%%%
%%%%%%%%%%%%%%%%%%%%%%%%%%%%%%%%%%%%%%%%%%%%%%%%%%%%%%%%%%%%%%%%%%%%%%%%
%%%%%%%%%%%%%%%%%%%%%%%%%%%%%%%%%%%%%%%%%%%%%%%%%%%%%%%%%%%%%%%%%%%%%%%%
%%%%%%%%%%%%%%%%%%%%%%%%%%%%%%%%%%%%%%%%%%%%%%%%%%%%%%%%%%%%%%%%%%%%%%%%

\subsection{Completed relative tensor products}\label{subsection completed relative}

Our next goal is to study relative tensor products in $\Groth_\comp$. As already noted, we will only allow ourselves to form relative tensor products over a class of algebras that has well behaved Bar resolutions. The necessary condition is supplied by the notion of almost rigidity from \ref{subsection almost rigid}, specialized to algebras in the symmetric monoidal $2$-category $\Groth^\st_\comp$ from \ref{subsection grothstcomp}:

\begin{definition}\label{definition admissible algebra}
Let $\ccal$ be an algebra in $\Groth_\comp$. We say that $\ccal$ is admissible if $(\Sp(\ccal), \ccal)$ defines an almost rigid algebra of $\Groth^\st_\comp$.
\end{definition}

\begin{remark}\label{remark conditions admissibility}
It follows from remark \ref{remark adjointability in groth overline} that an algebra $\ccal$ in $\Groth_\comp$ is admissible if and only if the following conditions are satisfied:
\begin{enumerate}[\normalfont (a)]
\item The functor $\mu: \Acal \otimeshat \acal \rightarrow \acal$ is compact and the right adjoint to $\Sp(\mu)$ is right t-exact up to a shift.
\item The commutative square of categories
\[
\begin{tikzcd}[column sep = large]
 \Sp( \Acal \otimeshat \acal \otimeshat \acal ) \arrow{d}{ \Sp(\mu) \otimeshat \id } \arrow{r}{ \id \otimeshat \Sp(\mu) } &\Sp( \acal \otimeshat \acal )  \arrow{d}{ \mu } \\
 \Sp( \acal \otimeshat \acal )  \arrow{r}{ \Sp(\mu) } & \Sp(\acal) 
\end{tikzcd}
\]
is both horizontally and vertically right adjointable.
\end{enumerate}
\end{remark}

Combining theorem \ref{theo modules over almost rigid algebras} and corollary \ref{coro apply in arrow cat}  with remark \ref{remark adjointability in groth overline} we obtain the following:
 
\begin{theorem}\label{theorem modules over admissible}
Let $\acal$ be an admissible algebra in $\Groth_\comp$ and let $\Mcal$ be a left $\acal$-module in $\Groth_\comp$. Then:
\begin{enumerate}[\normalfont (1)]
\item The action map $\mu_{\Mcal}: \acal \otimeshat \Mcal \rightarrow \Mcal$ is compact and the right adjoint to $\Sp(\mu_{\Mcal})$ is right t-exact up to a shift.
\item Let $f: \Mcal \rightarrow \Ncal$ be a morphism of left $\acal$-modules in $\Groth_\comp$. Assume that $f$ is compact and the right adjoint to $\Sp(f)$ is right t-exact up to a shift. Then the commutative square 
\[
\begin{tikzcd}[column sep = large]
\Sp(\acal \otimeshat \Mcal) \arrow{r}{\Sp(\mu_{\Mcal})} \arrow{d}{\Sp(\id \otimeshat f)} & \Sp(\Mcal) \arrow{d}{\Sp(f)} \\
\Sp(\Acal \otimeshat \Ncal) \arrow{r}{\Sp(\mu_{\Ncal})}  & \Sp(\Ncal)
\end{tikzcd}
\]
is vertically right adjointable.
\item Let $f: \Mcal \rightarrow \Ncal$ be a morphism of left $\acal$-modules in $\Groth_\comp$. Then the commutative square 
\[
\begin{tikzcd}[column sep = large]
\Sp(\acal \otimeshat \Mcal) \arrow{r}{\Sp(\mu_{\Mcal})} \arrow{d}{\Sp(\id \otimeshat f)} & \Sp(\Mcal) \arrow{d}{\Sp(f)} \\
\Sp(\Acal \otimeshat \Ncal) \arrow{r}{\Sp(\mu_{\Ncal})}  & \Sp(\Ncal)
\end{tikzcd}
\]
is horizontally right adjointable.
\end{enumerate}
\end{theorem}
 
\begin{notation}
Let $\acal$ be a algebra in $\Groth$, and let $\Mcal, \Ncal$ be a pair of a right and a left $\acal$-module in $\Groth$. We will denote by $\Bar_\acal(\Mcal, \Ncal)_\bullet$ the corresponding Bar construction. If $\acal, \Mcal, \Ncal$ are complete we will denote by $\Barhat_\acal(\Mcal, \Ncal)_\bullet$ the completion of $\Bar_\acal(\Mcal, \Ncal)_\bullet$. In other words, $\Barhat_\acal(\Mcal, \Ncal)_\bullet$ is the Bar construction computed in the symmetric monoidal category $\Groth_\comp$.
\end{notation}

\begin{corollary}\label{coro face maps compact}
Let $\acal$ be an admissible algebra in $\Groth_\comp$,  and let $\Mcal, \Ncal$ be a pair of a right and a left $\acal$-module in $\Groth_\comp$. Then:
\begin{enumerate}[\normalfont (1)]
\item Every face map of $\Barhat_{\Acal}(\Mcal, \Ncal)_\bullet$ is compact.
\item Every face map of $\Bar_{\acal}(\Mcal, \Ncal)_\bullet$ is almost compact.
\end{enumerate}
\end{corollary}
\begin{proof}
 Part (1) follows from theorem \ref{theorem modules over admissible} and its opposite, since every face map of the simplicial object $\Barhat_{\acal}(\Mcal, \Ncal)_{\bullet}$ is the action map  for some left or right $\acal$-module in $\Groth_\comp$. Part (2) follows from this together with proposition \ref{proposition lex y almost compact completion}.
\end{proof}

\begin{corollary}\label{corollary tensor product over admissible}
Let $\acal$ be an admissible commutative algebra in $\Groth_\comp$. Then:
\begin{enumerate}[\normalfont (1)]
\item  $\Mod_{\acal}(\Groth_\comp)$ has a symmetric monoidal structure where the tensor product of a pair of modules $\Mcal$ and $\Ncal$ is given by the geometric realization of $\Barhat_{\acal}(\Mcal, \Ncal)_{\bullet}$.
\item Let $f: \acal \rightarrow \acal'$ be a morphism of commutative algebras in $\Groth_\comp$, and assume that $\acal'$ is admissible. Then there is a symmetric monoidal extension of scalars functor $\Mod_{\acal}(\Groth_\comp) \rightarrow \Mod_{\acal'}(\Groth_\comp)$ which sends each $\acal$-module $\Mcal$ to the geometric realization of $\Barhat_{\acal}(\acal', \Mcal)_\bullet$.
\end{enumerate}
\end{corollary}
\begin{proof}
To prove the proposition it will suffice to show the following:
\begin{enumerate}[$(\ast)$]
\item  For every pair of $\acal$-modules $\Mcal$ and $\Ncal$ in $\Groth_\comp$ the Bar construction  $\Barhat_{\acal}(\Mcal, \Ncal)_{\bullet}$ admits a geometric realization in $\Groth_\comp$, which is preserved by the functor $- \otimeshat \ecal : \Groth_\comp \rightarrow \Groth_\comp$ for every object $\ecal$.
\end{enumerate}
Since $\Groth_\comp$ is a symmetric monoidal localization of $\Groth$, it will be enough to show that  $\Barhat_{\acal}(\Mcal, \Ncal)_{\bullet}$ admits a geometric realization in $\Groth$ which is preserved by the functor $- \otimes \ecal : \Groth  \rightarrow \Groth $ for every object $\ecal$.  This follows from corollary \ref{coro face maps compact}, in light of propositions \ref{prop colimits y limits} and \ref{proposition tensor products and colimits}.
\end{proof}

\begin{notation}
Let $\acal$ be an admissible commutative algebra in $\Groth_\comp$. For each pair of $\acal$-modules $\Mcal, \Ncal$ in $\Groth_\comp$ we denote by $\Mcal \otimeshat_{\acal} \Ncal$ their tensor product in $\Mod_{\acal}(\Groth_\comp)$.
\end{notation}

\begin{remark}\label{remark functoriality modcomp}
Let $\CAlg^{\text{adm}}(\Groth_\comp)$ be the category of admissible commutative algebras in $\Groth_{\comp}$. Then as in corollary \ref{corollary tensor product over admissible} we may construct a functor 
\[
\CAlg^{\text{adm}}(\Groth_\comp) \rightarrow \CAlg(\cathat)
\]
that sends each admissible commutative algebra $\acal$ to $\Mod_{\acal}(\Groth_\comp)$.
\end{remark}

We now discuss how the formation of completed relative tensor products interacts with various properties of morphisms.

\begin{definition}
Let $\acal$ be an algebra in $\Groth $ and let $f: \Mcal \rightarrow \Mcal'$ be a morphism of left $\acal$-modules. We say that $f$ is left exact (resp. compact, resp. almost compact) if the underlying functor of Grothendieck prestable categories is left exact (resp. compact, resp. almost compact). 
\end{definition}

\begin{proposition}\label{proposition relative tensor compact}
Let $\acal$ be an admissible commutative algebra in $\Groth_\comp$. Let $f: \Mcal \rightarrow \Mcal'$ be a  compact morphism of $\acal$-modules in $\Groth_\comp$, and let $\Ncal$ be an $\acal$-module in $\Groth_\comp$. Then the induced map  $f \otimeshat \id : \Mcal \otimeshat_{\acal} \Ncal \rightarrow \Mcal' \otimeshat_\acal \Ncal$ is almost compact.
\end{proposition}
\begin{proof}
 By proposition \ref{proposition tensor y compact y lex}  and corollary \ref{coro face maps compact} we have that 
 \[
 \Bar_{\acal}(f, \Ncal)_{\bullet}: \Bar_{\acal}(\Mcal, \Ncal)_{\bullet} \rightarrow \Bar_{\acal}(\Mcal', \Ncal)_{\bullet}
 \]
is a levelwise compact morphism of simplicial objects of $\Groth$ with almost compact face maps. Applying proposition \ref{prop colimits y limits} we deduce that $ \Bar_{\acal}(f, \Ncal)_{\bullet}$ admits a geometric realization in $\Groth$, which is almost compact. Its completion is $f \otimeshat \id $, so the proposition follows from an application of proposition \ref{proposition lex y almost compact completion}.
\end{proof}

\begin{proposition}\label{proposition relative tensor left exact}
Let $\acal$ be an admissible commutative algebra in $\Groth_\comp$. Let $f: \Mcal \rightarrow \Mcal'$ be a left exact morphism of $\acal$-modules in $\Groth_\comp$, and let $\Ncal$ be an $\acal$-module in $\Groth_\comp$. Then the induced map  $f \otimeshat \id : \Mcal \otimeshat_{\acal} \Ncal \rightarrow \Mcal' \otimeshat_\acal \Ncal$ is left exact.
\end{proposition}
\begin{proof}
 By proposition \ref{proposition tensor y compact y lex}  and corollary \ref{coro face maps compact} we have that 
 \[
 \Bar_{\acal}(f, \Ncal)_{\bullet}: \Bar_{\acal}(\Mcal, \Ncal)_{\bullet} \rightarrow \Bar_{\acal}(\Mcal', \Ncal)_{\bullet}
 \]
 is a levelwise left exact morphism of simplicial objects of $\Groth$ with almost compact face maps. Passing to completions and using proposition \ref{proposition lex y almost compact completion} we see that the same holds for $\Barhat_{\acal}(f, \Ncal)_\bullet$. To show that $f \otimeshat \id$ is left exact it will suffice to show that  $\Barhat_{\acal}(f, \Ncal)_\bullet$ admits a left exact geometric realization in $\Groth$. By a combination of proposition \ref{proposition commute limits and colimits} and lemma \ref{lemma vertical right adj and sp} it will be enough to show that for every face map $[n] \rightarrow [n+1]$ the commutative square
 \[
 \begin{tikzcd}
 \Sp(\Barhat_{\acal}(\Mcal, \Ncal)_{n+1}) \arrow{r}{} \arrow{d}{} & \Sp( \Barhat_{\acal}(\Mcal, \Ncal)_{n } ) \arrow{d}{} \\
 \Sp( \Barhat_{\acal}(\Mcal', \Ncal)_{n+1} ) \arrow{r}{} & \Sp( \Barhat_{\acal}(\Mcal', \Ncal)_{n} ) 
 \end{tikzcd} 
 \]
 is horizontally right adjointable. The above square has the form
 \[
 \begin{tikzcd}[column sep = large]
 \Sp(\acal \otimeshat \ccal) \arrow{r}{ \Sp(\mu_\ccal)} \arrow{d}{\Sp(\id \otimes g)} & \Sp(\ccal) \arrow{d}{\Sp(g)} \\
 \Sp(\acal \otimeshat \ccal') \arrow{r}{\Sp(\mu_{\ccal'})} & \Sp(\ccal')
 \end{tikzcd}
 \]
 where $g: \ccal \rightarrow \ccal'$ is a morphism of $\acal$-modules in $\Groth_\comp$. The desired claim now follows from   part (3) of theorem \ref{theorem modules over admissible}.
\end{proof}

We finish this section with two propositions concerning the compatibility of completed relative tensor products with colimits and limits.

\begin{proposition}\label{proposition colimits in Modcomp}
Let $\acal$ be an admissible commutative algebra in $\Groth_\comp$, and denote by $\Mod_{\acal}(\Groth_\comp)^\acnorm$ the wide subcategory of $\Mod_{\acal}(\Groth_\comp)^\acnorm$ on the almost compact morphisms.
\begin{enumerate}[\normalfont (1)]
\item The category $\Mod_{\acal}(\Groth_\comp)^\acnorm$ admits small colimits, which are preserved by the inclusion $\Mod_{\acal}(\Groth_\comp)^\acnorm \rightarrow \Mod_{\acal}(\Groth_\comp)$ and the forgetful functor to $\Groth_\comp$.
\item For every object $\ecal$ in $\Mod_{\acal}(\Groth_\comp)$ the composite functor
\[
\Mod_{\acal}(\Groth_\comp)^\acnorm \hookrightarrow \Mod_{\acal}(\Groth_\comp) \xrightarrow{\ecal \otimeshat_\acal -} \Mod_{\acal}(\Groth_\comp)
\]
preserves small colimits.
\end{enumerate}
\end{proposition}
\begin{proof}
Item (1) follows from proposition \ref{proposition completed tensor y colimits}, by \cite{HA} corollary 4.2.3.5. We now prove item (2). Let $\ccal_\alpha$ be a diagram in $\Mod_{\acal}(\Groth_\comp)^\acnorm$ with colimit $\ccal$. Applying \cite{HA} corollary 4.2.3.5 once more, we may reduce to showing that  for every $\dcal$ in $\Groth_\comp$ we have that $\dcal \otimeshat (\ecal \otimeshat_\acal \ccal)$ is the colimit of $ \dcal \otimeshat (\ecal \otimeshat_\acal \ccal_\alpha)$ in $\Groth_\comp$. Combining corollary \ref{coro face maps compact} with proposition \ref{proposition completed tensor y colimits} we reduce to showing that $| \dcal \otimeshat  \Barhat_{\acal}(\ecal, \ccal)_\bullet|$ is the colimit of $|\dcal \otimeshat \Barhat_{\acal}(\ecal, \ccal_\alpha)_\bullet|$ in $\Groth_\comp$. To do so it is enough to prove that $ \dcal \otimeshat  \Barhat_{\acal}(\ecal, \ccal)_n$ is the colimit of $\dcal \otimeshat \Barhat_{\acal}(\ecal, \ccal_\alpha)_n$ in $\Groth_\comp$ for every $n \geq 0$. This follows from proposition \ref{proposition completed tensor y colimits}.
\end{proof}

\begin{proposition}\label{proposition relative tensor y limits}
Let $\acal$ be an admissible commutative algebra in $\Groth_\comp$. Let $(\Mcal_\alpha)$ be a diagram of $\acal$-modules in $\Groth_\comp$  with left exact transitions. Let $\Ncal$ be an $\acal$-module in $\Groth_\comp$, and assume that $\Sp(\Ncal)$ is dualizable as an object of $\Mod_{\Sp(\acal)}(\Pr^L_{\St})$. Then the canonical map $(\lim \Mcal_\alpha) \otimeshat_{\acal} \Ncal  \rightarrow \lim (\mcal_\alpha \otimeshat_{\acal} \Ncal)$ is an equivalence.
\end{proposition}
\begin{proof}
Let $\Mcal = \lim \Mcal_\alpha$. Combining corollary  \ref{coro face maps compact} with proposition \ref{prop colimits y limits} we see that the simplicial objects $\Bar_{\acal}(\Mcal, \Ncal)_{\bullet}$ and $\Bar_{\acal}(\Mcal_\alpha, \Ncal)_\bullet$ admit geometric realizations in $\Groth$. Furthermore, combining propositions \ref{proposition tensor y compact y lex} and \ref{proposition relative tensor left exact} we see that the transition functors in the diagram  $(|\Bar_{\acal}(\Mcal_\alpha, \Ncal)_\bullet|)$  are left exact, and  so are the functors $|\Bar_{\acal}(\Mcal, \Ncal)_\bullet| \rightarrow |\Bar_{\acal}(\Mcal_\alpha, \Bcal)_\bullet|$. Since the completion functor $\Groth \rightarrow \Groth_\comp$ preserves limits of diagrams with left exact transitions, it will suffice to show that the canonical map
\[
|\Bar_{\acal}(\Mcal, \Ncal)_{\bullet}| \rightarrow \lim |\Bar_{\acal}(\Mcal_\alpha, \Bcal)_\bullet|
\]
(which we already know to be left exact) is an isomorphism. This will follow if we show that the functor
\[
\Sp(|\Bar_{\acal}(\Mcal, \Ncal)_{\bullet}|) \rightarrow  \Sp \lim (|\Bar_{\acal}(\Mcal_\alpha, \Bcal)_\bullet|)
\]
is an equivalence. We may identify the above with the canonical functor
\[
\Sp(\Mcal) \otimes_{\Sp(\acal)} \Sp(\Ncal) \rightarrow \lim  (\Sp(\Mcal_\alpha) \otimes_{\Sp(\acal)} \Sp(\Ncal)).
\]
Our claim now follows from the fact that $\Sp(\Ncal)$ is dualizable as a $\Sp(\acal)$-module.
\end{proof}

%%%%%%%%%%%%%%%%%%%%%%%%%%%%%%%%%%%%%%%%%%%%%%%%%%%%%%%%%%%%%%%%%%%%%%%%
%%%%%%%%%%%%%%%%%%%%%%%%%%%%%%%%%%%%%%%%%%%%%%%%%%%%%%%%%%%%%%%%%%%%%%%%
%%%%%%%%%%%%%%%%%%%%%%%%%%%%%%%%%%%%%%%%%%%%%%%%%%%%%%%%%%%%%%%%%%%%%%%%
%%%%%%%%%%%%%%%%%%%%%%%%%%%%%%%%%%%%%%%%%%%%%%%%%%%%%%%%%%%%%%%%%%%%%%%%
%%%%%%%%%%%%%%%%%%%%%%%%%%%%%%%%%%%%%%%%%%%%%%%%%%%%%%%%%%%%%%%%%%%%%%%%
%%%%%%%%%%%%%%%%%%%%%%%%%%%%%%%%%%%%%%%%%%%%%%%%%%%%%%%%%%%%%%%%%%%%%%%%

\ifx\inmain\undefined
\bibliographystyle{myamsalpha}
\bibliography{References}
\fi

%% file: Affineness.tex
%%%%%%%%%%%%%%%%%%%%%%%%%%%%%%%%%%%%%%%%%%%%%%%%%%%%%%%%%%%%%%%%%%%%%%%%
%%%%%%%%%%%%%%%%%%%%%%%%%%%%%%%%%%%%%%%%%%%%%%%%%%%%%%%%%%%%%%%%%%%%%%%%
%%%%%%%%%%%%%%%%%%%%%%%%%%%%%%%%%%%%%%%%%%%%%%%%%%%%%%%%%%%%%%%%%%%%%%%%
%%%%%%%%%%%%%%%%%%%%%%%%%%%%%%%%%%%%%%%%%%%%%%%%%%%%%%%%%%%%%%%%%%%%%%%%
%%%%%%%%%%%%%%%%%%%%%%%%%%%%%%%%%%%%%%%%%%%%%%%%%%%%%%%%%%%%%%%%%%%%%%%%
%%%%%%%%%%%%%%%%%%%%%%%%%%%%%%%%%%%%%%%%%%%%%%%%%%%%%%%%%%%%%%%%%%%%%%%%

\section{Affineness}\label{section affineness}

Let $X$ be a geometric stack. In \cite{SAG} chapter 10, Lurie defines a symmetric monoidal category $\twoQCoh^{\pst}(X)$ of quasicoherent sheaves of Grothendieck prestable categories on $X$. By definition, $\twoQCoh^{\pst}(X)$ is the limit of the symmetric monoidal categories $\Mod_{\Mod_R^\cn}(\Groth)$ over all morphisms $\Spec(R) \rightarrow X$ from an affine scheme into $X$. Inside $\twoQCoh^{\pst}(X)$ there is a full subcategory $\twoQCoh^{\pst}_\comp(X)$ consisting of those sheaves whose value on each map $\Spec(R) \rightarrow X$ is a complete Grothendieck prestable category. This may be equipped with an enhanced global sections functor
\[
\Gamma^\enh_\comp(X, -): \twoQCoh^{\pst}_\comp (X ) \rightarrow \Mod_{\QCoh(X)^\cn}(\Groth_\comp).
\]
The goal of this section is to prove that $\Gamma^{\enh}_\comp(X,-)$ is an equivalence whenever $X$ is quasi-compact and has quasi-affine diagonal (theorem \ref{theorem affineness} and corollary \ref{coro affineness with gammaenh}).

We begin this section in \ref{subsection admissibility qcoh} by showing that  the symmetric monoidal category $\QCoh(X)^\cn$ is an admissible algebra in $\Groth_\comp$, in the sense of definition \ref{definition admissible algebra}. This is  deduced as a consequence of a result that identifies the category $\QCoh(X\times Y)^\cn$ with the completion of $\QCoh(X)^\cn \otimes \QCoh(Y)^\cn$ for any geometric stack $Y$.

In \ref{subsection sheaves prestable} we review some elements of the theory of sheaves of Grothendieck prestable categories which will be used in the proof of affineness. The material here is for the most part a globalization of \ref{subsection groth}: we discuss limits, colimits, tensor products and completions in $\twoQCoh^{\pst}(X)$. We also include here a basic result concerning the pullback and pushforward functoriality for sheaves of complete Grothendieck prestable categories on geometric stacks.

Finally, in \ref{subsection affineness prestable} we state and give a proof of our affineness theorem. The admissibility of $\QCoh(X)^\cn$ plays a crucial role: combined with the material from \ref{subsection completed relative} it allows us in particular to obtain a left adjoint $\Phi_X$ to $\Gamma^\enh_\comp(X, -)$ (whose properties are used in the proof of  the full faithfulness of $\Gamma^\enh_\comp(X,-)$) and to deduce that the face maps of the completed Bar resolution of any object of $\Mod_{\QCoh(X)^\cn}(\Groth_\comp)$ are compact (which is used in our proof of the surjectivity of $\Gamma^\enh_\comp(X,-)$).

%%%%%%%%%%%%%%%%%%%%%%%%%%%%%%%%%%%%%%%%%%%%%%%%%%%%%%%%%%%%%%%%%%%%%%%%
%%%%%%%%%%%%%%%%%%%%%%%%%%%%%%%%%%%%%%%%%%%%%%%%%%%%%%%%%%%%%%%%%%%%%%%%
%%%%%%%%%%%%%%%%%%%%%%%%%%%%%%%%%%%%%%%%%%%%%%%%%%%%%%%%%%%%%%%%%%%%%%%%
%%%%%%%%%%%%%%%%%%%%%%%%%%%%%%%%%%%%%%%%%%%%%%%%%%%%%%%%%%%%%%%%%%%%%%%%
%%%%%%%%%%%%%%%%%%%%%%%%%%%%%%%%%%%%%%%%%%%%%%%%%%%%%%%%%%%%%%%%%%%%%%%%
%%%%%%%%%%%%%%%%%%%%%%%%%%%%%%%%%%%%%%%%%%%%%%%%%%%%%%%%%%%%%%%%%%%%%%%%

\subsection{Admissibility of \texorpdfstring{$\QCoh(X)^\cn$}{QCoh(X)cn}}\label{subsection admissibility qcoh}

The goal of this section is to prove the following:

\begin{theorem}\label{theorem admissibility qcoh}
Let $X$ be a quasi-compact geometric stack. Assume that the diagonal of $X$ is a quasi-compact quasi-separated relative algebraic space. Then:
\begin{enumerate}[\normalfont (1)]
\item  $\QCoh(X)^\cn$ is an admissible algebra in $\Groth_\comp$.
\item Assume that $X$ has affine diagonal. Then $\QCoh(X)^\cn$ and $\QCoh(X)^\heartsuit$ are almost rigid algebras in (the canonical $2$-categorical enhancement of) $\Groth_\comp$ and $\Pr^L$, respectively.
\end{enumerate}
\end{theorem}

We will deduce theorem \ref{theorem admissibility qcoh} from a tensor product formula for categories of quasicoherent sheaves.

\begin{notation}
Let $X$ and $Y$ be prestacks and let $p_X, p_Y$ be the projections from $X \times Y$ to $X$ and $Y$. We denote by $e_{X, Y}$ the composite map
\[
\QCoh(X)^\cn \otimes \QCoh(Y)^\cn \xrightarrow{p_X^* \otimes p_Y^*} \QCoh(X \times Y)^\cn \otimes \QCoh(X \times Y)^\cn \xrightarrow{\otimes} \QCoh(X \times Y)^\cn.
\]
In other words, $e_{X,Y}$ is the map arising from the canonical lax symmetric monoidal structure on the functor $\QCoh(-)^\cn$.
\end{notation}

\begin{theorem}\label{theo tensor product formulas}
Let $X$ be a quasi-compact geometric stack.  Assume that the diagonal of $X$ is a quasi-compact quasi-separated relative algebraic space. Then for every geometric stack $Y$ the functor $e_{X,Y}$  induces an equivalence  
\[
\QCoh(X)^\cn \otimeshat \QCoh(Y)^\cn  =  \QCoh(X \times Y)^\cn .
\] 
\end{theorem}

\begin{proof}[Proof of theorem \ref{theorem admissibility qcoh}]
We give the proof of (1); the proof of (2) is analogous. It suffices to show that $\QCoh(X)^\cn$ satisfies conditions (a) and (b) in remark \ref{remark conditions admissibility}. By theorem \ref{theo tensor product formulas}, the functor $\mu: \QCoh(X)^\cn \otimeshat \QCoh(X)^\cn \rightarrow \QCoh(X)^\cn$ may be identified with the functor of pullback along $\Delta: X \rightarrow X \times X$, so assertion (a) follows from the fact that $\Delta$ is a quasi-compact quasi-separated relative algebraic space, in light of \cite{SAG} corollaries 3.4.2.2 and 3.4.2.3. Similarly, the square in (b) is the image under $\QCoh(-)$ of the commutative square
\[
\begin{tikzcd}
X \times X \times X & \arrow{l}[swap]{\id \times \Delta} X \times X \\
X \times X \arrow{u}[swap]{\Delta \times \id} & X \arrow{l}[swap]{\Delta} \arrow{u}[swap]{\Delta}.
\end{tikzcd}
\]
The desired assertion follows from the fact that the above square is cartesian and its arrows are quasi-compact quasi-separated relative algebraic spaces.
\end{proof}

The remainder of this section is devoted to the proof of theorem \ref{theo tensor product formulas}.

\begin{lemma}\label{lemma comonadic from left exactness}
Let $f: \Ccal \rightarrow \Dcal$ be a colimit preserving functor between complete Grothendieck prestable categories. If $f$ is left exact and conservative then $f$ is comonadic.
\end{lemma}
\begin{proof}
By the monadicity theorem, it will suffice to prove that if $X^\bullet$ is an $f$-split cosimplicial diagram in $\ccal$ then $f$ preserves the totalization of $X^\bullet$. Since $\dcal$ is complete, it is enough to show that $\tau_{\leq n} \circ f$ preserves the totalization of $X^\bullet$ for all $n \geq 0$. Using the fact that left exact functors between $(n+1,1)$-categories preserve totalizations we may further reduce to showing that $\tau_{\leq n}: \ccal \rightarrow \ccal_{\leq n}$ preserves the totalization of $X^\bullet$ for all $n$. Since $\ccal$ is complete, it suffices to prove that for all $n \geq 1$ the truncation $\tau_{\leq n}$ preserves the totalization of $\tau_{\leq n+1}(X^\bullet)$. Since $f$ is left exact and conservative it will suffice to show that the truncation $\tau_{\leq n}$ preserves the totalization of $\tau_{\leq n+1}(f(X^\bullet))$. This follows from the fact that $X^\bullet$ is $f$-split.
\end{proof}

\begin{lemma}\label{lemma tensor y amplitude}
Let $\ccal, \dcal$ be Grothendieck prestable categories and let $h: \Sp(\ccal) \rightarrow \Sp(\dcal)$ be a left t-exact colimit preserving functor. Then for every Grothendieck prestable category $\Ecal$ the functor
\[
 h \otimes \id: \Sp(\ccal) \otimes \Sp(\ecal) \rightarrow \Sp(\dcal) \otimes \Sp(\ecal)
\]
is left t-exact.
\end{lemma}
\begin{proof}
By \cite{SAG} theorem C.2.4.1 we may pick a connective ring spectrum $A$ and a left exact localization functor $p: \LMod^\cn_A \rightarrow \Ecal$. Consider the commutative square
 \[
 \begin{tikzcd}
 \Sp(\ccal) \otimes \Sp(\LMod_B^\cn)\arrow{d}{\id \otimes \Sp(q)}  \arrow{r}{h \otimes \id} & \Sp(\Dcal) \otimes \Sp(\LMod_B^\cn) \arrow{d}{\id \otimes \Sp(q)} \\
 \Sp(\ccal) \otimes \Sp(\Ecal) \arrow{r}{h \otimes \id} & \Sp(\dcal) \otimes \Sp(\ecal).
 \end{tikzcd}
 \]
Our goal is to show that the bottom horizontal arrow is left t-exact. Since the right vertical arrow is t-exact and the left vertical arrow is surjective on hearts we may reduce to showing that the top horizontal arrow is left t-exact. Let $G: \LMod^\cn_B \rightarrow \Sp^\cn$ be the forgetful functor. Then  we have a commutative square
\[
\begin{tikzcd}
 \Sp(\ccal) \otimes \Sp(\LMod_B^\cn)\arrow{d}{\id \otimes \Sp(G)}  \arrow{r}{h \otimes \id} & \Sp(\Dcal) \otimes \Sp(\LMod_B^\cn) \arrow{d}{\id \otimes \Sp(G)}  \\
 \Sp(\ccal) \otimes \Sp(\Sp^\cn) \arrow{r}{h \otimes \id} & \Sp(\dcal) \otimes \Sp(\Sp^\cn).
\end{tikzcd}
\]
Here the vertical arrows are t-exact, and the bottom arrow is equivalent to $h$ so it is left t-exact. The lemma now follows from the fact that the right vertical arrow is conservative, since it may be identified with the forgetful functor $\LMod_B(\Sp(\Dcal)) \rightarrow \Sp(\Dcal)$.
\end{proof}

\begin{lemma}\label{lemma tensor comonadics}
Let $f: \ccal \rightarrow \dcal$ be a colimit preserving functor between complete Grothendieck prestable categories. Assume that $f$ is compact, left exact, and conservative. Then for every complete Grothendieck prestable category $\Ecal$ the functor $f \otimeshat \id : \ccal \otimeshat \ecal \rightarrow \dcal \otimeshat \ecal$ is left exact and conservative.
\end{lemma}
\begin{proof}
The left exactness of $f \otimeshat \id$ follows directly from the left exactness of $f$, by a combination of propositions \ref{proposition tensor y compact y lex} and \ref{proposition lex y almost compact completion}. It remains to address conservativity. Since $f \otimeshat \id$ is left exact and its source is complete, it is enough to show that if $M$ is a $0$-truncated object in $\ccal \otimeshat \dcal$ such that $(f\otimeshat \id)(M) = 0$ then $M = 0$.  Since completion induces an equivalence on $0$-truncated objects, it will suffice to show that if $N$ is a $0$-truncated object of $\ccal \otimes \dcal$ such that $(f \otimes \id)(N) = 0$ then $N = 0$. 

Consider the following commutative square:
\[
\begin{tikzcd}[column sep = large]
\ccal \otimes \ecal \arrow{d}{} \arrow{r}{f \otimes \id} & \dcal \otimes \ecal \arrow{d}{} \\
\Sp(\ccal) \otimes \Sp(\ecal) \arrow{r}{\Sp(f) \otimes \id} & \Sp(\dcal) \otimes \Sp(\ecal)
\end{tikzcd}
\]
Here the bottom row is the stabilization of the top row. We equip the categories in the bottom row with the induced t-structures so that the categories in the top row get identified with the connective subcategories. Let $g: \Sp(\dcal) \rightarrow \Sp(\ccal)$ be the right adjoint to $\Sp(f)$. Let $\eta: \id \rightarrow g \circ \Sp(f)$ be the unit and denote by $h$ the cofiber of $\eta$. We have an exact sequence
\[
\begin{tikzcd}
N \rightarrow (g \otimes \id) ( (\Sp(f) \otimes \id) ( N) ) \rightarrow (h \otimes \id )(N) 
\end{tikzcd}
\]
in $\Sp(\ccal) \otimes \Sp(\ecal)$. Since $(f \otimes \id)(N) = 0$, the middle term in the above sequence is $0$, so we have that $(h \otimes \id)(N) = \Sigma N$. To prove that $N = 0$ it will suffice to show that $h \otimes \id$ is left t-exact. By lemma \ref{lemma tensor y amplitude} it is enough to prove that $h$ is left t-exact. To prove this we have to show that for every object $Q$ in $\Sp(\ccal)^\heartsuit$ the cofiber of $\eta(Q)$ is $0$-truncated. Since $\Sp(f)$ is t-exact, we have that $g(\Sp(f)(Q))$ is $0$-truncated, so we may reduce to showing that $\eta(Q)$ induces a monomorphism on $H_0$. Let $g': \Dcal \rightarrow \Ccal$ be the right adjoint to $f$, and denote by $\eta': \id \rightarrow g' \circ f$ the unit of the adjunction. Then $H_0(\eta(Q)) = \eta'(Q)$, so it is enough to prove that the latter is a monomorphism. Since $f$ is left exact and conservative we may reduce to showing that $f(\eta'(Q)): f(Q) \rightarrow f(g'(f(Q))$ is a monomorphism. Indeed, this map admits a retraction induced by the counit $f \circ g' \rightarrow \id$.
\end{proof}

\begin{lemma}\label{lemma comonadicity spectral}
Let $f: \ccal \rightarrow \dcal$ be a colimit preserving functor between Grothendieck prestable categories. If $f$ is left exact and comonadic then $\Sp(f): \Sp(\ccal) \rightarrow \Sp(\dcal)$ is comonadic.
\end{lemma}
\begin{proof}
We will verify the conditions of the monadicity theorem. We first show that $\Sp(f)$ is conservative. Since $\Sp(f)$ is an exact functor between stable categories, it is enough to show that if $X$ is an object of $\Sp(\ccal)$ such that $\Sp(f)(X) = 0$ then $X = 0$. Using the fact that the t-structure on $\Sp(\ccal)$ is right complete we reduce to showing that $\tau_{\geq n}X = 0$ for all integers $n$. The left exactness of $f$ implies that $\Sp(f)(\tau_{\geq n}(X) ) = \tau_{\geq n} \Sp(f)(X) = 0$. Hence $f(\Omega^n \tau_{\geq n} X )= \Omega^n \Sp(f)(\tau_{\geq n} X) = 0$. Using the fact that $f$ is conservative we deduce that $\Omega^n \tau_{\geq n} X = 0$ and hence $\tau_{\geq n} X = 0$, as desired.

Assume now given a $\Sp(f)$-split cosimplicial object $X^\bullet$ in $\Sp(\ccal)$. We wish to show that $\Sp(f)$ preserves the totalization of $X^\bullet$. Since the t-structure on $\Sp(\dcal)$ is right complete we may reduce to showing that $\tau_{\geq n} \circ \Sp(f)$ preserves the totalization of $X^\bullet$ for all integers $n$. Since $f$ is left exact, this agrees with $\Sp(f) \circ \tau_{\geq n}$, so we may reduce to showing that $\Sp(f)$ preserves the totalization of $\Sp(f)$-split cosimplicial objects of $\Sp(\ccal)_{\geq n}$. This follows from the fact that the functor $\Sp(\ccal)_{\geq n} \rightarrow \Sp(\dcal)_{\geq n}$ induced from $\Sp(f)$ is equivalent to $f$.
\end{proof}

\begin{lemma}\label{lemma base change along flat quasi affines}
Let 
 \[
 \begin{tikzcd}
 X' \arrow{d}{f'} \arrow{r}{g'} & X \arrow{d}{f} \\
 Y' \arrow{r}{g} & Y
\end{tikzcd} 
 \]
 be a cartesian square  of quasi-compact geometric stacks whose maps are quasi-compact quasi-separated relative algebraic spaces. Then for every complete Grothendieck prestable category $\ccal$ the commutative square
 \[
 \begin{tikzcd}[column sep =large]
 \Sp(\QCoh( X')^\cn  \otimeshat \ccal) & \Sp(\QCoh( X )^\cn \otimeshat \ccal) \arrow{l}[swap]{\Sp(g'^* \otimeshat \id)}  \\
 \Sp(\QCoh( Y')^\cn \otimeshat \ccal) \arrow{u}[swap]{\Sp(f'^* \otimeshat \id)}& \Sp(\QCoh( Y)^\cn  \otimeshat \ccal) \arrow{l}[swap]{\Sp(g^* \otimeshat \id)}\arrow{u}[swap]{\Sp(f^* \otimeshat \id)}
 \end{tikzcd} 
 \]
 is vertically right adjointable
\end{lemma}
\begin{proof}
By remark \ref{remark adjointability in groth overline} it suffices to show that the  commutative square in $\Groth^\st_\comp$
\[
 \begin{tikzcd}[column sep =large]
 (\Sp(\QCoh( X')^\cn  \otimeshat \ccal),\QCoh( X')^\cn  \otimeshat \ccal) &  (\Sp(\QCoh( X )^\cn  \otimeshat \ccal),\QCoh( X )^\cn  \otimeshat \ccal) \arrow{l}[swap]{\Sp(g'^* \otimeshat \id)}  \\
  (\Sp(\QCoh( Y')^\cn  \otimeshat \ccal),\QCoh( Y')^\cn  \otimeshat \ccal) \arrow{u}[swap]{\Sp(f'^* \otimeshat \id)}& (\Sp(\QCoh( Y)^\cn  \otimeshat \ccal),\QCoh( Y)^\cn  \otimeshat \ccal) \arrow{l}[swap]{\Sp(g^* \otimeshat \id)}\arrow{u}[swap]{\Sp(f^* \otimeshat \id)}
 \end{tikzcd} 
\]
is vertically right adjointable. The above is equivalent to the completed tensor product of the object $(\Sp(\ccal), \ccal)$ with the following commutative square:
\[
 \begin{tikzcd}
 (\Sp(\QCoh(X')^\cn), \QCoh(X')^\cn)   &  (\Sp(\QCoh(X )^\cn), \QCoh(X )^\cn)  \arrow{l}[swap]{\Sp(g'^*)}  \\
  (\Sp(\QCoh(Y')^\cn), \QCoh(Y')^\cn)   \arrow{u}[swap]{\Sp(f'^*)}&  (\Sp(\QCoh(Y)^\cn), \QCoh(Y)^\cn)   \arrow{l}[swap]{\Sp(g^*)}\arrow{u}[swap]{\Sp(f^*)} 
 \end{tikzcd} 
\]
It is enough to show that the above square is vertically right adjointable. By remark \ref{remark adjointability in groth overline}, it suffices to prove that the square 
\[
 \begin{tikzcd}
 \QCoh( X')    & \QCoh( X )   \arrow{l}[swap]{g'^*}  \\
 \QCoh( Y')   \arrow{u}[swap]{f'^*}& \QCoh( Y) \arrow{l}[swap]{g^*}\arrow{u}[swap]{f^*} 
 \end{tikzcd}
\]
is vertically right adjointable, and the right adjoints to the vertical arrows are colimit preserving and right t-exact up to a shift. This follows directly from our assumptions, by a combination of \cite{SAG} corollaries 3.4.2.2 and 3.4.2.3.
\end{proof}

\begin{lemma}\label{lemma global sections y tensoring}
Let $X$ be a quasi-compact geometric stack, and assume that the diagonal of $X$ is a quasi-compact quasi-separated relative algebraic space. Let $p: U \rightarrow X$ be a faithfully flat map with $U$ affine and denote by $U_\bullet$ the \v{C}ech nerve of $f$. Then for every complete Grothendieck prestable category $\ccal$ the canonical map
\[
\QCoh (X)^\cn \otimeshat \ccal \rightarrow \Tot(\QCoh (U_\bullet)^\cn \otimeshat \ccal)
\]
is an equivalence.
\end{lemma}
\begin{proof}
We note that the map in the statement is a left exact functor between Grothendieck prestable categories. Consequently, to show that it is an equivalence it will suffice to show that the induced map
\[
\Sp(\QCoh (X)^\cn \otimeshat \ccal) \rightarrow \Sp( \Tot(\QCoh (U_\bullet)^\cn \otimeshat \ccal)) = \Tot (\Sp( \QCoh (U_\bullet)^\cn \otimeshat \ccal))
\]
is an equivalence. Combining lemmas  \ref{lemma comonadic from left exactness}, \ref{lemma tensor comonadics} and \ref{lemma comonadicity spectral} we see that the functor 
\[
\Sp(p^* \otimeshat \id): \Sp(\QCoh (X)^\cn \otimeshat \ccal) \rightarrow \Sp(\QCoh (U)^\cn \otimeshat \ccal)
\]
 is comonadic.  We may thus reduce to showing that the augmented cosimplicial category $\Sp(\QCoh(X)^\cn \otimeshat \ccal) \rightarrow \Sp(\QCoh (U_\bullet)^\cn \otimeshat \ccal)$ satisfies (the dual to) the Beck-Chevalley condition from \cite{HA} corollary 4.7.5.3. This follows from lemma \ref{lemma base change along flat quasi affines}.
\end{proof}

\begin{proof}[Proof of theorem \ref{theo tensor product formulas}]
Assume first that $X$ is affine. We claim that in this case $e_{X, Y}$ is an isomorphism for all $Y$. Since $X$ is affine we have that $\QCoh(X)^\cn$ is a self dual object of $\Mod_{\Sp^\cn}(\Pr^L)$ and therefore tensoring with $\QCoh(X)^\cn$ preserves limits of diagrams in $\Groth$ with left exact transitions. From this we deduce that the source of $e_{X, Y}$ preserves coproducts and geometric realizations of flat groupoids in the $Y$ variable. The same assertions hold for the target of $e_{X,Y}$, so by induction we may reduce to the case when $Y$ is also affine, in which case the assertion is clear.

We now address the general case. Let $p: U \rightarrow X$ be a faithfully flat map with $U$ affine and denote by $U_\bullet$ the \v{C}ech nerve of $p$. Assume that the theorem is known to hold for $U_n$ for all $n$. Consider the following commutative diagram:
  \[
\begin{tikzcd}[column sep = large]
\Tot (\QCoh(U_\bullet)^\cn \otimeshat \QCoh(Y)^\cn ) \arrow{r}{\Tot \widehat{e}_{U_\bullet, Y} } & \Tot \QCoh (U_\bullet \times Y )^\cn \\
\QCoh (X)^\cn \otimeshat \QCoh (Y)^\cn \arrow{r}{\widehat{e}_{X,Y}} \arrow{u}{} & \QCoh  (X \times Y)^\cn \arrow{u}{}
\end{tikzcd}
\]
Here the right vertical arrow is an equivalence, and the top vertical arrow is also an equivalence by our assumption on $U_\bullet$. The theorem then holds for $X$ since the left vertical arrow is an equivalence (lemma \ref{lemma global sections y tensoring}).

Since we have already proven that the theorem holds for affine $X$, the above implies that the theorem holds whenever $X$ has affine diagonal, and in particular if $X$ is a quasi-affine scheme. Another iteration of the same reasoning then shows that the theorem holds whenever $X$ has quasi-affine diagonal, and in particular if $X$ is a quasi-compact quasi-separated algebraic space (see \cite{SAG} proposition 3.4.1.3). The result now follows from yet another iteration of the same reasoning.
\end{proof}
 
%%%%%%%%%%%%%%%%%%%%%%%%%%%%%%%%%%%%%%%%%%%%%%%%%%%%%%%%%%%%%%%%%%%%%%%%
%%%%%%%%%%%%%%%%%%%%%%%%%%%%%%%%%%%%%%%%%%%%%%%%%%%%%%%%%%%%%%%%%%%%%%%%
%%%%%%%%%%%%%%%%%%%%%%%%%%%%%%%%%%%%%%%%%%%%%%%%%%%%%%%%%%%%%%%%%%%%%%%%
%%%%%%%%%%%%%%%%%%%%%%%%%%%%%%%%%%%%%%%%%%%%%%%%%%%%%%%%%%%%%%%%%%%%%%%%
%%%%%%%%%%%%%%%%%%%%%%%%%%%%%%%%%%%%%%%%%%%%%%%%%%%%%%%%%%%%%%%%%%%%%%%%
%%%%%%%%%%%%%%%%%%%%%%%%%%%%%%%%%%%%%%%%%%%%%%%%%%%%%%%%%%%%%%%%%%%%%%%%
 
\subsection{Sheaves of Grothendieck prestable categories}\label{subsection sheaves prestable}
 
 We now give an overview of the theory of  sheaves of Grothendieck prestable categories, as introduced in \cite{SAG} chapter 10.
 
\begin{notation}
 For each connective commutative ring spectrum $R$ we let 
 \[
 \Groth_R = \Mod_{\Mod_R^\cn}(\Groth).
  \]
  Objects of $\Groth_R$ are called $R$-linear Grothendieck prestable categories.  As discussed in \cite{SAG} proposition D.2.2.1, if $R$ is a connective commutative ring spectrum then $\Groth_R$ is  closed under tensor products inside $\Mod_{\Mod_R^\cn}(\Pr^L)$. In particular, $\Groth_R$ inherits a symmetric monoidal structure from $\Mod_{\Mod_R^\cn}(\Pr^L)$. We denote by 
  \[
  - \otimes_R - : \Groth_R \times \Groth_R \rightarrow \Groth_R
  \]
  the corresponding tensor product functor. Furthermore for each morphism of connective commutative ring spectra $R \rightarrow S$ we have an extension of scalars functor which will be denoted by $- \otimes_R S : \Groth_R \rightarrow \Groth_S$.
  
 The assignment $R \mapsto \Groth_R$ assembles into a functor $\CAlg^\cn \rightarrow \CAlg(\cathat)$ which satisfies \'etale descent (\cite{SAG} theorem D.4.1.2). Let 
 \[
 \twoQCoh^\pst : \PreStk^\op \rightarrow \CAlg(\cathat)
 \]
be  its right Kan extension along the inclusion $\CAlg^\cn = \Aff^\op \rightarrow \PreStk^\op$. For each prestack $X$ we call $\twoQCoh^\pst(X)$ the category of quasicoherent sheaves of Grothendieck prestable categories on $X$. For each map $f: X \rightarrow Y$ we denote by $f^*: \twoQCoh^\pst(Y) \rightarrow \twoQCoh^\pst(X)$ the induced pullback functor.
\end{notation}
 
The notion of quasicoherent sheaf of Grothendieck prestable categories globalizes the notion of $R$-linear Grothendieck prestable category. The classes of left exact and (almost) compact morphisms also admit a globalization.

\begin{proposition}
Let $R \rightarrow S$ be a morphism of connective commutative ring spectra and let $f: \ccal \rightarrow \dcal$ be a morphism in $\Groth_R$. If $f$ is left exact (resp. compact, resp. almost compact) then $f \otimes_R S : \ccal \otimes_R S \rightarrow \dcal \otimes_R S$ is left exact (resp. compact, resp. almost compact).
\end{proposition}
\begin{proof}
The cases of left exactness and compactness are given by \cite{SAG} propositions D.5.2.1 and D.5.2.2. The proof for almost compactness is analogous to the proof for compactness.
\end{proof}
 
\begin{definition}
Let $X$ be a prestack and let $f: \ccal \rightarrow \dcal$ be a morphism in $\twoQCoh^\pst(X)$. We say that $f$ is left exact (resp. compact, resp. almost compact) if for every connective commutative ring spectrum $R$ and every map $\eta: \Spec(R) \rightarrow X$ the functor of Grothendieck prestable categories underlying the map $\eta^*f: \eta^* \ccal \rightarrow \eta^*\dcal$  is left exact (resp. compact, resp. almost compact).
\end{definition}
 
\begin{notation}
Let $X$ be a prestack. We denote by $\twoQCoh^{\pst, \lex}(X)$ (resp. $\twoQCoh^{\pst, \cnorm}(X)$, resp. $\twoQCoh^{\pst, \acnorm}(X)$) the wide subcategory of $\twoQCoh^\pst(X)$ on the left exact (resp. complete, resp. almost complete) morphisms.  
\end{notation}
 
\begin{proposition}\label{prop tensor left exact y compact global}
Let $X$ be a prestack. Let $f: \ccal \rightarrow \ccal'$ be a morphism in $\twoQCoh^\pst(X)$, and let $\dcal$ be an object in $\twoQCoh^\pst(X)$. If $f$ is left exact (resp. compact) then  $f \otimes \id : \ccal \otimes \dcal \rightarrow \ccal' \otimes \dcal$ is left exact (resp. compact).
\end{proposition}
\begin{proof}
It suffices to address the case when $X = \Spec(R)$ is an affine scheme. The case of left exactness follows from \cite{Fully} proposition 2.4.20. Assume now that $f$ is compact. Then $ \Sp(f \otimes_R \id)$ is equivalent to 
 \[
 \Sp(f) \otimes_R \id : \Sp(\ccal) \otimes_R \Sp(\dcal) \rightarrow \Sp(\ccal') \otimes_R \Sp(\dcal).
 \]
 Since $\Sp(f)$ admits a colimit preserving right adjoint we have that the above functor also admits a colimit preserving right adjoint. Hence $f \otimes_R \id$ is compact, as desired.
\end{proof}
 
\begin{proposition}\label{proposition limits y colimits globalized}
  Let $X$ be a prestack.
 \begin{enumerate}[\normalfont (1)]
 \item The category  $\twoQCoh^{\pst, \lex}(X)$ admits small limits, which are preserved by the inclusion into $\twoQCoh^\pst(X)$. Furthermore, for every morphism of prestacks $f: X \rightarrow Y$  the pullback functor $f^*: \twoQCoh^{\pst, \lex}(Y) \rightarrow \twoQCoh^{\pst, \lex}(X)$ preserves small limits.
 \item The category $\twoQCoh^{\pst, \cnorm}(X)$ and $\twoQCoh^{\pst, \acnorm}(X)$ admit small colimits, which are preserved by the inclusion into $\twoQCoh^\pst(X)$. Furthermore, for every morphism $f: X \rightarrow Y$ the pullback functor $f^*: \twoQCoh^{\pst, \acnorm}(Y) \rightarrow \twoQCoh^{\pst, \acnorm}(X)$ preserves small colimits.
 \end{enumerate}
 \end{proposition}
 \begin{proof}
 For each morphism $R \rightarrow S$ of connective commutative ring spectra, we have that $\Mod_S^\cn$ is a self dual object of $\Groth_R$, and consequently the functor $- \otimes_R S : \Groth_R \rightarrow \Groth_S$ preserves all limits and colimits that exist on the source. The proposition now follows from a combination of propositions \ref{prop colimits y limits} and \ref{proposition tensor products and colimits}.
\end{proof}

\begin{proposition}
Let $X$ be a prestack and let $\ccal$ be an object in $\twoQCoh^{\pst}(X)$. Then the composite functor
\[
\twoQCoh^{\pst, \acnorm}(X) \hookrightarrow \twoQCoh^{\pst}(X) \xrightarrow{\ccal \otimes - } \twoQCoh^{\pst}(X)
\]
preserves small colimits.
\end{proposition}
\begin{proof}
By proposition \ref{proposition limits y colimits globalized} it suffices to address the case when $X = \Spec(R)$ is an affine scheme. In this case the task is to show that the functor $- \otimes_R \ccal : \Groth_R \rightarrow \Groth$ preserves colimits of diagrams with almost compact transitions. We may write this functor as the geometric realization of $\Bar_{\Mod^\cn_R}(-, \ccal)_\bullet$. We may thus reduce to showing that for each $n$ the functor $\Bar_{\Mod^\cn_R}(- ,\ccal)_n: \Groth_R \rightarrow \Groth$ preserves colimits of diagrams with almost compact transitions. This follows from proposition  \ref{proposition tensor products and colimits}.
\end{proof}

There is also a globalization of the notion of complete Grothendieck prestable category.

\begin{proposition}[\cite{SAG} proposition D.5.1.3]
Let $R \rightarrow S$ be a morphism of connective commutative ring spectra and let $\ccal$ be a complete $R$-linear Grothendieck prestable category. Then $\ccal \otimes_R S$ is complete.
\end{proposition}

\begin{definition}
Let $X$ be a prestack. We say that an object $\ccal$ in $\twoQCoh^\pst(X)$ is complete if for every connective commutative ring spectrum $R$ and every map $\eta: \Spec(R) \rightarrow X$ the Grothendieck prestable category underlying $\eta^*\ccal$ is complete.
\end{definition}

\begin{notation}
Let $X$ be a prestack. We denote by $\twoQCoh^\pst_\comp(X)$ the full subcategory of $\twoQCoh^\pst(X)$ on the complete objects.
\end{notation}

The procedure of completion may also be globalized:

\begin{proposition}[\cite{SAG} proposition 10.3.1.11]\label{proposition completion globalized}
Let $X$ be a prestack. Then the inclusion $\twoQCoh^\pst_\comp(X) \rightarrow \twoQCoh^\pst(X)$ admits a left adjoint, which we denote by $\ccal \mapsto \widehat{\ccal}$. Furthermore, if $\ccal$ is an object of $\twoQCoh^\pst(X)$ then for every connective commutative ring spectrum $R$ and every map $\eta: \Spec(R) \rightarrow X$ the induced map of Grothendieck prestable categories $\eta^*\ccal \rightarrow \eta^*\widehat{\ccal}$ presents $\eta^*\widehat{\ccal}$ as the completion of $\eta^*\ccal$.
\end{proposition}

\begin{remark}\label{remark pullbacks on twoqcohcomp}
It follows from proposition \ref{proposition completion globalized} that for every prestack $X$ we have a symmetric monoidal localization functor $\widehat{(-)}: \twoQCoh^\pst(X) \rightarrow \twoQCoh^\pst_\comp(X)$, which commutes with base change. Consequently, the assignment $X \mapsto \twoQCoh^\pst_\comp(X)$ gives rise to a functor
\[
\twoQCoh^\pst_\comp : \PreStk^\op \rightarrow \CAlg(\cathat).
\] 
The above is limit preserving so it is determined by its restriction to $\CAlg^\cn$. Unwinding the definitions, this recovers the functor  $\CAlg^\cn \rightarrow \CAlg(\cathat)$ that sends each connective commutative ring spectrum $R$ to $\Mod_{\Mod_R^\cn}(\Groth_\comp)$, with the symmetric monoidal structure and change of base functoriality arising from corollary \ref{corollary tensor product over admissible} by virtue of the fact that $\Mod_R^\cn$ is an admissible commutative algebra in $\Groth_\comp$.
\end{remark}

We will need the following:
\begin{theorem}[ \cite{SAG} theorem D.6.8.1]
The assignment $R \mapsto \Mod_{\Mod_R^\cn}(\Groth_{\comp})$ is a sheaf for the fpqc topology on $\Aff$.
\end{theorem}

We finish with a basic result regarding the pushforward functoriality of $\twoQCoh^\pst_\comp$ on geometric stacks.

\begin{proposition}\label{proposition pushforward functoriality}
Let $f: X \rightarrow Y$ be a morphism of geometric stacks. Then the functor $f^*: \twoQCoh^\pst_{\comp}(Y) \rightarrow \twoQCoh^\pst_{\comp}(X)$ admits a right adjoint $f_*: \twoQCoh^\pst_{\comp}(X) \rightarrow \twoQCoh^\pst_{\comp}(Y)$. Furthermore, for every cartesian square of geometric stacks
\[
\begin{tikzcd}
X' \arrow{d}{f'} \arrow{r}{g'} & X \arrow{d}{f} \\
Y' \arrow{r}{g} & Y
\end{tikzcd}
\]
the induced commutative square of categories
\[
\begin{tikzcd}
\twoQCoh^\pst_{\comp}(X' )& \twoQCoh^\pst_{\comp}(X ) \arrow{l}[swap]{g'^*} \\
\twoQCoh^\pst_{\comp}(Y' ) \arrow{u}[swap]{f'^*} & \arrow{u}[swap]{f^*} \arrow{l}[swap]{g^*}  \twoQCoh^\pst_{\comp}(Y )
\end{tikzcd}
\]
is vertically right adjointable. 
\end{proposition}

\begin{proof}
Assume first that $X, X', Y, Y'$ are affine, so that the first commutative square in the statement is obtained by taking $\Spec$ of a pushout square
\[
\begin{tikzcd}
S' & S  \arrow{l}{} \\
R' \arrow{u}{} & \arrow{l}{} R \arrow{u}{}
\end{tikzcd}
\]
of connective commutative ring spectra.  The functor $f^*$ is given by the restriction of $- \otimes_R S : \Groth_R \rightarrow \Groth_S $ to the full subcategories on the complete objects. Consequently, it admits a right adjoint which is the restriction of the forgetful functor $\Groth_S \rightarrow \Groth_R$. Similarly, $f'^*$ admits a right adjoint. The vertical right adjointability of the square 
\[
\begin{tikzcd}
\Mod_{\Mod^\cn_{S'}}(\Groth_\comp) & \Mod_{\Mod^\cn_{S}}(\Groth_\comp)  \arrow{l}{} \\
\Mod_{\Mod^\cn_{R'}}(\Groth_\comp) \arrow{u}{} & \arrow{u}{} \arrow{l}{}\Mod_{\Mod^\cn_{R}}(\Groth_\comp) 
\end{tikzcd}
\]
would then follow if we show that the square
\[
\begin{tikzcd}
\Mod_{\Mod^\cn_{S'}}(\Pr^L) & \Mod_{\Mod^\cn_{S}}(\Pr^L)  \arrow{l}{} \\
\Mod_{\Mod^\cn_{R'}}(\Pr^L) \arrow{u}{} & \arrow{u}{} \arrow{l}{}\Mod_{\Mod^\cn_{R}}(\Pr^L) 
\end{tikzcd}
\]
is vertically right adjointable. This follows from the fact that the square
\[
\begin{tikzcd}
\Mod^\cn_{S'} & \Mod^\cn_S \arrow{l}{} \\
\Mod^\cn_{R'} \arrow{u}{} & \Mod^\cn_R \arrow{u}{} \arrow{l}{}
\end{tikzcd}
\]
is a pushout of commutative algebras in $\Pr^L$.

We now establish the general case. Applying lemma \ref{lemma how to check vert adjointable} we may reduce to the case when $Y$ and $Y'$ are affine. Write $X$ as the colimit in $\Stk$ of a diagram of affine schemes $X_\alpha = \Spec(S_\alpha)$ and flat transitions. For each $\alpha$ denote by $g_\alpha: X'_\alpha \rightarrow X_\alpha$ the base change of $g$, and let $p_\alpha: X_\alpha \rightarrow X$ and $p'_\alpha: X'_\alpha \rightarrow X'$ be the projections.

Let $\ccal$ be an object of $\twoQCoh^{\pst}_\comp(X)$. To show that $f_*$ is defined at $\ccal$ it will suffice to prove that the diagram $(f \circ p_\alpha)_* p_\alpha^* \ccal$ admits a limit in $\twoQCoh^{\pst}_\comp(Y)$. To do so it will be enough to show that for each transition $p_{\alpha, \beta}: X_\alpha \rightarrow X_{\beta}$ the morphism
\[
(f \circ p_{\beta})_* p_{\beta}^* \ccal \rightarrow (f \circ p_\alpha)_* p_\alpha^* \ccal   = (f \circ p_{\beta})_* (p_{\alpha, \beta})_* ( p_{\alpha, \beta}) ^*  p_{\beta}^* \ccal
\]
induced from the unit $\id \rightarrow (p_{\alpha, \beta})_* (p_{\alpha, \beta})^* $, is left exact. To do so it suffices to show that for every $\ccal'$ in $\Groth_{S_\beta}$ the unit map $\ccal' \rightarrow \ccal' \otimes_{S_{\beta}} S_\alpha$ is left exact. This follows from proposition \ref{prop tensor left exact y compact global}, since the pullback map $\Mod_{S_{\beta}}^\cn \rightarrow \Mod_{S_\alpha}^\cn$ is left exact.

We now show that the map $g^*f_*\ccal \rightarrow f'_* g'^* \ccal$ is an equivalence. Using the above characterization of $f_*$ (and the similar characterization of $f'_*$) we may reduce to showing that $g^*$ preserves the limit of the diagram $(f \circ p_\alpha)_* p_\alpha^* \ccal$. We claim that in fact $g^*$ preserves all limits that exist in $\twoQCoh^{\pst}_\comp(Y)$. Write $Y = \Spec(R)$ and $Y' = \Spec(R')$. Then $g^*$ is obtained by restriction of the functor
\[
- \otimes_R R' : \Mod_{\Mod_R^\cn}(\Pr^L) \rightarrow \Mod_{\Mod_{R'}^\cn}(\Pr^L)
\]
to the full subcategories on the complete objects. The desired claim follows from the fact that the above functor is right adjoint to the forgetful functor (since $\Mod_{R'}^\cn$ is self dual as an $\Mod_{R}^\cn$-module).
\end{proof}

%%%%%%%%%%%%%%%%%%%%%%%%%%%%%%%%%%%%%%%%%%%%%%%%%%%%%%%%%%%%%%%%%%%%%%%%
%%%%%%%%%%%%%%%%%%%%%%%%%%%%%%%%%%%%%%%%%%%%%%%%%%%%%%%%%%%%%%%%%%%%%%%%
%%%%%%%%%%%%%%%%%%%%%%%%%%%%%%%%%%%%%%%%%%%%%%%%%%%%%%%%%%%%%%%%%%%%%%%%
%%%%%%%%%%%%%%%%%%%%%%%%%%%%%%%%%%%%%%%%%%%%%%%%%%%%%%%%%%%%%%%%%%%%%%%%
%%%%%%%%%%%%%%%%%%%%%%%%%%%%%%%%%%%%%%%%%%%%%%%%%%%%%%%%%%%%%%%%%%%%%%%%
%%%%%%%%%%%%%%%%%%%%%%%%%%%%%%%%%%%%%%%%%%%%%%%%%%%%%%%%%%%%%%%%%%%%%%%%
 
\subsection{\texorpdfstring{$\twoQCoh_\comp^{\pst}$}{2QCohPstcomp}-affineness}\label{subsection affineness prestable}

We now arrive at the main result of this section, which allows us to recover $\twoQCoh^\pst_\comp(X)$ from $\QCoh(X)^\cn$ whenever $X$ is a quasi-compact geometric stack with quasi-affine diagonal.

\begin{construction}\label{construction Phi}
Let $\Gcal$ be the full subcategory of $\PreStk$ on the quasi-compact geometric stacks $X$ such that the diagonal of $X$ is a quasi-compact quasi-separated relative algebraic space. It follows from  theorem \ref{theorem admissibility qcoh} (in light of remark \ref{remark functoriality modcomp}) that the assignment $X \mapsto \Mod_{\QCoh(X)^\cn}(\Groth_\comp)$ gives rise to a functor 
\[
\Mod_{\QCoh(-)^\cn}(\Groth_\comp): \Gcal^\op \rightarrow \CAlg(\cathat).
\] 
The right Kan extension of $\Mod_{\QCoh(-)^\cn}(\Groth_\comp)|_{\Aff^\op}$ along the inclusion $\Aff^\op \rightarrow \Gcal^\op$ recovers the restriction to $\mathcal{G}^\op$ of the functor $\twoQCoh^\pst_\comp$ (see remark \ref{remark pullbacks on twoqcohcomp}). It follows that for each object $X$ of $\Gcal$ we have a symmetric monoidal functor
\[
\Phi_X : \Mod_{\QCoh(X)^\cn}(\Groth_\comp) \rightarrow \twoQCoh^\pst_\comp(X).
\]
\end{construction}

\begin{theorem}\label{theorem affineness}
Let $X$ be a quasi-compact geometric stack with quasi-affine diagonal. Then the functor 
\[
\Phi_X : \Mod_{\QCoh(X)^\cn}(\Groth_\comp) \rightarrow \twoQCoh^\pst_\comp(X) 
\]
from construction \ref{construction Phi} is an equivalence.
\end{theorem}

\begin{remark}\label{remark generality affineness}
Theorem \ref{theorem affineness} continues to hold  with the same proof  provided that $X$ is a quasi-compact geometric stack, the diagonal of $X$ is a quasi-compact quasi-separated relative algebraic space, and the following condition is satisfied:
\begin{enumerate}[\normalfont $(\ast)$]
\item Let $U \rightarrow X$ and $V \rightarrow X$ be a pair of morphisms with $U, V$ affine. Then $\QCoh(U \times_X V)^\cn$ is generated under colimits and extensions by $\Ocal_{U \times_X V}$.
\end{enumerate}
\end{remark}

Before going into the proof of theorem \ref{theorem affineness}, we describe a slight reformulation.

\begin{notation}\label{notation gamma enh}
Let $X$ be a geometric stack. We denote by
\[
\Gamma_\comp(X, -): \twoQCoh^\pst_\comp(X) \rightarrow \twoQCoh^\pst_\comp(\Spec(\mathbb{S})) = \Groth_\comp
\]
the functor of pushforward along the projection $X \rightarrow \Spec(\mathbb{S})$ (see proposition \ref{proposition pushforward functoriality}). We equip $\Gamma_{\comp}(X, -)$ with its canonical lax symmetric monoidal structure. Let $\QCoh_X^\cn$ be the unit of $\twoQCoh^\pst_\comp(X)$, and observe that we have an equivalence of symmetric monoidal categories $\Gamma_\comp(X, \QCoh_X^\cn) = \QCoh(X)^\cn$. We denote by
\[
\Gamma_\comp^\enh(X, -): \twoQCoh^\pst_\comp(X)  = \Mod_{\QCoh_X^\cn}(\twoQCoh^\pst_\comp(X)) \rightarrow \Mod_{\QCoh(X)^\cn}(\Groth_\comp)
\]
the induced functor.
\end{notation}

\begin{corollary}\label{coro affineness with gammaenh}
Let $X$ be a quasi-compact geometric stack with quasi-affine diagonal. Then the functor
\[
\Gamma_\comp^\enh(X, -): \twoQCoh^\pst_\comp(X) \rightarrow \Mod_{\QCoh(X)^\cn}(\Groth_\comp)
\]
from notation \ref{notation gamma enh} is an equivalence.
\end{corollary}
\begin{proof}
This follows from theorem \ref{theorem affineness}, using the fact that the functor
\begin{align*}
\Mod_{\QCoh(X)^\cn}(\Groth_\comp) & =   \Mod_{\QCoh(X)^\cn}(\Mod_{\QCoh(X)^\cn}(\Groth_\comp)) \\ & \rightarrow \Mod_{\QCoh(X)^\cn}(\Groth_\comp)
\end{align*}
induced from the lax symmetric monoidal structure on the forgetful functor 
\[
\Mod_{\QCoh(X)^\cn}(\Groth_\comp) \rightarrow \Groth_\comp
\]
 is an equivalence.
\end{proof}

We devote the remainder of this section to the proof of theorem \ref{theorem affineness}.

\begin{lemma}\label{lemma Gammaenh right adjoint}
Let $X$ be a quasi-compact geometric stack with quasi-affine diagonal. Then the functor $\Gamma^\enh_\comp(X, -)$ from notation \ref{notation gamma enh} is right adjoint to $\Phi_X$.
\end{lemma}
\begin{proof}
Write $X$ as the colimit in $\Stk$ of a diagram of affine schemes $X_\alpha$ and flat transitions, and for each $\alpha$ let $p_\alpha: X_\alpha \rightarrow X$ be the projection.  Arguing as in the proof of proposition \ref{proposition pushforward functoriality}, we have that $\Phi_X$ admits a right adjoint that sends each object $\ccal$ in $\twoQCoh^{\pst}_\comp(X)$ to the limit of the diagram $p_\alpha^* \ccal$ (where here $p_\alpha^*\ccal$ is regarded as a $\QCoh(X)^\cn$-module via restriction of scalars along $p_\alpha^*: \QCoh(X)^\cn \rightarrow \QCoh(X_\alpha)^\cn$). Since $\QCoh(X)^\cn = \lim \QCoh(X_\alpha)^\cn$, we see that the right adjoint to $\Phi_X$ preserves the unit. We may now identify $\Gamma^\enh_\comp(X,-)$ with the composite functor
\begin{align*}
\twoQCoh^\pst_\comp(X) = \Mod_{\QCoh_X^\cn}(\twoQCoh^\pst_\comp(X)) & \rightarrow \Mod_{\QCoh(X)^\cn}( \Mod_{\QCoh(X)^\cn}(\Groth_\comp) ) \\ & \rightarrow  \Mod_{\QCoh(X)^\cn}(\Groth_\comp)
\end{align*}
where the first arrow is induced from the lax symmetric monoidal structure on the right adjoint to $\Phi_X$ and the second arrow is induced from the lax symmetric monoidal structure on the forgetful functor $\Mod_{\QCoh(X)^\cn}(\Groth_\comp) \rightarrow \Groth_\comp$. The lemma follows from the fact that the second arrow is an equivalence.
\end{proof}
 
\begin{lemma}\label{lemma phi of left exact}
Let $X$ be a quasi-compact geometric stack with quasi-affine diagonal. 
\begin{enumerate}[\normalfont (1)]
\item Let $f: \ccal \rightarrow \dcal$ be a morphism in $\Mod_{\QCoh(X)^\cn}(\Groth_\comp)$. Assume that the functor underlying $f$ is left exact (resp. compact). Then $\Phi_X(f)$ is left exact (resp. almost compact).
\item The functor $\Phi_X$ preserves limits of diagrams with left exact transitions.
\end{enumerate}
\end{lemma}
\begin{proof}
To prove (1) it  suffices to show that for every connective commutative ring spectrum $R$ and every map $\eta: \Spec(R) \rightarrow X$, the induced functor
\[
- \otimeshat_{\QCoh(X)^\cn} \Mod_R^\cn : \Mod_{\QCoh(X)^\cn}(\Groth_\comp) \rightarrow \Groth_\comp
\]
preserves left exactness and sends compact maps to almost compact maps. This follows directly from propositions \ref{proposition relative tensor compact} and  \ref{proposition relative tensor left exact}.   By proposition \ref{proposition limits y colimits globalized}, in order to prove (2) we must show in addition that the above functor preserves limits of diagrams with left exact transitions. Using proposition \ref{proposition relative tensor y limits} we may reduce to showing that $\Mod_R$ is dualizable as a $\QCoh(X)$-module in $\Pr^L$. Since  $\eta$ is quasi-affine we have a $\QCoh(X)$-linear identification $\Mod_R = \Mod_B(\QCoh(X))$ for some commutative algebra $B$ in $\QCoh(X)$. The fact that this is dualizable is \cite{HA} remark 4.8.4.8.
\end{proof}

\begin{lemma}\label{lemma base change pullback}
Let $R \rightarrow S$ be a flat morphism of connective commutative ring spectra, and let $h: \ccal \rightarrow \dcal$ be a compact morphism in $\Groth_R$. Then the commutative square
\[
\begin{tikzcd}
\ccal \otimes_R S \arrow{r}{h \otimes_R S} & \dcal \otimes_R S \\
\ccal \arrow{r}{h} \arrow{u}{- \otimes_R S} & \dcal \arrow{u}{- \otimes_R S}
\end{tikzcd}
\] 
is horizontally right adjointable.
\end{lemma}
\begin{proof}
The square in the statement sits in a commutative diagram
\[
\begin{tikzcd}
\ccal \arrow{r}{h} & \dcal \\
\ccal \otimes_R S \arrow{r}{h \otimes_R S} \arrow{u}{} & \dcal \otimes_R S \arrow{u}{} \\
\ccal \arrow{r}{h} \arrow{u}{- \otimes_R S} & \dcal \arrow{u}{- \otimes_R S}
\end{tikzcd}
\] 
where the top vertical arrows are the forgetful functors. We note that the top square is vertically left adjointable, and since the horizontal arrows admit right adjoints we have that it is also horizontally right adjointable. Since the top left vertical arrow is conservative, the lemma will follow if we are able to show that the outer commutative square in the above diagram is horizontally right adjointable. In other words, we have to show that the right adjoint to $h$ commutes with the actions of the connective $R$-module $S$. By \cite{HA} theorem 7.2.2.15, we may write $S$ as a filtered colimit of a diagram of finitely generated free $R$-modules. The desired claim follows from the fact that the right adjoint to $h$ commutes with filtered colimits and with the action of dualizable objects.
\end{proof}

\begin{lemma}\label{lemma Gamma y colimits}
Let $X$ be a geometric stack. 
\begin{enumerate}[\normalfont(1)]
\item The functor  $\Gamma^\enh_\comp(X, -)$ sends compact arrows to compact arrows. 
\item Let $\ccal_\beta$ be a diagram in $\twoQCoh^{\pst}_\comp(X)$ with compact transitions. Then $\Gamma^\enh_\comp(X, -)$ preserves the colimit of $\ccal_\beta$.
\end{enumerate}
\end{lemma}
\begin{proof}
By proposition \ref{proposition colimits in Modcomp}, it is enough to show that the lemma holds with $\Gamma^\enh_\comp(X,-)$ replaced with $\Gamma_\comp(X, -): \twoQCoh^\pst_\comp(X) \rightarrow \Groth_\comp$. We will show (1) and (2) simultaneously, by proving that $\Gamma_\comp(X, -)$ preserves the colimit of $\ccal_\beta$, and that the transitions in $\Gamma_\comp(X, \ccal_\beta)$ are compact.

  Write $X$ as the colimit of a diagram of affine schemes $X_\alpha$ with flat transitions, and denote by $p_\alpha: X_\alpha \rightarrow X$ the projections. Then we have $  \Gamma_\comp(X, \ccal_\beta) = \lim p_\alpha^* \ccal_\beta $ (where here we regard $p_\alpha^* \ccal_\beta$ as an object of $\Groth_\comp$). Let $\ccal$ be the colimit of the diagram $\ccal_\beta$ in $\twoQCoh^{\pst}(X)$, so that $\widehat{\ccal}$ is the colimit of $\ccal_\beta$ in $\twoQCoh^{\pst}_\comp(X)$. By proposition \ref{proposition limits y colimits globalized} for each $\alpha$ we have that $p_\alpha^* \ccal$ is the colimit in $\Groth^\cnorm$ of $p_\alpha^*\ccal_\beta$. Combining lemma \ref{lemma base change pullback} and proposition \ref{proposition commute limits and colimits} we deduce that the transitions in the diagram $\lim p_\alpha^* \ccal_\beta$ are compact, and its colimit in $\Groth$ is given by $\lim p^*_\alpha \ccal$. Since the completion functor $\Groth \rightarrow \Groth_\comp$ preserves limits of diagrams with left exact transitions, we deduce that the colimit in $\Groth_\comp$ of the diagram $\lim p_\alpha^* \ccal_\beta$ is given by $\lim \widehat{p_\alpha^* \ccal} =  \lim  p_\alpha^*\widehat{\ccal}$, as desired.   
\end{proof}

\begin{notation}
Let $\ccal$ be a category with finite limits, and assume given a pair of maps $U \rightarrow X \leftarrow V$ in $\ccal$. We denote by $\operatorname{CoBar}_X(U, V)_\bullet$ be the cosimplicial object of $\ccal$ with entries $\operatorname{CoBar}_X(U, V)_n = U \times X^n \times V$. Note that there is a canonical coaugmentation $U \times_X V \rightarrow \operatorname{CoBar}_X(U, V)_\bullet$.
\end{notation}

\begin{lemma}\label{lemma QCoh of cobar}
Let $X$ be a  geometric stack with quasi-affine diagonal, and let $p: U \rightarrow X$ and $q: V \rightarrow X$ be a pair of maps with $U, V$ affine. Then $\QCoh(U \times_X V)^\cn$ is the geometric realization of $\QCoh( \operatorname{CoBar}_X(U, V)_\bullet)^\cn$ in $\Groth$. 
\end{lemma}
\begin{proof}
Since the diagonal of $X$ is quasi-affine, the coface maps of $\operatorname{CoBar}_X(U, V)_\bullet$ are quasi-affine. Consequently, the face maps of $\QCoh( \operatorname{CoBar}_X(U, V)_\bullet)^\cn$ are compact. To prove the lemma it will suffice to show the following:
\begin{enumerate}[\normalfont (i)]
\item $\QCoh(U \times_X V)$ is the geometric realization of $\QCoh( \operatorname{CoBar}_X(U, V)_\bullet)$ in $\Pr^L$.
\item $\QCoh(U \times_X V)^\cn$ is generated under colimits and extensions by the image of the pullback functor $\QCoh(U \times V)^\cn \rightarrow \QCoh(U \times_X V)^\cn$. 
\end{enumerate}

Since $U \times_X V$ is quasi-affine we have that $\QCoh(U \times_X V)^\cn$ is generated under colimits and extensions by the unit (\cite{SAG} corollary 2.5.6.4), and hence (ii) follows. We now establish (i). Passing to right adjoints we reduce to showing that $\QCoh(U \times_X V)$ is the totalization of the  cosimplicial category $\QCoh(\operatorname{CoBar}_X(U, V)_\bullet)_*$. obtained from $\operatorname{CoBar}_X(U, V)_\bullet$ by using the pushforward functoriality on $\QCoh$. 

Consider the cosimplicial stack $ \operatorname{CoBar}_X(U, V)_{\bullet + 1}$  which is obtained from $ \operatorname{CoBar}_X(U, V)_\bullet $ by composing with the functor $- \star [0] : \Delta \rightarrow \Delta$, and note that we have a canonical map $ \operatorname{CoBar}_X(U, V)_\bullet  \rightarrow  \operatorname{CoBar}_X(U, V)_{\bullet + 1} $.  For every injective map $\varphi: [n] \rightarrow [m]$  the induced commutative square of stacks 
\[
\begin{tikzcd}
\operatorname{CoBar}_X(U, V)_n  \arrow{d}{} \arrow{r}{} & \operatorname{CoBar}_X(U, V)_m \arrow{d}{} \\
 \operatorname{CoBar}_X(U, V)_{n+1}  \arrow{r}{} & \operatorname{CoBar}_X(U, V)_{m+1}
\end{tikzcd}
\]
is cartesian and has quasi-affine horizontal arrows. Consequently, the commutative square of  categories
\[
\begin{tikzcd}
\QCoh(\operatorname{CoBar}_X(U, V)_n)_* \arrow{d}{} \arrow{r}{} & \QCoh(\operatorname{CoBar}_X(U, V)_m)_* \arrow{d}{} \\
  \QCoh(\operatorname{CoBar}_X(U, V)_{n+1})_* \arrow{r}{} &  \QCoh(\operatorname{CoBar}_X(U, V)_{m+1})_*
\end{tikzcd}
\]
is vertically left adjointable. Similarly, the commutative square 
\[
\begin{tikzcd}
U \times_X V \arrow{d}{} \arrow{r}{} & \operatorname{CoBar}_X(U, V)_0 \arrow{d}{} \\
\operatorname{CoBar}_X(U, V)_0 \arrow{r}{} & \operatorname{CoBar}_X(U, V)_{1}
\end{tikzcd}
\]
is cartesian and has quasi-affine arrows, and hence the commutative square of categories 
\[
\begin{tikzcd}
\QCoh(U \times_X V) \arrow{d}{} \arrow{r}{} & \QCoh( \operatorname{CoBar}_X(U, V)_0 )_* \arrow{d}{} \\
\QCoh(\operatorname{CoBar}_X(U, V)_0 )_* \arrow{r}{} & \QCoh( \operatorname{CoBar}_X(U, V)_1)_*
\end{tikzcd}
\]
is vertically left adjointable. An application of lemma \ref{lemma how to check vert adjointable} now shows that the square
\[
\begin{tikzcd}
\QCoh(U \times_X V) \arrow{d}{} \arrow{r}{} & \Tot \QCoh(\operatorname{CoBar}_X(U, V)_\bullet)_* \arrow{d}{} \\
\QCoh(\operatorname{CoBar}_X(U, V)_0)_* \arrow{r}{\id} & \QCoh(\operatorname{CoBar}_X(U, V)_0)_*
\end{tikzcd}
\]
is vertically left adjointable. Since the right vertical arrow is monadic we may reduce to showing that the left vertical arrow is monadic. This follows from the fact that the map $U \times_X V \rightarrow U \times V$ is quasi-affine. 
\end{proof}

\begin{lemma}\label{lemma adjointability vs tensor products}
Let
\[
\begin{tikzcd}
\Bcal' & \acal' \arrow{l}{} \\
\Bcal \arrow{u}{} & \acal \arrow{u}{} \arrow{l}{}
\end{tikzcd}
\]
be a commutative square of admissible commutative algebras in $\Groth_\comp$. Then the following are equivalent:
\begin{enumerate}[\normalfont (1)]
\item The square
\[
\begin{tikzcd}
\Mod_{\Bcal'}(\Groth_\comp) & \Mod_{\acal'}(\Groth_\comp) \arrow{l}[swap]{- \otimeshat_{\acal'} \Bcal'} \\
\Mod_{\Bcal }(\Groth_\comp)\arrow{u}[swap]{- \otimeshat_{\Bcal} \Bcal'} &  \Mod_{\acal}(\Groth_\comp) \arrow{l}[swap]{- \otimeshat_{\acal} \Bcal} \arrow{u}[swap]{- \otimeshat_{\acal} \acal'}
\end{tikzcd}
\]
is vertically right adjointable.
\item The canonical map $\Bcal \otimeshat_{\acal} \acal' \rightarrow \Bcal'$ is an equivalence.
\end{enumerate}
\end{lemma}
\begin{proof}
Item (1) amounts to the assertion that for every object $\Mcal$ in $\Mod_{\acal'}(\Groth_\comp)$ the canonical map
\[
\eta_{\Mcal} : \Mcal \otimeshat_{\acal} \Bcal \rightarrow \Mcal \otimeshat_{\acal'} \Bcal'
\]
is an equivalence. If this holds then in particular setting $\Mcal = \acal'$ we deduce that (2) holds. It remains to show the converse. By proposition \ref{proposition colimits in Modcomp}, the source and the target of $\eta_{\Mcal}$ are functors of $\Mcal$ which preserve colimits of diagrams with almost compact transitions. We note that $\Mcal$ is the colimit of $\Barhat_{\acal'}(\acal', \Mcal)_\bullet$, which has compact transitions. It therefore suffices to show that for every $n \geq 0$ the map $\eta_{\Barhat_{\acal'}(\acal', \Mcal)_n}$ is an isomorphism. Replacing $\Mcal$ with $\Barhat_{\acal'}(\acal', \Mcal)_n$ we may now reduce to the case when $\Mcal = \acal' \otimeshat \Ncal$ is a free $\acal'$-module. Since the functors of extension and restriction of scalars commute with the action of $\Groth_\comp$, we may further reduce to the case $\Mcal = \acal'$, which follows from (2).
\end{proof}

\begin{proof}[Proof of theorem \ref{theorem affineness}]
We first prove that $\Phi_X$ is a localization functor. To show this it is enough to prove that the commutative square of categories
\[
\begin{tikzcd}
\twoQCoh^{\pst}_\comp(X) & \twoQCoh^{\pst}_\comp(X) \arrow{l}[swap]{\id} \\
\twoQCoh^{\pst}_\comp(X) \arrow{u}{\id} & \Mod_{\QCoh(X)^\cn}(\Groth_\comp) \arrow{l}[swap]{\Phi_X} \arrow{u}{\Phi_X}
\end{tikzcd}
\]
is vertically right adjointable. Applying lemma \ref{lemma how to check vert adjointable} we may reduce to showing that for every map $p: U \rightarrow X$ with $U$ affine the commutative square
\[
\begin{tikzcd}[column sep = large]
\twoQCoh^{\pst}_\comp(U) & \twoQCoh^{\pst}_\comp(X) \arrow{l}[swap]{p^*} \\
\twoQCoh^{\pst}_\comp(U) \arrow{u}{\id} & \Mod_{\QCoh(X)^\cn}(\Groth_\comp) \arrow{l}[swap]{p^* \circ \Phi_X} \arrow{u}{\Phi_X}
\end{tikzcd}
\]
is vertically right adjointable.  Combining lemma \ref{lemma phi of left exact} with proposition \ref{proposition limits y colimits globalized} we see that $p^* \circ \Phi_X$ preserves limits of diagrams with left exact transitions. Arguing as in the proof of proposition \ref{proposition pushforward functoriality}, we may now reduce to showing that for every map $q: V \rightarrow X$ with $V$ affine the commutative square
\[
\begin{tikzcd}
\twoQCoh^{\pst}_\comp(U\times_X V) & \twoQCoh^{\pst}_\comp(V) \arrow{l}[swap]{p'^*} \\
\twoQCoh^{\pst}_\comp(U) \arrow{u}{q'^*} & \Mod_{\QCoh(X)^\cn}(\Groth_\comp) \arrow{l}[swap]{p^* \circ \Phi_X} \arrow{u}{q^* \circ \Phi_X}
\end{tikzcd}
\]
is vertically right adjointable, where here $p'$ and $q'$ denote the base changes of $p$ and $q$. The above square is equivalent to the square obtained by passage to categories of modules in $\Groth_\comp$ of the square of admissible commutative algebras 
\[
\begin{tikzcd}
\QCoh(U\times_X V)^\cn & \QCoh(V)^\cn \arrow{l}[swap]{p'^*} \\
\QCoh(U)^\cn \arrow{u}{q'^*} & \QCoh(X)^\cn \arrow{l}[swap]{p^*} \arrow{u}{q^*}.
\end{tikzcd}
\]
Using lemma \ref{lemma adjointability vs tensor products} we may reduce to showing that the canonical map 
\[
\QCoh(U)^\cn \otimeshat_{\QCoh(X)^\cn} \QCoh(V)^\cn \rightarrow \QCoh(U \times_X V)^\cn
\] is an equivalence.  By theorem \ref{theo tensor product formulas} we have an equivalence of simplicial objects 
\[
\Barhat_{\QCoh(X)^\cn}(\QCoh(U)^\cn, \QCoh(V)^\cn)_{\bullet} = \QCoh(\operatorname{CoBar}_X(U, V)_\bullet)^\cn 
\]
which commutes with the augmentations to $\QCoh(U \times_X V)^\cn$. The desired claim now follows from  lemma \ref{lemma QCoh of cobar}.
 
 We next show that for every complete Grothendieck prestable category $\Dcal$ the free module $  \QCoh(X)^\cn \otimeshat \dcal$  belongs to the image of $\Gamma_\comp^\enh(X, -)$.  Let $p: U \rightarrow X$ be a faithfully flat map with $U$ affine, and denote by $U_\bullet$ the \v{C}ech nerve of $U$. By lemma \ref{lemma global sections y tensoring} we have that $\QCoh(X)^\cn \otimeshat \dcal$ is the totalization of the cosimplicial $\QCoh(X)^\cn$-module $\QCoh(U_\bullet)^\cn \otimeshat \dcal$. Since the face maps in this cosimplicial category are left exact, we see by lemma \ref{lemma phi of left exact} that the face maps in $\Phi_X( \QCoh(U_\bullet)^\cn \otimeshat \dcal)$ are left exact as well, and hence they admit a limit in $\twoQCoh^\pst_\comp(X)$ by  proposition \ref{proposition limits y colimits globalized}. It follows that the totalization of 
 \[
 \Gamma^\enh_\comp(X, \Phi_X( \QCoh(U_\bullet)^\cn \otimeshat \dcal))
 \] 
 belongs to the image of $\Gamma^\enh_\comp(X, -)$. Assume for a moment that $\QCoh(U_n)^\cn \otimeshat \dcal$ is known to be in the image of $\Gamma^\enh_\comp(X, -)$ for all $n$. Then, since $\Gamma^\enh_\comp(X,-)$ is fully faithful, we have an equivalence of cosimplicial $\QCoh(X)^\cn$-modules 
 \[
  \Gamma^\enh_\comp(X, \Phi_X( \QCoh(U_\bullet)^\cn \otimeshat \dcal)) =   \QCoh(U_\bullet)^\cn \otimeshat \dcal
 \]
 and therefore the totalization  of the right hand side belongs to the image of $\Gamma_\comp^\enh(X, -)$, as desired.
 
 We may thus reduce to showing that $\QCoh(U_n)^\cn \otimeshat \dcal$ belongs to the image of $\Gamma^\enh_\comp(X, -)$ for all $n$. This is obtained by restriction of scalars from $\QCoh(U_n)^\cn$. Consider the commutative square of categories
 \[
 \begin{tikzcd}[column sep = huge]
 \twoQCoh^\pst_\comp(U_n) \arrow{d}{} \arrow{r}{\Gamma^\enh_\comp(U_n, -)} & \Mod_{\QCoh(U_n)^\cn}(\Groth^\lex_\comp) \arrow{d}{} \\
  \twoQCoh^\pst_\comp(X) \arrow{r}{\Gamma^\enh_\comp(X, -)} & \Mod_{\QCoh(X)^\cn}(\Groth^\lex_\comp) 
 \end{tikzcd}
 \]
 where the left vertical arrow is given by pushforward along $U_n \rightarrow X$ and the right vertical arrow is given by restriction of scalars. Using the commutativity of the above we may reduce to showing that $\QCoh(U_n)^\cn \otimeshat \dcal$ belongs to the image of $\Gamma^\enh_\comp(U_n, -)$. Replacing $X$ by $U_n$ we may now assume that $X$ is quasi-affine. In this case $U_n$ is affine so that $\Phi_{U_n}$ is an equivalence. In particular, $\Gamma^\enh_{\comp}(U_n,-)$ is surjective, so its image contains $\QCoh(U_n)^\cn \otimeshat \dcal$, as desired.
 
 To finish the proof we will show that $\Gamma_\comp^\enh(X, -)$ is surjective. Let $\ccal$ be an object of $\Mod_{\QCoh(X)^\cn}(\Groth_\comp)$. We may write $\ccal$ as the geometric realization of the completed Bar construction $\Barhat_{\QCoh(X)^\cn}(\QCoh(X)^\cn, \ccal)_\bullet$. Each entry in this simplicial object is a free $\QCoh(X)^\cn$-module and in particular belongs to the image of $\Gamma^\enh_\comp(X, -)$. Using the fully faithfulness of $\Gamma^\enh_\comp(X, -)$   we have
 \[
 \Barhat_{\QCoh(X)^\cn}(\QCoh(X)^\cn, \ccal)_\bullet = \Gamma^\enh_\comp(X, \Phi_X(  \Barhat_{\QCoh(X)^\cn}(\QCoh(X)^\cn, \ccal)_\bullet  ) ).
 \]
 We claim that the face maps of $ \Phi_X(  \Barhat_{\QCoh(X)^\cn}(\QCoh(X)^\cn, \ccal)_\bullet  )$ are compact. To prove this it suffices to show that for every map $\eta: \Spec(R) \rightarrow X$ the face maps of 
 \[
 \eta^* \Phi_X(  \Barhat_{\QCoh(X)^\cn}(\QCoh(X)^\cn, \ccal)_\bullet  )
 \]
 are compact. The simplicial Grothendieck prestable category underlying the above is given by $\Barhat_{\QCoh(X)^\cn}(\Mod_R^\cn, \ccal)_\bullet$. The fact that the face maps are compact now follows from  corollary \ref{coro face maps compact} . Applying lemma \ref{lemma Gamma y colimits} we now have
 \begin{align*}
 \ccal &= | \Barhat_{\QCoh(X)^\cn}(\QCoh(X)^\cn, \ccal)_\bullet| \\
  & = | \Gamma^\enh_\comp(X, \Phi_X(  \Barhat_{\QCoh(X)^\cn}(\QCoh(X)^\cn, \ccal)_\bullet  ) )| \\
  & =  \Gamma^\enh_\comp(X, |\Phi_X(  \Barhat_{\QCoh(X)^\cn}(\QCoh(X)^\cn, \ccal)_\bullet  )| )
 \end{align*}
 and in particular $\ccal$ belongs to the image of $\Gamma^\enh_\comp(X,-)$, as desired.
\end{proof}

%%%%%%%%%%%%%%%%%%%%%%%%%%%%%%%%%%%%%%%%%%%%%%%%%%%%%%%%%%%%%%%%%%%%%%%%
%%%%%%%%%%%%%%%%%%%%%%%%%%%%%%%%%%%%%%%%%%%%%%%%%%%%%%%%%%%%%%%%%%%%%%%%
%%%%%%%%%%%%%%%%%%%%%%%%%%%%%%%%%%%%%%%%%%%%%%%%%%%%%%%%%%%%%%%%%%%%%%%%
%%%%%%%%%%%%%%%%%%%%%%%%%%%%%%%%%%%%%%%%%%%%%%%%%%%%%%%%%%%%%%%%%%%%%%%%
%%%%%%%%%%%%%%%%%%%%%%%%%%%%%%%%%%%%%%%%%%%%%%%%%%%%%%%%%%%%%%%%%%%%%%%%
%%%%%%%%%%%%%%%%%%%%%%%%%%%%%%%%%%%%%%%%%%%%%%%%%%%%%%%%%%%%%%%%%%%%%%%%

\ifx\inmain\undefined
\bibliographystyle{myamsalpha}
\bibliography{References}
\fi

%% file: ConsequencesAffineness.tex
%%%%%%%%%%%%%%%%%%%%%%%%%%%%%%%%%%%%%%%%%%%%%%%%%%%%%%%%%%%%%%%%%%%%%%%%
%%%%%%%%%%%%%%%%%%%%%%%%%%%%%%%%%%%%%%%%%%%%%%%%%%%%%%%%%%%%%%%%%%%%%%%%
%%%%%%%%%%%%%%%%%%%%%%%%%%%%%%%%%%%%%%%%%%%%%%%%%%%%%%%%%%%%%%%%%%%%%%%%
%%%%%%%%%%%%%%%%%%%%%%%%%%%%%%%%%%%%%%%%%%%%%%%%%%%%%%%%%%%%%%%%%%%%%%%%
%%%%%%%%%%%%%%%%%%%%%%%%%%%%%%%%%%%%%%%%%%%%%%%%%%%%%%%%%%%%%%%%%%%%%%%%
%%%%%%%%%%%%%%%%%%%%%%%%%%%%%%%%%%%%%%%%%%%%%%%%%%%%%%%%%%%%%%%%%%%%%%%%

\section{Consequences of affineness}\label{section consequences}

The goal of this section is to present some consequences of theorem \ref{theorem affineness}. We begin in \ref{subsection tensor product and integral transform formulas} by discussing applications to the study of tensor product and integral transform formulas for categories of quasicoherent sheaves. Here we give proofs of theorem  \ref{theorem tensor products introduction} and corollaries \ref{corollary integral transforms introduction} and \ref{corollary factorization homology introduction}, as well as some related statements. The fundamental observation is that the combination of $\twoQCoh^{\pst}_\comp$-affineness (theorem \ref{theorem affineness}) together with base change for $\twoQCoh^{\pst}_\comp$ (proposition \ref{proposition pushforward functoriality}) immediately gives rise to a tensor product formula identifying the category of connective quasicoherent sheaves on a fiber product $X \times_Y Z$ with the completed tensor product $\QCoh(X)^\cn \otimeshat_{\QCoh(Y)^\cn} \QCoh(Z)^\cn$.

We then proceed in \ref{subsection tannaka prestable} to give a proof of part (1) of theorem \ref{theorem intro tannaka}. We in fact present two proofs: in the first one, we make use of theorem \ref{theorem admissibility qcoh} to reduce this result to Lurie's Tannaka duality theorem (in its quasi-affine diagonal version, see \cite{SAG} theorem 9.2.0.2); our second proof is relatively self contained, and makes explicit use of the theory of sheaves of categories and theorem \ref{theorem affineness}.

%%%%%%%%%%%%%%%%%%%%%%%%%%%%%%%%%%%%%%%%%%%%%%%%%%%%%%%%%%%%%%%%%%%%%%%%
%%%%%%%%%%%%%%%%%%%%%%%%%%%%%%%%%%%%%%%%%%%%%%%%%%%%%%%%%%%%%%%%%%%%%%%%
%%%%%%%%%%%%%%%%%%%%%%%%%%%%%%%%%%%%%%%%%%%%%%%%%%%%%%%%%%%%%%%%%%%%%%%%
%%%%%%%%%%%%%%%%%%%%%%%%%%%%%%%%%%%%%%%%%%%%%%%%%%%%%%%%%%%%%%%%%%%%%%%%
%%%%%%%%%%%%%%%%%%%%%%%%%%%%%%%%%%%%%%%%%%%%%%%%%%%%%%%%%%%%%%%%%%%%%%%%
%%%%%%%%%%%%%%%%%%%%%%%%%%%%%%%%%%%%%%%%%%%%%%%%%%%%%%%%%%%%%%%%%%%%%%%%

\subsection{Tensor product and integral transform formulas}\label{subsection tensor product and integral transform formulas}

We begin by discussing an application of theorem \ref{theorem affineness} to the study of tensor product formulas for categories of quasicoherent sheaves.

\begin{theorem}\label{theorem relative tensor products}
Let $f: X \rightarrow Y$ and $g: Z \rightarrow Y$  be maps of quasi-compact geometric stacks with quasi-affine diagonal. Then the canonical functor
\[
\QCoh(X)^\cn \otimeshat_{\QCoh(Y)^\cn} \QCoh(Z)^\cn \rightarrow \QCoh(X \times_Y Z)^\cn
\]
is an equivalence.
\end{theorem}
\begin{proof}
By lemma \ref{lemma adjointability vs tensor products}, it is enough to show that the commutative square
\[
\begin{tikzcd}[column sep = 5cm]
\Mod_{\QCoh(X\times_Y Z)^\cn}(\Groth_\comp) & \Mod_{\QCoh(X)^\cn}(\Groth_\comp) \arrow{l}[swap]{- \otimeshat_{\QCoh(X)^\cn} \QCoh(X \times_Y Z)^\cn} \\
\Mod_{\QCoh(Z)^\cn}(\Groth_\comp) \arrow{u}[swap]{- \otimeshat_{\QCoh(Z)^\cn} \QCoh(X\times_Y Z)^\cn} & \Mod_{\QCoh(Y)^\cn}(\Groth_\comp) \arrow{l}[swap]{- \otimeshat_{\QCoh(Y)^\cn} \QCoh(Z)^\cn} \arrow{u}[swap]{- \otimeshat_{\QCoh(Y)^\cn} \QCoh(X)^\cn}
\end{tikzcd}
\]
is vertically right adjointable. Using theorem \ref{theorem affineness} we see that the above is equivalent to the square
\[
\begin{tikzcd}
\twoQCoh^{\pst}_\comp(X \times_Y Z) & \twoQCoh^{\pst}_\comp(X ) \arrow{l}{} \\
\twoQCoh^{\pst}_\comp(  Z) \arrow{u}{} & \twoQCoh^{\pst}_\comp(Y) \arrow{l}[swap]{g^*} \arrow{u}[swap]{f^*}.
\end{tikzcd}
\]
The result now follows from an application of proposition \ref{proposition pushforward functoriality}.
\end{proof}

\begin{remark}
Theorem \ref{theorem relative tensor products} continues to hold with the same proof provided that $X$ is a quasi-compact geometric stack, the diagonal of $X$ is a quasi-compact quasi-separated relative algebraic space, and the following condition is satisfied:
\begin{enumerate}[\normalfont $(\ast)$]
\item Let $U \rightarrow X$ and $V \rightarrow X$ be a pair of morphisms with $U, V$ affine. Then $\QCoh(U \times_X V)^\cn$ is generated under colimits and extensions by $\Ocal_{U \times_X V}$.
\end{enumerate}
This reduces to the assertion that theorem \ref{theorem affineness} holds in this generality, see remark \ref{remark generality affineness}.
\end{remark}

\begin{corollary}\label{coro tensor product stable w t structure}
Let $f: X \rightarrow Y$ and $g: Z \rightarrow Y$  be maps of quasi-compact geometric stacks with quasi-affine diagonal. Equip $\QCoh(X) \otimes_{\QCoh(Y)} \QCoh(Z)$ with the t-structure whose connective half is generated under colimits and extensions by those objects of the form $\Fcal \otimes \Gcal$ with $\Fcal$ and $\Gcal$ connective. Then the canonical functor
\[
\QCoh(X) \otimes_{\QCoh(Y) } \QCoh(Z)  \rightarrow \QCoh(X \times_Y Z) 
\]
is t-exact, and exhibits $\QCoh(X \times_Y Z)$ as the left completion of $\QCoh(X) \otimes_{\QCoh(Y) } \QCoh(Z) $.
\end{corollary}
\begin{proof}
By corollary \ref{coro face maps compact} the Bar construction $\Bar_{\QCoh(Y)^\cn}(\QCoh(X)^\cn, \QCoh(Z)^\cn)_\bullet$ has almost compact face maps, and it therefore admits a geometric realization in $\Groth$. It follows from theorem \ref{theorem relative tensor products} that $\QCoh(X \times_Y Z)^\cn$ is the completion of this geometric realization. The corollary follows from this by passing to spectrum objects, using the characterization of colimits in $\Groth^\acnorm$ from remark \ref{remark colimit along compact}.
\end{proof}

\begin{corollary}\label{coro tensor product clasico}
Let $f: X \rightarrow Y$ and $g: Z \rightarrow Y$  be maps of quasi-compact geometric stacks with quasi-affine diagonal. Assume that $Y$ has affine diagonal. Then for every $n \geq 0$ the canonical functor 
\[
\QCoh(X)^\cn_{\leq n} \otimes_{\QCoh(Y)^\cn_{\leq n}} \QCoh(Z)^\cn_{\leq n} \rightarrow \QCoh(X \times_Y Z)^\cn_{\leq n} 
\]
is an equivalence. In particular, setting $n = 0$ we have that the canonical functor
 \[
 \QCoh(X)^\heartsuit \otimes_{\QCoh(Y)^\heartsuit} \QCoh(Z)^\heartsuit \rightarrow \QCoh(X \times_Y Z)^\heartsuit
 \]
is an equivalence.
\end{corollary}
\begin{proof}
  Combining part (2) of theorem \ref{theorem admissibility qcoh} with theorem \ref{theo modules over almost rigid algebras} we see that the face maps in  $\Barhat_{\QCoh(Y)^\cn}(\QCoh(X)^\cn, \QCoh(Z)^\cn)_\bullet$ admit colimit preserving right adjoints.  By \cite{SAG} remark C.3.5.4 we have that the geometric realization of $\Barhat_{\QCoh(Y)^\cn}(\QCoh(X)^\cn, \QCoh(Z)^\cn)_\bullet$ inside $\Pr^L$ is in fact Grothendieck prestable. This geometric realization may be computed as the limit of the semicosimplicial diagram obtained by passing to right adjoints of the face maps, and hence it is complete. Using theorem \ref{theorem relative tensor products} we deduce that $\QCoh(X \times_Y Z)^\cn$ is the geometric realization in $\Pr^L$ of $\Barhat_{\QCoh(Y)^\cn}(\QCoh(X)^\cn, \QCoh(Z)^\cn)_\bullet$. We now have
  \begin{align*}
  \QCoh(X \times_Y Z)^\cn_{\leq n} & =  |\Barhat_{\QCoh(Y)^\cn}(\QCoh(X)^\cn, \QCoh(Z)^\cn)_\bullet|_{\leq n}  \\ & = |(\Barhat_{\QCoh(Y)^\cn}(\QCoh(X)^\cn, \QCoh(Z)^\cn)_\bullet)_{\leq n}| \\ & = |(\Bar_{\QCoh(Y)^\cn}(\QCoh(X)^\cn, \QCoh(Z)^\cn)_\bullet)_{\leq n}| \\ &= \QCoh(X)^{\cn}_{\leq n}  \otimes_{\QCoh(Y)^\cn_{\leq n}} \QCoh(Z)^\cn_{\leq n}. \qedhere
  \end{align*}
\end{proof}

The tensor product formulas above involve either completions or restriction to truncated objects. We now discuss some results that apply to the stable categories without any reference to t-structures, and which are thus of the same nature as the tensor product formulas found for instance in \cite{BZFN}.

\begin{corollary}\label{coro dualizability QCoh}
Let $X$ be a quasi-compact geometric stack with quasi-affine diagonal. Assume that the structure sheaf $\Ocal_X$ is compact and truncated. Let $\mu : \QCoh(X) \otimes  \QCoh(X) \rightarrow \QCoh(X)$ be the functor induced by tensor products. Then the composite functor
\[
\QCoh(X) \otimes \QCoh(X) \xrightarrow{\mu} \QCoh(X) \xrightarrow{\Gamma(X, -)} \Sp
\]
is the counit of a self-duality in $\Pr^L_{\St}$.
\end{corollary}
\begin{proof}
Using corollary \ref{coro tensor product stable w t structure} together with the fact that $\Ocal_X$ is truncated we may find an object in $\QCoh(X) \otimes \QCoh(X)$ whose image in $\QCoh(X \times  X)$ recovers the structure sheaf of the diagonal. The desired claim now follows from \cite{SAG} proposition 9.4.3.1.
\end{proof}

\begin{corollary}\label{corollary smoothness and properness}
Let $X$ be a quasi-compact geometric stack with quasi-affine diagonal. Assume that $\Ocal_X$ is compact and truncated. Then $\QCoh(X)$ is smooth and proper as an algebra in $\Pr^L_{\St}$. 
\end{corollary}
\begin{proof}
Properness follows directly from corollary \ref{coro dualizability QCoh}. Smoothness follows from \cite{SAG} proposition 9.4.2.1, in light of \cite{SAG} corollary 9.4.2.6.
\end{proof}

\begin{corollary}\label{corollary tensor product when compact and truncated}
Let $f: X \rightarrow Y$ and $g: Z \rightarrow Y$  be maps of stacks, where $X$ and $Y$ are quasi-compact geometric stacks with quasi-affine diagonal. Assume that $\Ocal_X$ and $\Ocal_Y$ are compact and truncated. Then the canonical functor
\[
\QCoh(X) \otimes_{\QCoh(Y)} \QCoh(Z)  \rightarrow \QCoh(X \times_Y Z) 
\] 
is an equivalence. 
\end{corollary}
\begin{proof}
It follows from corollaries \ref{coro dualizability QCoh} and  \ref{corollary smoothness and properness} that $\QCoh(X)$ is dualizable as a module over $\QCoh(Y)$.  Consequently, both sides of the functor in the statement are limit preserving in $Z$. We may thus reduce to the case when $Z$ is affine. In this case the map $g$ is quasi-affine, and hence $\QCoh(Z) = \Mod_{g_*\Ocal_Z}(\QCoh(Y))$. Let $f':X \times_Y Z \rightarrow Z$ and $g': X \times_Y Z \rightarrow X$ be the projections. Then we have
\begin{align*}
\QCoh(X) \otimes_{\QCoh(Y)} \QCoh(Z) & = \Mod_{f^* g_*\Ocal_Z}(\QCoh(X)) \\ &=  \Mod_{g'_*f'^*\Ocal_Z}(\QCoh(X)) \\ & = \Mod_{g'_* \Ocal_{X \times_Y Z}}(\QCoh(X)).
\end{align*}
The above agrees with $\QCoh(X \times_Y Z)$ since $g'$ is quasi-affine.
\end{proof}

We also have the following integral transforms result:

\begin{corollary}
Let $f: X \rightarrow Y$ and $g: Z \rightarrow Y$  be maps of stacks, where $X$ and $Y$ are quasi-compact geometric stacks with quasi-affine diagonal. Assume that $\Ocal_X$ and $\Ocal_Y$ are compact and truncated.  Then there is an equivalence
\[
\QCoh(X \times_Y Z) = \Fun^L_{\QCoh(Y)}(\QCoh(X), \QCoh(Z)).
\]
\end{corollary}
\begin{proof}
By corollary \ref{coro dualizability QCoh} we have that $\QCoh(Y)$ and $\QCoh(X)$ are Frobenius algebras in $\Pr^L_{\St}$, in the sense of \cite{HA} 4.6.5. Combining this with \cite{HA} corollary 4.6.5.14 we deduce that $\QCoh(X)$ is self dual as a $\QCoh(Y)$-module in $\Pr^L$. The desired assertion now follows from corollary \ref{corollary tensor product when compact and truncated}.
\end{proof}

Theorem \ref{theorem relative tensor products} may be used to obtain a computation of the factorization homology of categories of quasicoherent sheaves  up to completion. More generally, we have the following result that deals with tensoring by finite spaces:
 
\begin{corollary}\label{corollary tensor and mapping spaces}
Let $X$ be a quasi-compact geometric stack with quasi-affine diagonal. Let $S$ be a finite space and let $S \otimes \QCoh(X)$ be the tensor (i.e., copower) of $\QCoh(X)$ by $S$ in the category $\CAlg(\Pr^L)$. Equip $S \otimes \QCoh(X)$ with the t-structure whose connective half is generated under colimits, extensions and tensor products by the  union of the images of $\QCoh(X)^\cn$ under the canonical functors $\QCoh(X) \rightarrow S \otimes \QCoh(X)$. Then the canonical functor 
\[
S \otimes \QCoh(X) \rightarrow \QCoh(X^S)
\]
 is t-exact, and exhibits $\QCoh(X^S)$ as the left completion of $S \otimes \QCoh(X)$.
\end{corollary}
\begin{proof}
Let $\mathcal{G}$ be the full subcategory of $\PreStk$ on the quasi-compact geometric stacks with quasi-affine diagonal. By theorem \ref{theorem relative tensor products} the functor $\QCoh^\cn: \mathcal{G}^\op \rightarrow \CAlg(\Groth_\comp)$ preserves finite colimits, and in particular we may identify $\QCoh(X^S)^\cn$ with the tensor of $\QCoh(X)^\cn$ by $S$ in $\CAlg(\Groth_\comp)$. To finish the proof it will suffice to show the following:
\begin{enumerate}[$(\ast)$]
\item The tensor $S \otimes \QCoh(X)^\cn$ exists in $\CAlg(\Groth)$ and is preserved by the inclusion into $\CAlg(\Groth^+)$, where $\Groth^+$ is the symmetric monoidal category of cored presentable stable categories (see \ref{subsection grothstcomp}).
\end{enumerate}
The above holds whenever $S$ is empty or a singleton. We may thus reduce to proving the following:
\begin{enumerate}[$(\ast')$]
\item Let $S_1 \leftarrow S_0 \rightarrow S_2$ be a span of finite spaces with pushout $S$. If $S_0, S_1, S_2$ satisfy $(\ast)$ then $S$ satisfies $(\ast)$. 
\end{enumerate}
Since the inclusion $\Groth \rightarrow \Groth^+$ preserves colimits of diagrams with almost compact transitions (remark \ref{remark colimit along compact}) and the symmetric monoidal structure on $\Groth^+$ is compatible with colimits, it is enough to show that the face maps in 
\[
\Bar_{S_0 \otimes \QCoh(X)^\cn}(S_1 \otimes \QCoh(X)^\cn, S_2 \otimes \QCoh(X)^\cn)_\bullet
\] 
are almost compact (where here the tensors are taken in $\CAlg(\Groth)$). Combining proposition \ref{proposition lex y almost compact completion} and corollary \ref{coro face maps compact} we reduce to proving that the completion of $S_0 \otimes \QCoh(X)^\cn$ is an admissible commutative algebra in $\Groth_\comp$. This completion may be identified with the tensor of $\QCoh(X)^\cn$ by $S_0$ in $\CAlg(\Groth_\comp)$, which as observed above agrees with $\QCoh(X^{S_0})^\cn$. The desired claim now follows from an application of theorem \ref{theorem admissibility qcoh}.
\end{proof}

\begin{corollary}
Let $X$ be a quasi-compact spectral geometric stack with quasi-affine diagonal and let $M$ be a compact manifold. Then there is a canonical t-structure on the factorization homology $\int_M \QCoh(X)$ of the $E_\infty$ presentable stable $\infty$-category $\QCoh(X)$, whose left completion is equivalent to $\QCoh(X^M)$.
\end{corollary}
\begin{proof}
Specialize corollary \ref{corollary tensor and mapping spaces} to the case when $S$ is the homotopy type of $M$.
\end{proof}

%%%%%%%%%%%%%%%%%%%%%%%%%%%%%%%%%%%%%%%%%%%%%%%%%%%%%%%%%%%%%%%%%%%%%%%%
%%%%%%%%%%%%%%%%%%%%%%%%%%%%%%%%%%%%%%%%%%%%%%%%%%%%%%%%%%%%%%%%%%%%%%%%
%%%%%%%%%%%%%%%%%%%%%%%%%%%%%%%%%%%%%%%%%%%%%%%%%%%%%%%%%%%%%%%%%%%%%%%%
%%%%%%%%%%%%%%%%%%%%%%%%%%%%%%%%%%%%%%%%%%%%%%%%%%%%%%%%%%%%%%%%%%%%%%%%
%%%%%%%%%%%%%%%%%%%%%%%%%%%%%%%%%%%%%%%%%%%%%%%%%%%%%%%%%%%%%%%%%%%%%%%%
%%%%%%%%%%%%%%%%%%%%%%%%%%%%%%%%%%%%%%%%%%%%%%%%%%%%%%%%%%%%%%%%%%%%%%%%

\subsection{Tannaka duality}\label{subsection tannaka prestable}

Our next goal is to prove the following:

\begin{theorem}\label{theorem tannaka}
Let $X$ and $Y$ be geometric stacks, and assume that $X$ is quasi-compact and has quasi-affine diagonal. Then the assignment $f \mapsto f^*$ provides an equivalence
\[
\Hom_{\Stk}(Y, X) = \Hom_{\CAlg(\Pr^L)}(\QCoh(X)^\cn, \QCoh(Y)^\cn).
\]
\end{theorem}

Theorem \ref{theorem tannaka} is a strengthening of \cite{SAG} theorem 9.2.0.2, which holds in the case when $Y$ is a quasi-compact quasi-separated algebraic space  and identifies $\Hom_{\Stk}(Y, X)$ with the subspace of $ \Hom_{\CAlg(\Pr^L)}(\QCoh(X)^\cn, \QCoh(Y)^\cn)$ on those colimit preserving symmetric monoidal functors $F: \QCoh(X)^\cn \rightarrow \QCoh(Y)^\cn$ such that $\Sp(F)$ admits a colimit preserving right adjoint which commutes with the action of $\QCoh(X)$.

Below we present two proofs of theorem \ref{theorem tannaka}. In the first one we verify the hypothesis of \cite{SAG} theorem 9.2.0.2 directly by using theorem \ref{theorem admissibility qcoh}. The second one is relatively self contained and makes explicit use of the theory of sheaves of categories and theorem \ref{theorem affineness}.

\begin{proof}[First proof of theorem \ref{theorem tannaka}]
 Both sides are limit preserving in $Y$, so it is enough to prove the proposition in the case when $Y$ is affine. By \cite{SAG} theorem 9.2.0.2, we need to check that if $F: \QCoh(X)^\cn \rightarrow \QCoh(Y)^\cn$ is a colimit preserving symmetric monoidal functor then $\Sp(F)$ admits a colimit preserving right adjoint which commutes with the action of $\QCoh(X) $. We may write $F$ as the composition
 \[
 \QCoh(X)^\cn =  \QCoh(X)^\cn \otimeshat \Sp^\cn \xrightarrow{\id \otimes e} \QCoh(X)^\cn \otimeshat \QCoh(Y)^\cn \xrightarrow{\mu_{\QCoh(Y)^\cn}} \QCoh(Y)^\cn
 \]
 where the first map is induced from the unit $e: \Sp^\cn \rightarrow \QCoh(Y)^\cn$ and the second map is the structure morphism from the action of $\QCoh(X)^\cn$ on $\QCoh(Y)^\cn$ induced from $F$. Since $Y$ is affine we have that $\Sp(e)$ admits a t-exact colimit preserving right adjoint. Furthermore, combining   theorem \ref{theorem admissibility qcoh} and part (1) of theorem \ref{theorem modules over admissible} we see  that $\Sp(\mu_{\QCoh(Y)^\cn})$ admits a colimit preserving right adjoint which is right t-exact up to a shift. It follows that $\Sp(F)$ admits a colimit preserving right adjoint which is right t-exact up to a shift.
 
 It remains to show that the right adjoint to $\Sp(F)$ commutes  with the action of $\QCoh(X)^\cn$. By part (2) of theorem \ref{theorem modules over admissible} the commutative square
 \[
 \begin{tikzcd}[column sep = huge]
 \Sp(\QCoh(X)^\cn \otimeshat \QCoh(X)^\cn ) \arrow{r}{\Sp(\mu_{\QCoh(X)^\cn})} \arrow{d}{\Sp(\id \otimeshat F) }& \Sp(\QCoh(X)^\cn) \arrow{d}{\Sp(F)} \\
 \Sp(\QCoh(X)^\cn \otimeshat \QCoh(Y)^\cn )  \arrow{r}{\Sp(\mu_{\QCoh(Y)^\cn})} & \Sp(\QCoh(Y)^\cn)
 \end{tikzcd}
 \]
 is vertically right adjointable. Using remark \ref{remark adjointability in groth overline} we may regard $\Sp(F)$ as a right adjointable morphism in $\Groth^{\st}$, and consequently $\Sp(\id \otimes F)$ is also right adjointable. Using the $2$-categorical structure on the completion functor $\Groth^{\st} \rightarrow \Groth^{\st}_\comp$ we deduce that the square of categories
 \[
 \begin{tikzcd}
  \Sp(\QCoh(X)^\cn \otimes \QCoh(X)^\cn ) \arrow{d}{\Sp(\id \otimes  F)}  \arrow{r}{ }&  \Sp(\QCoh(X)^\cn \otimeshat \QCoh(X)^\cn ) \arrow{d}{\Sp(\id \otimeshat F) } \\
   \Sp(\QCoh(X)^\cn \otimes  \QCoh(Y)^\cn ) \arrow{r}{ }  &  \Sp(\QCoh(X)^\cn \otimeshat \QCoh(Y)^\cn ) 
 \end{tikzcd}
 \]
 is vertically right adjointable. It now follows that the square of categories
 \[
  \begin{tikzcd}[column sep = huge]
 \Sp(\QCoh(X)^\cn \otimes  \QCoh(X)^\cn ) \arrow{r}{\Sp(\mu_{\QCoh(X)^\cn})} \arrow{d}{\Sp(\id \otimes  F) }& \Sp(\QCoh(X)^\cn) \arrow{d}{\Sp(F)} \\
 \Sp(\QCoh(X)^\cn \otimes  \QCoh(Y)^\cn )  \arrow{r}{\Sp(\mu_{\QCoh(Y)^\cn})} & \Sp(\QCoh(Y)^\cn)
 \end{tikzcd}
 \]
  is vertically right adjointable. The above is equivalent to the square
   \[
  \begin{tikzcd}[column sep = huge]
 \QCoh(X) \otimes \QCoh(X) \arrow{r}{\mu_{\QCoh(X)}} \arrow{d}{\id \otimes \Sp(F) }&  \QCoh(X) \arrow{d}{\Sp(F)} \\
 \QCoh(X) \otimes \QCoh(Y) \arrow{r}{\mu_{\QCoh(Y)}} & \QCoh(Y).
 \end{tikzcd}
 \]
 This fits in a commutative diagram
  \[
  \begin{tikzcd}[column sep = huge]
 \QCoh(X) \times \QCoh(X) \arrow{r}{\boxtimes} \arrow{d}{\id \times \Sp(F)} & \QCoh(X) \otimes \QCoh(X) \arrow{r}{\mu_{\QCoh(X)}} \arrow{d}{\id \otimes \Sp(F) } &  \QCoh(X) \arrow{d}{\Sp(F)} \\
 \QCoh(X) \times \QCoh(Y) \arrow{r}{\boxtimes} & \QCoh(X) \otimes \QCoh(Y) \arrow{r}{\mu_{\QCoh(Y)}} & \QCoh(Y).
 \end{tikzcd}
 \]
 where the left square is vertically right adjointable by virtue of the fact that $\Sp(F)$ admits a colimit preserving right adjoint. It now follows that the outer square in the above diagram is vertically right adjointable, which means that the right adjoint to $\Sp(F)$ commutes strictly with the action of $\QCoh(X)$, as desired.
\end{proof}

We now turn to the second proof of theorem \ref{theorem tannaka}. We need some preliminary lemmas.

\begin{lemma}\label{lemma recognize affine}
Let $X$ be a quasi-compact geometric stack with quasi-affine diagonal. Assume that there is a symmetric monoidal equivalence $\QCoh(X)^\cn = \Mod_R^\cn$ for some connective commutative ring spectrum $R$. Then $X$ is affine.
\end{lemma}
\begin{proof}
The equivalence in the statement is given by pullback along a map of geometric stacks $f: X \rightarrow \Spec(R)$. We claim that $f$ is an equivalence. Let $p: U \rightarrow X$ be a faithfully flat map with $U$ affine. Then $p^* \circ f^* : \QCoh(\Spec(R))^\cn \rightarrow \QCoh(U)^\cn$ is left exact and conservative, and consequently $p \circ f$ is faithfully flat. Since $X$ and $\Spec(R)$ are the geometric realizations of the \v{C}ech nerves of $p$ and $p \circ f$, the lemma will follow if we show that the canonical map $\rho: U \times_X U \rightarrow U \times_{\Spec(R)} U$ is an isomorphism. By theorem \ref{theorem relative tensor products} we see that $\rho^*: \QCoh(U \times_{\Spec(R)} U)^\cn \rightarrow \QCoh(U \times_X U)^\cn$ is an equivalence. Replacing $X$ with $U \times_X U$ we may now reduce to proving that $f$ is an equivalence in the case when $X$ is quasi-affine. Replacing $X$ by $U \times_X U$ once more we may further reduce to showing that $f$ is an equivalence in the case when $X$ is affine, which is clear.
\end{proof}

\begin{lemma}\label{lemma tensor affine and quasi affine}
Let $X$ be a quasi-compact geometric stack with quasi-affine diagonal. Let $p: U \rightarrow X$ be a morphism of geometric stacks. Let $Z$ be an affine scheme and let $F: \QCoh(X)^\cn \rightarrow \QCoh(Z)^\cn$ be a colimit preserving symmetric monoidal functor.
\begin{enumerate}[\normalfont (1)]
\item If $p$ is affine then $\QCoh(U)^\cn \otimeshat_{\QCoh(X)^\cn} \QCoh(Z)^\cn = \QCoh(Z')^\cn$ for some affine scheme $Z'$.
\item If $p$ is quasi-affine then $\QCoh(U)^\cn \otimeshat_{\QCoh(X)^\cn} \QCoh(Z)^\cn = \QCoh(Z')^\cn$ for some quasi-affine scheme $Z'$.
\end{enumerate}
\end{lemma}
\begin{proof}
We first prove (1). Since $p$ is affine we have an equivalence 
\[
\QCoh(U)^\cn = \Mod_{A}(\QCoh(X)^\cn)
\] for some commutative algebra $A$ in $\QCoh(X)^\cn$, and therefore we have 
\[
\QCoh(U)^\cn \otimes_{\QCoh(X)^\cn} \ccal = \Mod_A(\ccal).
\]
The above is a complete Grothendieck prestable category. It follows that we also have 
\[
\QCoh(U)^\cn \otimeshat_{\QCoh(X)^\cn} \ccal = \Mod_A(\ccal)
\]
and consequently the canonical  functor $\ccal \rightarrow \QCoh(U)^\cn \otimeshat_{\QCoh(X)^\cn} \ccal $ has a colimit preserving conservative right adjoint. Since the unit of $\ccal$ is a compact projective colimit generator we see that the same holds for $\QCoh(U)^\cn \otimeshat_{\QCoh(X)^\cn} \ccal $. We therefore have an equivalence 
\[
\QCoh(U)^\cn \otimeshat_{\QCoh(X)^\cn} \ccal = \QCoh(Z')^\cn
\]
 for $Z'$ the spectrum of the algebra of endomorphisms of the unit in $\QCoh(U)^\cn \otimeshat_{\QCoh(X)^\cn} \ccal $.
 
 We now prove (2). Since $p$ is quasi-affine, it factors as a composition $U \rightarrow V \rightarrow X$ where the first map is a quasi-compact open immersion and the second map is affine. By (1), we have that $\QCoh(V) \otimeshat_{\QCoh(X)^\cn} \QCoh(Z)^\cn = \QCoh(W)^\cn$ for some affine scheme $W$. Replacing $X$ by $V$ and $Z$ by $W$ we may now reduce to the case when $p$ is a quasi-compact open immersion. In this case we have an equivalence $\QCoh(V) = \Mod_{A}(\QCoh(X))$ for some  some commutative algebra $A$ in $\QCoh(X)$ which is eventually connective, idempotent, and compact. Furthermore, we have
\[
\QCoh(U) \otimes_{\QCoh(X)} \QCoh(Z) = \Mod_{\Sp(F)(A)}(\QCoh(Z)).
\]
The geometric realization of $\Bar_{\QCoh(X)^\cn}(\QCoh(U)^\cn, \QCoh(Z)^\cn)_\bullet$ inside $\Groth$ consists of  the connective half of a t-structure on  $\Mod_{\Sp(F)(A)}(\QCoh(Z))$. This t-structure is defined by the following properties:
\begin{itemize}
\item The connective half is generated under colimits and extensions by those $\Sp(F)(A)$-modules of the form $\Sp(F)(A) \otimes M$ with $M $ in $\QCoh(Z)^\cn$.
\item A  $\Sp(F)(A)$-module is coconnective if and only if its image in $\QCoh(Z)$ is coconnective.
\end{itemize}

Since $\QCoh(Z)^\cn$ is generated under colimits by the unit, we in fact have that the connective half of $\Mod_{\Sp(F)(A)}(\QCoh(Z))$ is  generated under colimits and extensions by its unit $\Sp(F)(A)$. Consequently, to prove (2) it will be enough to show that we have a symmetric monoidal equivalence  $\Mod_{\Sp(F)(A)}(\QCoh(Z)) = \QCoh(Z')$ for some quasi-affine scheme $Z'$.  Using \cite{SAG} proposition 2.6.0.3 we may reduce to proving that $\Sp(F)(A)$ is eventually connective, idempotent, and compact. The first two properties follow directly from the fact that $A$ satisfies those properties. For the last property it will suffice to show that the functor $\CAlg(\QCoh(X)) \rightarrow \CAlg(\QCoh(Z))$ induced by $\Sp(F)$ admits a filtered colimit preserving right adjoint. This will follow if we show that $\Sp(F): \QCoh(X) \rightarrow \QCoh(Z)$ admits a colimit preserving right adjoint. In other words, we need to prove that $F$ is compact. 

We regard $F$ as a morphism of $\QCoh(X)^\cn$-modules in $\Groth_\comp$, which by theorem \ref{theorem affineness} may be written as $\Gamma^\enh(X, \overline{F})$ for some morphism $\overline{F}$ of commutative algebras in $\twoQCoh^{\pst}_\comp(X)$. By lemma \ref{lemma Gamma y colimits}, to prove that $F$ is compact it will suffice to show that $\overline{F}$ is compact. To do so it is enough to show that for every map $\eta: V \rightarrow X$ with $V$ affine the functor of Grothendieck prestable categories underlying $\eta^*\overline{F}$ is compact. In other words, we have to show that
\[
\id \otimeshat F: \QCoh(V)^\cn = \QCoh(V)^\cn \otimeshat_{\QCoh(X)^\cn} \QCoh(X)^\cn \rightarrow  \QCoh(V)^\cn \otimeshat_{\QCoh(X)^\cn} \QCoh(Z)^\cn
\]
is compact. This amounts to proving that the unit in $ \QCoh(V)^\cn \otimeshat_{\QCoh(X)^\cn} \QCoh(Z)^\cn$ is compact. As above, we may identify  $ \QCoh(V)^\cn \otimeshat_{\QCoh(X)^\cn} \QCoh(Z)^\cn$ with the completion of the connective half of a t-structure on $\Mod_{\Sp(F)(B)}(\QCoh(Z))$ for some eventually connective commutative algebra $B$ in $\QCoh(X)$. The forgetful functor $G: \Mod_{\Sp(F)(B)}(\QCoh(Z)) \rightarrow \QCoh(Z)$ has the following properties:
\begin{itemize}
\item $G$ creates limits.
\item $G$ is t-exact up to shifts.
\item If $G(M)$ is connective, then $M$ is connective (since in this case the Bar resolution of $M$ is levelwise connective).
\end{itemize}
Combined with the fact that the t-structure on $\QCoh(Z)$ is left complete, the above properties imply that the $t$-structure on $\Mod_{\Sp(F)(B)}(\QCoh(Z))$ is also left complete. We may thus reduce to showing that the unit of $\Mod_{\Sp(F)(B)}(\QCoh(Z)) $ is compact, which follows from the fact that the unit in $\QCoh(Z)$ is compact.
\end{proof}
 
\begin{proof} [Second proof of theorem \ref{theorem tannaka}]
Both sides are limit preserving in $Y$, so it is enough to prove the theorem in the case when $Y$ is affine. Consider the commutative square of categories
\[
\begin{tikzcd}
(\Aff_{/X})^\op \arrow{r}{\QCoh^\cn}\arrow{d}{} &  \CAlg(\Groth_\comp)_{\QCoh(X)^\cn /} \arrow{d}{} \\
\Aff^\op \arrow{r}{\QCoh^\cn} & \CAlg(\Groth_\comp).
\end{tikzcd}
\]
 The vertical arrows in the above square are left fibrations, and the morphism in the statement is obtained from the above square by passing to fibers over $Y$ and $\QCoh(Y)^\cn$. We may thus reduce to showing that the above is a pullback square. By \cite{HA} corollary 4.8.5.21 we have that the bottom horizontal arrow in the above square is fully faithful. Let
 \[
 \CAlg(\Groth_\comp)^{\text{aff}}_{\QCoh(X)^\cn /} = \Aff^{\op} \times_{ \CAlg(\Groth_\comp)} \CAlg(\Groth_\comp)_{\QCoh(X)^\cn /}.
 \]
In other words, this is the full subcategory of $\CAlg(\Groth_\comp)_{\QCoh(X)^\cn /}$  consisting of those objects $F: \QCoh(X)^\cn \rightarrow \ccal$ such that $\ccal = \QCoh(Z)^\cn$ for some affine scheme $Z$. Applying theorem \ref{theorem affineness} we may identify $ \CAlg(\Groth_\comp)^{\text{aff}}_{\QCoh(X)^\cn /}$ with a full subcategory $\CAlg(\twoQCoh^\pst_\comp(X))^{\text{aff}}$ of $\CAlg(\twoQCoh^\pst_\comp(X))$. Let 
\[
\Xi^{\text{aff}}_X : (\Aff_{/X})^\op \rightarrow \CAlg(\twoQCoh^\pst_\comp(X))^{\text{aff}}
\]
be the induced functor. Unwinding the definitions, this is the functor sending each map $f: Z \rightarrow X$ to $f_* \QCoh^\cn_Z$. Our goal is to show that $\Xi^{\text{aff}}_X$ is an equivalence.

Write $X$ as the colimit of a diagram of affine schemes $X_\alpha$. Assume for the moment that the diagonal of $X$ is affine. For each $\alpha$ define 
\[
\Xi^{\text{aff}}_{X_\alpha} : (\Aff_{/{X_\alpha}})^\op \rightarrow \CAlg(\twoQCoh^\pst_\comp(X_\alpha))^{\text{aff}}
\]
as above. Combining lemma \ref{lemma tensor affine and quasi affine} and proposition \ref{proposition pushforward functoriality} we obtain a commutative square
\[
\begin{tikzcd}
\lim (\Aff_{/X_\alpha})^\op \arrow{d}{\lim \Xi^{\text{aff}}_{X_\alpha}} & \arrow{l}{} (\Aff_{/X})^\op \arrow{d}{\Xi^{\text{aff}}_X} \\
\lim  \CAlg(\twoQCoh^\pst_\comp(X_\alpha))^{\text{aff}} & \arrow{l}{} \CAlg(\twoQCoh^\pst_\comp(X)) ^{\text{aff}}.
\end{tikzcd}
\]
Here the horizontal arrows are fully faithful. Furthermore, an application of \cite{HA} corollary 4.8.5.21 shows that $\Xi^{\text{aff}}_{X_\alpha}$ is an equivalence  for every $\alpha$, thus the left vertical arrow is an equivalence. It follows that $\Xi^{\text{aff}}_{X}$ is fully faithful. Assume now given an object $  \ccal$ in $\CAlg(\twoQCoh^\pst_\comp(X))^{\text{aff}}$. Then for each $\alpha$ the pullback $\ccal_\alpha$ of $\ccal$ to  $\CAlg(\twoQCoh^\pst_\comp(X))^{\text{aff}}$ is of the form $(f_\alpha)_* \QCoh^\cn_{Z_\alpha}$ for some map of affine schemes $f_\alpha: Z_\alpha \rightarrow X_\alpha$. The maps $f_\alpha$ assemble together into an affine schematic map $f: Z \rightarrow X$ with the property that $f_* \QCoh^\cn_Z = \ccal$. We now deduce that $\ccal$ belongs to the image of $\Xi^{\text{aff}}_X $ by an application of lemma \ref{lemma recognize affine}.

This finishes the proof in the case when the diagonal of $X$ is affine. We now address the general case.  For each $\alpha$ let $\CAlg(\Groth_\comp)^{\text{qaff}}_{\QCoh(X_\alpha)^\cn /}$ be the full subcategory of $\CAlg(\Groth_\comp)_{\QCoh(X_\alpha)^\cn /}$ on those objects $F: X_\alpha \rightarrow \ccal$ such that $\ccal = \QCoh(Z)^\cn$ for some quasi-affine scheme. We denote by $\CAlg(\twoQCoh^\pst_\comp(X_\alpha))^{\text{qaff}}$ the corresponding full subcategory of $\CAlg(\twoQCoh^\pst_\comp(X_\alpha))$. Let $\QAff$ be the category of quasi-affine schemes and denote by 
\[
\Xi^{\text{qaff}}_{X_\alpha} : (\QAff_{/{X_\alpha}})^\op \rightarrow \CAlg(\twoQCoh^\pst_\comp(X_\alpha))^{\text{qaff}}
\]
the functor induced by $\QCoh^\cn$. Combining lemma \ref{lemma tensor affine and quasi affine} and proposition \ref{proposition pushforward functoriality} we obtain a commutative square
\[
\begin{tikzcd}
\lim (\QAff_{/X_\alpha})^\op \arrow{d}{\lim \Xi^{\text{qaff}}_{X_\alpha}} & \arrow{l}{} (\Aff_{/X})^\op \arrow{d}{\Xi^{\text{aff}}_X} \\
\lim  \CAlg(\twoQCoh^\pst_\comp(X_\alpha))^{\text{aff}} & \arrow{l}{} \CAlg(\twoQCoh^\pst_\comp(X)) ^{\text{aff}}.
\end{tikzcd}
\]
Here the horizontal arrows are fully faithful. Furthermore, an application of the affine diagonal case of the theorem shows that $\Xi^{\text{qaff}}_{X_\alpha}$ is an equivalence  for every $\alpha$, thus the left vertical arrow is an equivalence.  It follows that $\Xi^{\text{ aff}}_{X}$ is fully faithful. Assume now given an object $\ccal$ in $\CAlg(\twoQCoh^\pst_\comp(X))^{\text{aff}}$. Then for each $\alpha$ the pullback $\ccal_\alpha$ of $\ccal$ to  $\CAlg(\twoQCoh^\pst_\comp(X))^{\text{qaff}}$ is of the form $(f_\alpha)_* \QCoh^\cn_{Z_\alpha}$ for some map of quasi-affine schemes $f_\alpha: Z_\alpha \rightarrow X_\alpha$.  The maps $f_\alpha$ assemble together into a quasi-affine map $f: Z \rightarrow X$ with the property that $f_* \QCoh^\cn_Z = \ccal$. We now deduce that $\ccal$ belongs to the image of $\Xi^{\text{aff}}_X$ by an application of lemma \ref{lemma recognize affine}.
\end{proof}

%%%%%%%%%%%%%%%%%%%%%%%%%%%%%%%%%%%%%%%%%%%%%%%%%%%%%%%%%%%%%%%%%%%%%%%%
%%%%%%%%%%%%%%%%%%%%%%%%%%%%%%%%%%%%%%%%%%%%%%%%%%%%%%%%%%%%%%%%%%%%%%%%
%%%%%%%%%%%%%%%%%%%%%%%%%%%%%%%%%%%%%%%%%%%%%%%%%%%%%%%%%%%%%%%%%%%%%%%%
%%%%%%%%%%%%%%%%%%%%%%%%%%%%%%%%%%%%%%%%%%%%%%%%%%%%%%%%%%%%%%%%%%%%%%%%
%%%%%%%%%%%%%%%%%%%%%%%%%%%%%%%%%%%%%%%%%%%%%%%%%%%%%%%%%%%%%%%%%%%%%%%%
%%%%%%%%%%%%%%%%%%%%%%%%%%%%%%%%%%%%%%%%%%%%%%%%%%%%%%%%%%%%%%%%%%%%%%%%

\ifx\inmain\undefined
\bibliographystyle{myamsalpha}
\bibliography{References}
\fi

%% file: Classical.tex
%%%%%%%%%%%%%%%%%%%%%%%%%%%%%%%%%%%%%%%%%%%%%%%%%%%%%%%%%%%%%%%%%%%%%%%%
%%%%%%%%%%%%%%%%%%%%%%%%%%%%%%%%%%%%%%%%%%%%%%%%%%%%%%%%%%%%%%%%%%%%%%%%
%%%%%%%%%%%%%%%%%%%%%%%%%%%%%%%%%%%%%%%%%%%%%%%%%%%%%%%%%%%%%%%%%%%%%%%%
%%%%%%%%%%%%%%%%%%%%%%%%%%%%%%%%%%%%%%%%%%%%%%%%%%%%%%%%%%%%%%%%%%%%%%%%
%%%%%%%%%%%%%%%%%%%%%%%%%%%%%%%%%%%%%%%%%%%%%%%%%%%%%%%%%%%%%%%%%%%%%%%%
%%%%%%%%%%%%%%%%%%%%%%%%%%%%%%%%%%%%%%%%%%%%%%%%%%%%%%%%%%%%%%%%%%%%%%%%

\section{Abelian variants}

The goal of this section is to provide variants of the results from sections \ref{section prestable}, \ref{section affineness} and \ref{section consequences} where the role of complete Grothendieck prestable categories is played by Grothendieck abelian categories. We begin in \ref{subsection Grothabelian} with a general discussion of the category $\Groth_1$ of Grothendieck abelian categories and colimit preserving functors, with a focus on the behavior of limits, colimits and tensor products. The discussion here parallels that of the prestable case from \ref{subsection groth}, with the difference that instead of (almost) compact functors we focus on strongly compact functors (that is, those functors which admit colimit preserving right adjoints).

In \ref{subsection relative tensor abelian} we study the formation of relative tensor products in $\Groth_1$. As in the prestable situation, we only consider relative tensor products over algebras which are almost rigid in the sense of \ref{subsection almost rigid}. By theorem \ref{theorem admissibility qcoh}, this applies in particular to the symmetric monoidal category $\QCoh(X)^\heartsuit$ whenever $X$ is a quasi-compact geometric stack with affine diagonal.

In \ref{subsection sheaves abelian} we give an overview of the theory of sheaves of Grothendieck abelian categories. Finally, in \ref{subsection affineness abelian} we give a proof of our affineness theorem in the abelian setting (theorem \ref{theorem affineness abelian intro}) and use it to establish our Tannaka duality result for classical geometric stacks (part (2) of theorem \ref{theorem intro tannaka}).

%%%%%%%%%%%%%%%%%%%%%%%%%%%%%%%%%%%%%%%%%%%%%%%%%%%%%%%%%%%%%%%%%%%%%%%%
%%%%%%%%%%%%%%%%%%%%%%%%%%%%%%%%%%%%%%%%%%%%%%%%%%%%%%%%%%%%%%%%%%%%%%%%
%%%%%%%%%%%%%%%%%%%%%%%%%%%%%%%%%%%%%%%%%%%%%%%%%%%%%%%%%%%%%%%%%%%%%%%%
%%%%%%%%%%%%%%%%%%%%%%%%%%%%%%%%%%%%%%%%%%%%%%%%%%%%%%%%%%%%%%%%%%%%%%%%
%%%%%%%%%%%%%%%%%%%%%%%%%%%%%%%%%%%%%%%%%%%%%%%%%%%%%%%%%%%%%%%%%%%%%%%%
%%%%%%%%%%%%%%%%%%%%%%%%%%%%%%%%%%%%%%%%%%%%%%%%%%%%%%%%%%%%%%%%%%%%%%%%

\subsection{The category of Grothendieck abelian categories}\label{subsection Grothabelian}

We begin with some recollections on the theory of Grothendieck abelian categories.

\begin{definition}\label{def grothendieck abelian}
A Grothendieck abelian category is an abelian category $\Ccal$ which is presentable and such that filtered colimits in $\Ccal$ are exact. We denote by $\Groth_1$ the category of Grothendieck abelian categories and colimit preserving functors.
\end{definition}

\begin{example}
Let $A$ be a ring. Then the category $\LMod_A(\Ab)$ is Grothendieck abelian.
\end{example}

\begin{example}
Let $X$ be a geometric stack. Then the category $\QCoh(X)^\heartsuit$ of quasicoherent sheaves of abelian groups on $X$ is Grothendieck abelian.
\end{example}

Two classes of morphisms in $\Groth_1$ will be of importance to us: left exact morphisms (that is, functors which preserve finite limits) and strongly compact morphisms (that is, functors whose right adjoint preserves small colimits).

\begin{notation}
We denote by $\Groth^\lex_1$ (resp. $\Groth_1^{\sc}$) the wide subcategory of $\Groth_1$ on the left exact morphisms (resp. the strongly compact morphisms). 
\end{notation}

\begin{proposition}\label{prop limits and colimits abelian} \hfill
\begin{enumerate}[\normalfont (1)]
\item The category $\Groth_1^\lex$ admits small limits, and these are preserved by the inclusions $\Groth_1^\lex \rightarrow \Groth_1 \rightarrow \cathat$.
\item  The category $\Groth_1^\sc$ admits small colimits, and these are preserved by the inclusions $\Groth_1^\sc \rightarrow \Groth_1 \rightarrow \Pr^L$.
\end{enumerate}
\end{proposition}
\begin{proof}
Item (1) is \cite{SAG} proposition C.5.4.21. We now prove (2). Let $(\Pr^L)^\radj$ be the wide subcategory of $\Pr^L$ on the functors which admit a colimit preserving right adjoint. Then it follows from \cite{HTT} theorem 5.5.3.18 that $(\Pr^L)^\radj$ admits small colimits, which are preserved by the inclusion into $\Pr^L$. It now suffices to show that if $\ccal_\alpha$ is a diagram in $\Groth_1$ with strongly compact transitions then its colimit in $\Pr^L$ is Grothendieck abelian. This colimit is turned into a limit in $\cathat$ after passing to right adjoints, so the desired claim follows from (1).
\end{proof}

We have the following commutation property of limits and colimits in $\Groth_1$, which parallels proposition \ref{proposition commute limits and colimits}.

\begin{proposition}\label{proposition commute limits and colimits abelian}
Let $\Ical, \Jcal$ be small categories and let $F: \Ical \times \Jcal \rightarrow \Groth_1$ be a diagram. Assume the following:
\begin{enumerate}[\normalfont (a)]
\item For every $\alpha \rightarrow \alpha'$ in $\Ical$ and every $\beta$ in $\Jcal$ the map $F(\alpha, \beta) \rightarrow F(\alpha', \beta)$ is left exact.
\item For every $\alpha$ in $\Ical$ and every $\beta \rightarrow \beta'$ in $\Jcal$ the map $F(\alpha, \beta) \rightarrow F(\alpha, \beta')$ is strongly compact.
\item For every $\alpha \rightarrow \alpha'$ in $\Ical$ and every $\beta \rightarrow \beta'$ in $\Jcal$ the commutative square
\[
\begin{tikzcd}
F(\alpha, \beta) \arrow{r}{} \arrow{d}{} & F(\alpha, \beta') \arrow{d}{} \\
F(\alpha', \beta) \arrow{r}{} & F(\alpha', \beta')
\end{tikzcd}
\]
is horizontally right adjointable.
\end{enumerate}
 For each $\alpha$ in $\Ical$ let $F(\alpha, \ast) =\colim_{\Jcal} F(\alpha, \beta)$, and for each $\beta$ in $\Jcal$ let $F(\ast, \beta) = \lim_{\Ical} F(\alpha, \beta)$. Then:
\begin{enumerate}[\normalfont (1)]
\item For each $\alpha \rightarrow \alpha'$ in $\Ical$ the map $F(\alpha, \ast) \rightarrow F(\alpha', \ast)$ is left exact.
\item For each $\beta \rightarrow \beta'$ in $\Jcal$ the map $F(\ast, \beta) \rightarrow F(\ast, \beta')$ is strongly compact.
\item The canonical functor $\colim_{\Jcal} F(\ast, \beta) \rightarrow \lim_{\Ical} F(\alpha, \ast)$ is an equivalence.
\end{enumerate}
\end{proposition}
\begin{proof}
Assumption (c) implies that we may obtain a functor $\Ical \times \Jcal^\op \rightarrow \Groth_1$ by passage to right adjoints of $F$ in the second variable. For each $\alpha \rightarrow \alpha'$ in $\Ical$ the functor $F(\alpha, \ast) \rightarrow F(\alpha', \ast)$ is the limit of the functors $F(\alpha, \beta) \rightarrow F(\alpha', \beta)$. Each of these is left exact by (a), so we deduce that (1) holds. Furthermore, for each $\beta \rightarrow \beta'$ in $\Jcal$ the functor $F(\ast, \beta) \rightarrow F(\ast, \beta')$ has a right adjoint given by the limit of the right adjoints to the functors $F(\alpha, \beta) \rightarrow F(\alpha, \beta')$, so we deduce that (2) holds.  It remains to address (3). Let $F(\ast, \ast) = \colim_{\Jcal} F(\ast, \beta)$ and note that all the values of $F$ fit (after passage to right adjoints in the second variable) into a functor $\Ical^\lhd \times (\Jcal^\rhd)^\op \rightarrow \Groth_1$. We now have
\[
F(\ast, \ast) = \lim_{\Jcal^\op} F(\ast, \beta) = \lim_{\Jcal^\op} \lim_{\Ical} F(\alpha, \beta) = \lim_{\Ical} \lim_{\Jcal^\op} F(\alpha, \beta) = \lim_\Ical F(\alpha, \ast)
\]
as desired.
\end{proof}

There is a good theory of tensor products of Grothendieck abelian categories:

\begin{theorem}[\cite{SAG} theorem C.5.4.16, \cite{Tensor} theorem 5.4]\label{teo tensor product abelian}
Let $\Ccal, \Dcal$ be Grothendieck abelian categories. Then their tensor product $\Ccal \otimes \Dcal$ (formed in $\Pr^L$) is Grothendieck abelian.
\end{theorem}

It follows from theorem \ref{teo tensor product abelian} that $\Groth_1$ inherits a symmetric monoidal structure from the  category $\Mod_{\Ab}(\Pr^L)$ of classical additive categories. This restricts to a symmetric monoidal structure on $\Groth^\lex_1$ and $\Groth^\sc_1$:

\begin{proposition}\label{prop tensor y sc y lex abelian}
Let $f: \ccal \rightarrow \ccal'$ be a left exact (resp. strongly compact) morphism in $\Groth_1$. Then for every Grothendieck abelian category $\dcal$ the morphism $f \otimes \id: \ccal \otimes \dcal \rightarrow \ccal' \otimes \dcal$ is left exact (resp. strongly compact).
\end{proposition}
\begin{proof}
The case of flatness is \cite{ExponentiableGroth} theorem 3.11 (see also \cite{Fully} proposition 2.2.23 for a relative version). The case of strong compactness follows from the fact that the symmetric monoidal structure on $\Pr^L$ is compatible with its $2$-categorical structure.  
\end{proof}

\begin{proposition}  \label{proposition tensor product compatible with colimits abelian}
Let $\ccal$ be a Grothendieck abelian category. Then the composite functor
\[
\Groth^{\sc}_1 \hookrightarrow \Groth_1 \xrightarrow{- \otimes \ccal} \Groth_1
\]
preserves small colimits.
\end{proposition}
\begin{proof}
Follows from proposition \ref{proposition commute limits and colimits} using the fact that the symmetric monoidal structure on $\Pr^L$ is compatible with small colimits.
\end{proof}

%%%%%%%%%%%%%%%%%%%%%%%%%%%%%%%%%%%%%%%%%%%%%%%%%%%%%%%%%%%%%%%%%%%%%%%%
%%%%%%%%%%%%%%%%%%%%%%%%%%%%%%%%%%%%%%%%%%%%%%%%%%%%%%%%%%%%%%%%%%%%%%%%
%%%%%%%%%%%%%%%%%%%%%%%%%%%%%%%%%%%%%%%%%%%%%%%%%%%%%%%%%%%%%%%%%%%%%%%%
%%%%%%%%%%%%%%%%%%%%%%%%%%%%%%%%%%%%%%%%%%%%%%%%%%%%%%%%%%%%%%%%%%%%%%%%
%%%%%%%%%%%%%%%%%%%%%%%%%%%%%%%%%%%%%%%%%%%%%%%%%%%%%%%%%%%%%%%%%%%%%%%%
%%%%%%%%%%%%%%%%%%%%%%%%%%%%%%%%%%%%%%%%%%%%%%%%%%%%%%%%%%%%%%%%%%%%%%%%

\subsection{Relative tensor products}\label{subsection relative tensor abelian}

We now discuss the formation of relative tensor products in $\Groth_1$.

\begin{definition}
We say that an algebra $\Acal$ in $\Groth_1$ is almost rigid if it is almost rigid in the sense of definition \ref{definition almost rigid} when considered as an algebra in the canonical $2$-categorical enhancement of $\Groth_1$.
\end{definition}

\begin{remark}\label{remark almost rigidity in groth1}
Unwinding the definitions, an algebra $\acal$ in $\Groth_1$ is almost rigid if and only if the following conditions are satisfied:
\begin{enumerate}[\normalfont (a)]
\item The functor $\mu: \Acal \otimes \acal \rightarrow \acal$ admits a colimit preserving right adjoint.
\item The commutative square of categories
\[
\begin{tikzcd}[column sep = large]
 \Acal \otimes \acal \otimes \acal  \arrow{d}{ \mu \otimes \id } \arrow{r}{ \id \otimes \mu } &  \acal \otimes \acal   \arrow{d}{ \mu } \\
 \acal \otimes \acal   \arrow{r}{ \mu } & \acal 
\end{tikzcd}
\]
is both horizontally and vertically right adjointable.
\end{enumerate}
\end{remark}

\begin{remark}
The inclusion $\Groth_1 \rightarrow \Pr^L$ may be enhanced to a fully faithful non-unital symmetric monoidal functor of symmetric monoidal $2$-categories. It follows from this that an algebra $\acal$ in $\Groth_1$ is almost rigid if and only if it is rigid as an algebra in $\Pr^L$. Combining with theorem \ref{theorem admissibility qcoh} we deduce that $\QCoh(X)^\heartsuit$ is an almost rigid algebra in $\Groth_1$ for every quasi-compact geometric stack $X$ with affine diagonal.
\end{remark}

As a consequence of theorem \ref{theo modules over almost rigid algebras}  we obtain the following:

\begin{theorem}\label{theorem modules over almost rigid abelian}
Let $\acal$ be an almost rigid algebra in $\Groth_1$ and let $\Mcal$ be a left $\acal$-module in $\Groth_1$. Then:
\begin{enumerate}[\normalfont (1)]
\item The action map $\mu_{\Mcal}: \acal \otimes \Mcal \rightarrow \Mcal$ is strongly compact
\item Let $f: \Mcal \rightarrow \Ncal$ be a morphism of left $\acal$-modules in $\Groth_1$. Assume that $f$ is strongly compact. Then the commutative square 
\[
\begin{tikzcd}[column sep = large]
 \acal \otimes  \Mcal  \arrow{r}{\mu_{\Mcal}} \arrow{d}{ \id \otimes f } & \Mcal \arrow{d}{f} \\
 \Acal \otimes \Ncal \arrow{r}{ \mu_{\Ncal}}  & \Ncal
\end{tikzcd}
\]
is vertically right adjointable.
\item Let $f: \Mcal \rightarrow \Ncal$ be a morphism of left $\acal$-modules in $\Groth_1$.  Then the commutative square 
\[
\begin{tikzcd}[column sep = large]
 \acal \otimes  \Mcal  \arrow{r}{\mu_{\Mcal}} \arrow{d}{\id \otimes f} & \Mcal \arrow{d}{f} \\
 \Acal \otimes \Ncal \arrow{r}{ \mu_{\Ncal}}  & \Ncal
\end{tikzcd}
\]
is horizontally right adjointable.
\end{enumerate}
\end{theorem}

\begin{notation}
Let $\acal$ be a algebra in $\Groth_1$, and let $\Mcal, \Ncal$ be a pair of a right and a left $\acal$-module in $\Groth_1$. We will denote by $\Bar_\acal(\Mcal, \Ncal)_\bullet$ the corresponding Bar construction.  
\end{notation}

\begin{corollary}\label{coro face maps right adjointable}
Let $\acal$ be an almost rigid algebra in $\Groth_1$, and let $\Mcal, \Ncal$ be a pair of a right and a left $\acal$-module in $\Groth_1$. Then every face map of $\Bar_{\acal}(\Mcal, \Ncal)_\bullet$ is strongly compact.
\end{corollary}
\begin{proof}
Follows from theorem \ref{theorem modules over almost rigid abelian} (and its opposite), since every face map of  $\Bar_{\acal}(\Mcal, \Ncal)_\bullet$ is  the action map for some left or right $\acal$-module in $\Groth_1$.
\end{proof}

\begin{corollary}\label{corollary tensor product over almost rigid abelian}
Let $\acal$ be an almost rigid commutative algebra in $\Groth_1$. Then:
\begin{enumerate}[\normalfont (1)]
\item  $\Mod_{\acal}(\Groth_1)$ has a symmetric monoidal structure where the tensor product of a pair of modules $\Mcal$ and $\Ncal$ is given by the geometric realization of $\Bar_{\acal}(\Mcal, \Ncal)_\bullet$.
\item Let $f: \acal \rightarrow \acal'$ be a morphism of  commutative algebras in $\Groth_1$, and assume that $\acal'$ is almost rigid. Then there is a symmetric monoidal extension of scalars functor $\Mod_{\acal}(\Groth_1) \rightarrow \Mod_{\acal'}(\Groth_1)$ which sends each $\acal$-module $\Mcal$ to the geometric realization of $\Bar_{\acal}(\Mcal, \Ncal)_\bullet$.
\end{enumerate}
\end{corollary}
\begin{proof}
Follows from corollary \ref{coro face maps right adjointable} in light of propositions \ref{prop limits and colimits abelian} and \ref{proposition tensor product compatible with colimits abelian}.
\end{proof}

\begin{notation}
Let $\acal$ be an almost rigid commutative  algebra in $\Groth_1$. For each pair of $\acal$-modules $\Mcal, \Ncal$ in $\Groth_1$ we denote by $\Mcal \otimes_{\acal} \Ncal$ their tensor product in $\Mod_{\acal}(\Groth_1)$. We note that this is compatible with the usual notation for relative tensor products in $\Pr^L$ since the inclusion $\Mod_\Acal(\Groth_1) \rightarrow \Mod_\acal(\Pr^L)$ is symmetric monoidal.
\end{notation}

We now discuss how the formation of relative tensor products interacts with various properties of morphisms.

\begin{definition}
Let $\acal$ be an almost rigid algebra in $\Groth_1$ and let $f: \Mcal \rightarrow \Mcal'$ be a morphism of left $\acal$-modules. We say that $f$ is left exact (resp. strongly compact) if the underlying functor of Grothendieck abelian categories is left exact (resp. strongly compact). 
\end{definition}

\begin{proposition}\label{proposition relative tensor strongly compact abelian}
Let $\acal$ be an almost rigid commutative algebra in $\Groth_1$. Let $f: \Mcal \rightarrow \Mcal'$ be a strongly  compact morphism of $\acal$-modules in $\Groth_1$, and let $\Ncal$ be an $\acal$-module in $\Groth_1$. Then the induced map  $f \otimes \id : \Mcal \otimes_{\acal} \Ncal \rightarrow \Mcal' \otimes_\acal \Ncal$ is strongly compact.
\end{proposition}
\begin{proof}
 By proposition \ref{prop tensor y sc y lex abelian}  and corollary \ref{coro face maps right adjointable} we have that 
 \[
 \Bar_{\acal}(f, \Ncal)_{\bullet}: \Bar_{\acal}(\Mcal, \Ncal)_{\bullet} \rightarrow \Bar_{\acal}(\Mcal', \Ncal)_{\bullet}
 \]
is a levelwise strongly compact morphism of simplicial objects of $\Groth_1$ with strongly compact face maps. The fact that the geometric realization of $ \Bar_{\acal}(f, \Ncal)_{\bullet}$ is strongly compact follows from proposition \ref{prop limits and colimits abelian}.
\end{proof}

\begin{proposition}\label{proposition relative tensor left exact abelian}
Let $\acal$ be an almost rigid commutative algebra in $\Groth_1$. Let $f: \Mcal \rightarrow \Mcal'$ be a left exact morphism of $\acal$-modules in $\Groth_1$, and let $\Ncal$ be an $\acal$-module in $\Groth_1$. Then the induced map  $f \otimes  \id : \Mcal \otimes_{\acal} \Ncal \rightarrow \Mcal' \otimeshat_\acal \Ncal$ is left exact.
\end{proposition}
\begin{proof}
 By proposition  \ref{prop tensor y sc y lex abelian}  and corollary  \ref{coro face maps right adjointable}  we have that 
 \[
 \Bar_{\acal}(f, \Ncal)_{\bullet}: \Bar_{\acal}(\Mcal, \Ncal)_{\bullet} \rightarrow \Bar_{\acal}(\Mcal', \Ncal)_{\bullet}
 \]
is a levelwise left exact morphism of simplicial objects of $\Groth_1$ with strongly compact face maps.  To show that $f \otimes  \id$ is left exact it will suffice to show that   $\Bar_{\acal}(f, \Ncal)_{\bullet}$ admits a left exact geometric realization in $\Groth_1$. By proposition \ref{proposition commute limits and colimits abelian} it will be enough to show that for every face map $[n] \rightarrow [n+1]$ the commutative square
 \[
 \begin{tikzcd}
 \Bar_{\acal}(\Mcal, \Ncal)_{n+1} \arrow{r}{} \arrow{d}{} & \Bar_{\acal}(\Mcal, \Ncal)_{n } \arrow{d}{} \\
 \Bar_{\acal}(\Mcal', \Ncal)_{n+1}\arrow{r}{} & \Bar_{\acal}(\Mcal', \Ncal)_{n} 
 \end{tikzcd} 
 \]
 is horizontally right adjointable. The above square has the form
 \[
 \begin{tikzcd}[column sep = large]
 \acal \otimes  \ccal \arrow{r}{  \mu_\ccal } \arrow{d}{ \id \otimes g } &  \ccal  \arrow{d}{ g } \\
 \acal \otimes  \ccal' \arrow{r}{ \mu_{\ccal'} } &  \ccal' 
 \end{tikzcd}
 \]
 where $g: \ccal \rightarrow \ccal'$ is a morphism of $\acal$-modules in $\Groth_1$. The desired claim now follows from part (3) of theorem \ref{theorem modules over almost rigid abelian}.
\end{proof}

We finish this section with two propositions concerning the compatibility of relative tensor products with colimits and limits.

\begin{proposition}\label{proposition colimits in Modabelian}
Let $\acal$ be an almost rigid commutative algebra in $\Groth_1$, and denote by $\Mod_{\acal}(\Groth_1)^{\sc}$ the wide subcategory of $\Mod_{\acal}(\Groth_1)^{\sc}$ on the strongly compact morphisms.
\begin{enumerate}[\normalfont (1)]
\item The category $\Mod_{\acal}(\Groth_1)^{\sc}$ admits small colimits, which are preserved by the inclusion $\Mod_{\acal}(\Groth_1)^{\sc} \rightarrow \Mod_{\acal}(\Groth_1)$ and the forgetful functor to $\Groth_1$.
\item For every object $\ecal$ in $\Mod_{\acal}(\Groth_1)$ the composite functor
\[
\Mod_{\acal}(\Groth_1)^{\sc} \hookrightarrow \Mod_{\acal}(\Groth_1) \xrightarrow{\ecal \otimes_\acal -} \Mod_{\acal}(\Groth_1)
\]
preserves small colimits.
\end{enumerate}
\end{proposition}
\begin{proof}
Item (1) follows from proposition \ref{proposition tensor product compatible with colimits abelian}, by \cite{HA} corollary 4.2.3.5. We now prove item (2). Let $\ccal_\alpha$ be a diagram in $\Mod_{\acal}(\Groth_1)^{\sc}$ with colimit $\ccal$. Applying \cite{HA} corollary 4.2.3.5 once more, we may reduce to showing that  for every $\dcal$ in $\Groth_1$ we have that $\dcal \otimes  (\ecal \otimes_\acal \ccal)$ is the colimit of $ \dcal \otimes  (\ecal \otimes_\acal \ccal_\alpha)$ in $\Groth_1$. Combining corollary \ref{coro face maps right adjointable} with proposition \ref{proposition tensor product compatible with colimits abelian} we reduce to showing that $| \dcal \otimes   \Bar_{\acal}(\ecal, \ccal)_\bullet|$ is the colimit of $|\dcal \otimes \Bar_{\acal}(\ecal, \ccal_\alpha)_\bullet|$ in $\Groth_1$. To do so it is enough to prove that $ \dcal \otimes   \Bar_{\acal}(\ecal, \ccal)_n$ is the colimit of $\dcal \otimes  \Bar_{\acal}(\ecal, \ccal_\alpha)_n$ in $\Groth_1$ for every $n \geq 0$. This follows from proposition  \ref{proposition tensor product compatible with colimits abelian}.
\end{proof}

\begin{proposition}\label{proposition relative tensor y limits abelian}
Let $\acal$ be an  almost rigid commutative algebra in $\Groth_1$. Let $(\Mcal_\alpha)$ be a diagram of $\acal$-modules in $\Groth_1$  with left exact transitions. Let $\Ncal$ be an $\acal$-module in $\Groth_1$, and assume that $\Ncal$ is dualizable as an object of $\Mod_{\Acal}(\Pr^L)$. Then the canonical map $(\lim \Mcal_\alpha) \otimes_{\acal} \Ncal  \rightarrow \lim (\mcal_\alpha \otimes_{\acal} \Ncal)$ is an equivalence.
\end{proposition}
\begin{proof}
Follows  from proposition \ref{prop limits and colimits abelian}, using the fact that the inclusion  $\Mod_{\Acal}(\Groth_1) \rightarrow \Mod_{\Acal}(\Pr^L)$ is symmetric monoidal.
\end{proof}

%%%%%%%%%%%%%%%%%%%%%%%%%%%%%%%%%%%%%%%%%%%%%%%%%%%%%%%%%%%%%%%%%%%%%%%%
%%%%%%%%%%%%%%%%%%%%%%%%%%%%%%%%%%%%%%%%%%%%%%%%%%%%%%%%%%%%%%%%%%%%%%%%
%%%%%%%%%%%%%%%%%%%%%%%%%%%%%%%%%%%%%%%%%%%%%%%%%%%%%%%%%%%%%%%%%%%%%%%%
%%%%%%%%%%%%%%%%%%%%%%%%%%%%%%%%%%%%%%%%%%%%%%%%%%%%%%%%%%%%%%%%%%%%%%%%
%%%%%%%%%%%%%%%%%%%%%%%%%%%%%%%%%%%%%%%%%%%%%%%%%%%%%%%%%%%%%%%%%%%%%%%%
%%%%%%%%%%%%%%%%%%%%%%%%%%%%%%%%%%%%%%%%%%%%%%%%%%%%%%%%%%%%%%%%%%%%%%%%

\subsection{Sheaves of Grothendieck abelian categories}\label{subsection sheaves abelian}

 We now give an overview of the theory of  sheaves of Grothendieck abelian categories, as introduced in \cite{SAG} chapter 10.
 
 \begin{notation}
 For each connective commutative ring spectrum $R$ we let 
 \[
 \Groth_{1,R} = \Mod_{\Mod_R^\heartsuit}(\Groth_1).
  \]
  Objects of $\Groth_{1,R}$ are called $R$-linear Grothendieck abelian categories.  As discussed in \cite{SAG} proposition D.2.2.1, if $R$ is a connective commutative ring spectrum then $\Groth_{1,R}$ is  closed under tensor products inside $\Mod_{\Mod_R^\heartsuit}(\Pr^L)$. In particular, $\Groth_{1,R}$ inherits a symmetric monoidal structure from $\Mod_{\Mod_R^\heartsuit}(\Pr^L)$. We denote by 
  \[
  - \otimes_R - : \Groth_{1,R} \times \Groth_{1,R} \rightarrow \Groth_{1,R}
  \]
  the corresponding tensor product functor. Furthermore for each morphism of connective commutative ring spectra $R \rightarrow S$ we have an extension of scalars functor which will be denoted by $- \otimes_R S : \Groth_{1,R} \rightarrow \Groth_{1,S}$.
  
 The assignment $R \mapsto \Groth_{1,R}$ assembles into a functor $\CAlg^\cn \rightarrow \CAlg(\cathat)$ which satisfies fpqc descent (\cite{SAG} corollary D.6.8.4). Let 
 \[
 \twoQCoh^\ab : \PreStk^\op \rightarrow \CAlg(\cathat)
 \]
be  its right Kan extension along the inclusion $\CAlg^\cn = \Aff^\op \rightarrow \PreStk^\op$. For each prestack $X$ we call $\twoQCoh^\ab(X)$ the category of quasicoherent sheaves of Grothendieck abelian categories on $X$. For each map $f: X \rightarrow Y$ we denote by $f^*: \twoQCoh^\ab(Y) \rightarrow \twoQCoh^\ab(X)$ the induced pullback functor.
 \end{notation}
 
The notion of quasicoherent sheaf of Grothendieck abelian categories globalizes the notion of $R$-linear Grothendieck abelian category. The classes of left exact and strongly compact morphisms also admit a globalization.

\begin{proposition}
Let $R \rightarrow S$ be a morphism of connective commutative ring spectra and let $f: \ccal \rightarrow \dcal$ be a morphism in $\Groth_{1,R}$. If $f$ is left exact (resp. strongly compact) then $f \otimes_R S : \ccal \otimes_R S \rightarrow \dcal \otimes_R S$ is left exact (resp. strongly compact).
\end{proposition}
\begin{proof}
Follows from propositions \ref{proposition relative tensor strongly compact abelian} and \ref{proposition relative tensor left exact abelian}.
\end{proof}
 
 \begin{definition}\label{definition right adjointability global}
 Let $X$ be a prestack and let $f: \ccal \rightarrow \dcal$ be a morphism in $\twoQCoh^\ab(X)$. We say that $f$ is left exact (resp. strongly compact) if for every connective commutative ring spectrum $R$ and every map $\eta: \Spec(R) \rightarrow X$ the functor of Grothendieck abelian categories underlying the map $\eta^*f: \eta^* \ccal \rightarrow \eta^*\dcal$  is left exact (resp. strongly compact).
 \end{definition}

 \begin{notation}
Let $X$ be a prestack. We denote by $\twoQCoh^{\ab, \lex}(X)$ (resp. $\twoQCoh^{\ab, \sc}(X)$) the wide subcategory of $\twoQCoh^\ab(X)$ on the left exact (resp. strongly compact) morphisms.  
 \end{notation}
 
 \begin{proposition}\label{prop tensor left exact y compact global abelian}
 Let $X$ be a prestack. Let $f: \ccal \rightarrow \ccal'$ be a morphism in $\twoQCoh^\ab(X)$, and let $\dcal$ be an object in $\twoQCoh^\ab(X)$. If $f$ is left exact (resp. strongly compact) then  $f \otimes \id : \ccal \otimes \dcal \rightarrow \ccal' \otimes \dcal$ is left exact (resp. strongly compact).
 \end{proposition}
 \begin{proof}
 It suffices to address the case when $X = \Spec(R)$ is an affine scheme, in which case the proposition follows from propositions \ref{proposition relative tensor strongly compact abelian} and \ref{proposition relative tensor left exact abelian}.
 \end{proof}
 
 \begin{proposition}\label{proposition limits y colimits globalized abelian}
  Let $X$ be a prestack.
 \begin{enumerate}[\normalfont (1)]
 \item The category $\twoQCoh^{\ab, \lex}(X)$ admits small limits, which are preserved by the inclusion into $\twoQCoh^\ab(X)$. Furthermore, for every morphism of prestacks $f: X \rightarrow Y$  the pullback functor $f^*: \twoQCoh^{\ab, \lex}(Y) \rightarrow \twoQCoh^{\ab, \lex}(X)$ preserves small limits.
 \item The category $\twoQCoh^{\ab, \sc}$ admits small colimits, which are preserved by the inclusion into $\twoQCoh^\ab(X)$. Furthermore, for every morphism $f: X \rightarrow Y$ the pullback functor $f^*: \twoQCoh^{\ab, \sc}(Y) \rightarrow \twoQCoh^{\ab, \sc}(X)$ preserves small colimits.
 \end{enumerate}
 \end{proposition}
 \begin{proof}
 For each morphism $R \rightarrow S$ of connective commutative ring spectra, we have that $\Mod_S^\heartsuit$ is a self dual object of $\Groth_{1,R}$, and consequently the functor $- \otimes_R S : \Groth_{1,R} \rightarrow \Groth_{1,S}$ preserves all limits and colimits that exist on the source. We may thus reduce to showing that $\Groth_{1,R}$ admits limits of diagrams with left exact transitions, and colimits of diagrams with strongly compact transitions. This follows from a combination of propositions \ref{prop limits and colimits abelian} and \ref{proposition colimits in Modabelian}.
  \end{proof}

\begin{proposition}
Let $X$ be a prestack and let $\ccal$ be an object in $\twoQCoh^{\ab}(X)$. Then the composite functor
\[
\twoQCoh^{\ab, \sc}(X) \hookrightarrow \twoQCoh^{\ab}(X) \xrightarrow{\ccal \otimes - } \twoQCoh^{\ab}(X)
\]
preserves small colimits.
\end{proposition}
\begin{proof}
By proposition \ref{proposition limits y colimits globalized abelian} it suffices to address the case when $X = \Spec(R)$ is an affine scheme. In this case the desired assertion follows from proposition \ref{proposition colimits in Modabelian}.
\end{proof}
 
We finish by recording the following basic result concerning the pushforward functoriality of $\twoQCoh^\ab$ on geometric stacks:

\begin{proposition}\label{proposition pushforward functoriality abelian}
Let $f: X \rightarrow Y$ be a morphism of geometric stacks. Then the functor $f^*: \twoQCoh^\ab(Y) \rightarrow \twoQCoh^\ab(X)$ admits a right adjoint $f_*: \twoQCoh^\ab(X) \rightarrow \twoQCoh^\ab(Y)$. Furthermore, for every cartesian square of geometric stacks
\[
\begin{tikzcd}
X' \arrow{d}{f'} \arrow{r}{g'} & X \arrow{d}{f} \\
Y' \arrow{r}{g} & Y
\end{tikzcd}
\]
the induced commutative square of categories
\[
\begin{tikzcd}
\twoQCoh^\ab(X' )& \twoQCoh^\ab(X ) \arrow{l}[swap]{g'^*} \\
\twoQCoh^\ab(Y' ) \arrow{u}[swap]{f'^*} & \arrow{u}[swap]{f^*} \arrow{l}[swap]{g^*}  \twoQCoh^\ab(Y )
\end{tikzcd}
\]
is vertically right adjointable. 
\end{proposition}
\begin{proof}
Analogous to the proof of proposition \ref{proposition pushforward functoriality}.
\end{proof}

%%%%%%%%%%%%%%%%%%%%%%%%%%%%%%%%%%%%%%%%%%%%%%%%%%%%%%%%%%%%%%%%%%%%%%%%
%%%%%%%%%%%%%%%%%%%%%%%%%%%%%%%%%%%%%%%%%%%%%%%%%%%%%%%%%%%%%%%%%%%%%%%%
%%%%%%%%%%%%%%%%%%%%%%%%%%%%%%%%%%%%%%%%%%%%%%%%%%%%%%%%%%%%%%%%%%%%%%%%
%%%%%%%%%%%%%%%%%%%%%%%%%%%%%%%%%%%%%%%%%%%%%%%%%%%%%%%%%%%%%%%%%%%%%%%%
%%%%%%%%%%%%%%%%%%%%%%%%%%%%%%%%%%%%%%%%%%%%%%%%%%%%%%%%%%%%%%%%%%%%%%%%
%%%%%%%%%%%%%%%%%%%%%%%%%%%%%%%%%%%%%%%%%%%%%%%%%%%%%%%%%%%%%%%%%%%%%%%%

\subsection{\texorpdfstring{$\twoQCoh^{\Ab}$}{2QCohAb}-affineness and Tannaka duality}\label{subsection affineness abelian}

Our next goal is to establish a variant of theorem \ref{theorem affineness} that applies to quasicoherent sheaves of Grothendieck abelian categories.

\begin{construction}\label{construction Phi abelian} 
Let $\Gcal$ be the full subcategory of $\PreStk$ on the quasi-compact geometric stacks with affine diagonal. It follows from  corollary \ref{corollary tensor product over almost rigid abelian} and theorem \ref{theorem admissibility qcoh} that the assignment $X \mapsto \Mod_{\QCoh(X)^\heartsuit}(\Groth_1)$ gives rise to a functor 
\[
\Mod_{\QCoh(-)^\heartsuit}(\Groth_1): \Gcal^\op \rightarrow \CAlg(\cathat).
\] 
The right Kan extension of $\Mod_{\QCoh(-)^\heartsuit}(\Groth_1)|_{\Aff^\op}$ along the inclusion $\Aff^\op \rightarrow \Gcal^\op$ recovers the restriction to $\mathcal{G}^\op$ of the functor $\twoQCoh^\ab$. It follows that for each object $X$ of $\Gcal$ we have a symmetric monoidal functor
\[
\Phi_X : \Mod_{\QCoh(X)^\heartsuit}(\Groth_1) \rightarrow \twoQCoh^\ab(X).
\]
\end{construction}

\begin{theorem}\label{theorem affineness abelian}
Let $X$ be a quasi-compact geometric stack with affine diagonal. Then the functor 
\[
\Phi_X : \Mod_{\QCoh(X)^\heartsuit}(\Groth_1) \rightarrow \twoQCoh^\ab(X) 
\]
from construction \ref{construction Phi abelian} is an equivalence.
\end{theorem}

Theorem \ref{theorem affineness abelian} admits a reformulation in terms of the functor of global sections.

\begin{notation}\label{notation gamma enh abelian}
Let $X$ be a geometric stack. We denote by
\[
\Gamma(X, -): \twoQCoh^\ab(X) \rightarrow \twoQCoh^\ab(\Spec(\mathbb{S})) = \Groth_1
\]
the functor of pushforward along the projection $X \rightarrow \Spec(\mathbb{S})$ (see proposition \ref{proposition pushforward functoriality abelian}). We equip $\Gamma(X, -)$ with its canonical lax symmetric monoidal structure. Let $\QCoh_X^\heartsuit$ be the unit of $\twoQCoh^\ab(X)$, and observe that we have an equivalence of symmetric monoidal categories $\Gamma(X, \QCoh_X^\heartsuit) = \QCoh(X)^\heartsuit$. We denote by
\[
\Gamma^\enh(X, -): \twoQCoh^\ab(X)  = \Mod_{\QCoh_X^\heartsuit}(\twoQCoh^\ab(X)) \rightarrow \Mod_{\QCoh(X)^\heartsuit}(\Groth_1)
\]
the induced functor.
\end{notation}

\begin{corollary}
Let $X$ be a quasi-compact geometric stack with affine diagonal. Then the functor
\[
\Gamma^\enh(X, -): \twoQCoh^\ab(X) \rightarrow \Mod_{\QCoh(X)^\heartsuit}(\Groth_1)
\]
from notation \ref{notation gamma enh abelian} is an equivalence.
\end{corollary}
\begin{proof}
This follows from theorem \ref{theorem affineness abelian} using the fact that the functor
\[
\Mod_{\QCoh(X)^\heartsuit}(\Groth_1) = \Mod_{\QCoh(X)^\heartsuit}(\Mod_{\QCoh(X)^\heartsuit}(\Groth_1)) \rightarrow \Mod_{\QCoh(X)^\heartsuit}(\Groth_1)
\]
induced from the lax symmetric monoidal structure on the forgetful functor 
\[
\Mod_{\QCoh(X)^\ab}(\Groth_1) \rightarrow \Groth_1,
\]
 is an equivalence.
\end{proof}

Before going into the proof of theorem \ref{theorem affineness abelian}, we indicate how it can be used to deduce a Tannaka duality result for classical geometric stacks.

\begin{theorem}\label{theorem tannaka abelian}
Let $X$ and $Y$ be classical geometric stacks, and assume that $X$ is quasi-compact and has affine diagonal. Then the assignment $f \mapsto f^*$ provides an equivalence
\[
\Hom_{\Stk}(Y, X) = \Hom_{\CAlg(\Pr^L)}(\QCoh(X)^\heartsuit, \QCoh(Y)^\heartsuit).
\]
\end{theorem}

It is possible to give a proof of theorem \ref{theorem tannaka abelian} using theorem \ref{theorem affineness abelian} along the lines of our second proof of theorem \ref{theorem tannaka}. Here we will instead use theorem \ref{theorem affineness abelian} to verify the tameness conditions in Lurie's abelian Tannaka duality:

\begin{proof}[Proof of theorem \ref{theorem tannaka abelian}]
Both sides are limit preserving in $Y$ so we may reduce to the case when $Y$ is the spectrum of a (classical) commutative ring. By \cite{SAG} theorem 9.7.0.1, the assignment $f \mapsto f^*$ is an embedding, so we only need to verify surjectivity. In other words, we have to show that if  $F: \QCoh(X)^\heartsuit \rightarrow \QCoh(Y)^\heartsuit$ is a colimit preserving functor then the following conditions are satisfied:
\begin{enumerate}[\normalfont (i)]
\item $F$ sends flat sheaves to flat sheaves.
\item Let $0 \rightarrow \Fcal' \rightarrow \Fcal \rightarrow \Fcal'' \rightarrow 0$ be an exact sequence in $\QCoh(X)^\heartsuit$ such that $\Fcal''$ is flat. Then the sequence $0 \rightarrow F(\Fcal') \rightarrow F(\Fcal) \rightarrow F(\Fcal'') \rightarrow 0$ is exact.
\end{enumerate}
We regard $\QCoh(Y)^\heartsuit$ as a $\QCoh(X)^\heartsuit$-module in $\Groth_1$ by restriction of scalars.  By theorem \ref{theorem affineness abelian} we have that $\QCoh(Y)^\heartsuit$ belongs to the image of $\Gamma(X,-)$, and in particular it satisfies conditions $(\ast)$ and $(\ast')$ from \cite{SAG} theorem 10.6.2.1. It follows from $(\ast)$ that for every flat sheaf $\Fcal$ in $\QCoh(X)^\heartsuit$ the functor
\[
F(\Fcal) \otimes -: \QCoh(Y)^\heartsuit \rightarrow \QCoh(Y)^\heartsuit
\]
is left exact, so we see that (i) holds. Similarly,  if $0 \rightarrow \Fcal' \rightarrow \Fcal \rightarrow \Fcal'' \rightarrow 0$ is as in (ii) we have that 
\[
0 \rightarrow F(\Fcal') \rightarrow F(\Fcal) \rightarrow F(\Fcal'') \rightarrow 0
\]
is exact by an application of $(\ast')$ with $E = \Ocal_Y$, so we see that (ii) holds.
\end{proof}

We devote the remainder of this section to the proof of theorem \ref{theorem affineness abelian}.

\begin{lemma}
Let $X$ be a quasi-compact geometric stack with  affine diagonal. Then the functor $\Gamma^\enh(X, -)$ from notation \ref{notation gamma enh abelian} is right adjoint to $\Phi_X$.
\end{lemma}
\begin{proof}
Analogous to the proof of lemma \ref{lemma Gammaenh right adjoint}.
\end{proof}

\begin{lemma}\label{lemma Gamma y colimits abelian}
Let $X$ be a geometric stack. 
\begin{enumerate}[\normalfont(1)]
\item The functor  $\Gamma^\enh(X, -)$ sends strongly compact arrows to strongly compact arrows. 
\item Let $\ccal_\beta$ be a diagram in $\twoQCoh^{\ab}(X)$ with strongly compact transitions. Then $\Gamma^\enh(X, -)$ preserves the colimit of $\ccal_\beta$.
\end{enumerate}
\end{lemma}
\begin{proof}
By proposition \ref{proposition colimits in Modabelian}, it is enough to show that the lemma holds with $\Gamma^\enh(X,-)$ replaced with $\Gamma(X, -): \twoQCoh^\ab(X) \rightarrow \Groth_1$. We will show (1) and (2) simultaneously, by proving that $\Gamma(X, -)$ preserves the colimit of $\ccal_\beta$, and that the transitions in $\Gamma(X, \ccal_\beta)$ are strongly compact.

  Write $X$ as the colimit of a diagram of affine schemes $X_\alpha$ with flat transitions, and denote by $p_\alpha: X_\alpha \rightarrow X$ the projections. Then we have $\Gamma(X, \ccal_\beta) = \lim p_\alpha^* \ccal_\beta $. Let $\ccal$ be the colimit of the diagram $\ccal_\beta$ in $\twoQCoh^{\ab}(X)$. By proposition \ref{proposition limits y colimits globalized abelian} for each $\alpha$ we have that $p_\alpha^* \ccal$ is the colimit in $\Groth^{\sc}$ of $p_\alpha^*\ccal_\beta$. Using proposition \ref{proposition commute limits and colimits abelian} we deduce that the transitions in the diagram $\lim p_\alpha^* \ccal_\beta$ are strongly compact, and its colimit in $\Groth$ is given by $\lim p^*_\alpha \ccal =\Gamma(X, \ccal)$, as desired.
\end{proof}

\begin{proof}[Proof of theorem \ref{theorem affineness abelian}]
By \cite{SAG} theorem 10.6.2.1 we have that the functor
\[
\Gamma^\enh(X, -): \twoQCoh^\ab(X) \rightarrow \Mod_{\QCoh(X)^\heartsuit}(\Groth_1)
\]
 is fully faithful. To prove our theorem it will suffice to show that $\Gamma^\enh(X, -)$ is essentially surjective. We begin by showing that for every Grothendieck abelian category $\Dcal$ the free module $  \QCoh(X)^\heartsuit \otimes \dcal$  belongs to the image of $\Gamma^\enh(X, -)$.  To do so it is enough, by \cite{SAG} theorem 10.6.2.1, to verify the following conditions:
\begin{enumerate}[\normalfont (i)]
\item Let $\Gcal$ be a flat sheaf in $\QCoh(X)^\cn$ and let $\Fcal = H_0(\Gcal)$. Then the functor $\Fcal \otimes - : \QCoh(X)^\heartsuit \otimes \dcal \rightarrow \QCoh(X)^\heartsuit \otimes \dcal$ is left exact.
\item Let $0 \rightarrow \Fcal' \rightarrow \Fcal \rightarrow \Fcal'' \rightarrow 0$ be an exact sequence in $\QCoh(X)^\heartsuit$, and assume that $\Fcal'' = H_0(\Gcal)$ for some flat sheaf in $\QCoh(X)^\cn$. Then for every object $E$ in $\QCoh(X)^\heartsuit \otimes \dcal$ the sequence
\[
0 \rightarrow \Fcal' \otimes E \rightarrow \Fcal \otimes E \rightarrow \Fcal'' \otimes E \rightarrow 0
\]
is exact.
\end{enumerate}
Let $\der(\dcal)^\cn$ be the   connective derived category of $\dcal$ and note that passage to $H_0$ provides a $\QCoh(X)^\cn$-linear functor 
\[
\QCoh(X)^\cn \otimes  \der(\dcal)^\cn \rightarrow (\QCoh(X)^\cn \otimes \der(\dcal)^\cn)_{\leq 0} = \QCoh(X)^\heartsuit \otimes \dcal.
\]
 Denote by 
 \[
 - \otimes' -: \QCoh(X)^\cn \times \QCoh(X)^\cn \otimes  \der(\dcal)^\cn  \rightarrow \QCoh(X)^\cn \otimes  \der(\dcal)^\cn 
 \]
  the action map. Let $\Gcal$ be as in (i). Then we have a commutative square
 \[
 \begin{tikzcd}
 \QCoh(X)^\cn \otimes  \der(\dcal)^\cn  \arrow{r}{H_0} \arrow{d}{\Gcal \otimes' -} & \QCoh(X)^\heartsuit \otimes \dcal \arrow{d}{\Fcal \otimes -} \\
  \QCoh(X)^\cn \otimes  \der(\dcal)^\cn  \arrow{r}{H_0}  &  \QCoh(X)^\heartsuit \otimes \dcal .
 \end{tikzcd}
 \]
To prove that $\Fcal \otimes -$ is left exact it suffices to show that the left vertical arrow in the above square is left exact. This follows from the fact that $\Gcal$ is flat, by virtue of proposition  \ref{proposition tensor y compact y lex}.

We now address (ii). We have an exact sequence
\[
\Fcal' \otimes' E \rightarrow \Fcal \otimes' E \rightarrow \Fcal'' \otimes' E
\]
in $\QCoh(X)^\cn \otimes  \der(\dcal)^\cn $, which after passage to homology induces an exact sequence
\[
H_1(\Fcal'' \otimes' E) \rightarrow \Fcal' \otimes E \rightarrow \Fcal \otimes E \rightarrow \Fcal'' \otimes E \rightarrow 0
\]
in  $\QCoh(X)^\heartsuit \otimes  \dcal $. Consequently, to prove (ii) it suffices to show that $H_1(\Fcal'' \otimes' E)  = 0$. This fits into an exact sequence
\[
H_1(\Gcal \otimes' E) \rightarrow H_1(\Fcal'' \otimes' E) \rightarrow H_0(\tau_{\geq 1}(\Gcal) \otimes' E)
\]
where here the last term is $0$ since $\tau_{\geq 1}(\Gcal)$ is $1$-connective. We may thus reduce to showing that $H_1(\Gcal \otimes' E) = 0$, which once again follows from the left exactness of $\Gcal \otimes' -$.

Assume now given an arbitrary object $\ccal$ in $\Mod_{\QCoh(X)^\heartsuit}(\Groth_1)$. We may write $\ccal$ as the geometric realization of the Bar construction $\Bar_{\QCoh(X)^\heartsuit}(\QCoh(X)^\heartsuit, \ccal)_\bullet$. Each entry in this simplicial object is a free $\QCoh(X)^\heartsuit$-module and in particular belongs to the image of $\Gamma^\enh(X, -)$. Using the fully faithfulness of $\Gamma^\enh(X, -)$   we have
 \[
 \Bar_{\QCoh(X)^\heartsuit}(\QCoh(X)^\heartsuit, \ccal)_\bullet = \Gamma^\enh(X, \Phi_X(  \Bar_{\QCoh(X)^\heartsuit}(\QCoh(X)^\heartsuit, \ccal)_\bullet  ) ).
 \]
 We claim that the face maps of $ \Phi_X(  \Bar_{\QCoh(X)^\heartsuit}(\QCoh(X)^\heartsuit, \ccal)_\bullet  )$ are strongly compact. To prove this it suffices to show that for every map $\eta: \Spec(R) \rightarrow X$ the face maps of 
 \[
 \eta^* \Phi_X(  \Bar_{\QCoh(X)^\cn}(\QCoh(X)^\cn, \ccal)_\bullet  )
 \]
 are strongly compact. The simplicial Grothendieck abelian category underlying the above is given by $\Bar_{\QCoh(X)^\heartsuit}(\Mod_R^\heartsuit, \ccal)_\bullet$. The fact that the face maps are strongly compact now follows from  corollary \ref{coro face maps right adjointable}. Applying lemma \ref{lemma Gamma y colimits abelian} we have
 \begin{align*}
 \ccal &= | \Bar_{\QCoh(X)^\heartsuit}(\QCoh(X)^\heartsuit, \ccal)_\bullet| \\
  & = | \Gamma^\enh(X, \Phi_X(  \Bar_{\QCoh(X)^\heartsuit}(\QCoh(X)^\heartsuit, \ccal)_\bullet  ) )| \\
  & =  \Gamma^\enh(X, |\Phi_X(  \Bar_{\QCoh(X)^\heartsuit}(\QCoh(X)^\heartsuit, \ccal)_\bullet  )| )
 \end{align*}
 and in particular $\ccal$ belongs to the image of $\Gamma^\enh(X,-)$, as desired. 
\end{proof}
 
%%%%%%%%%%%%%%%%%%%%%%%%%%%%%%%%%%%%%%%%%%%%%%%%%%%%%%%%%%%%%%%%%%%%%%%%
%%%%%%%%%%%%%%%%%%%%%%%%%%%%%%%%%%%%%%%%%%%%%%%%%%%%%%%%%%%%%%%%%%%%%%%%
%%%%%%%%%%%%%%%%%%%%%%%%%%%%%%%%%%%%%%%%%%%%%%%%%%%%%%%%%%%%%%%%%%%%%%%%
%%%%%%%%%%%%%%%%%%%%%%%%%%%%%%%%%%%%%%%%%%%%%%%%%%%%%%%%%%%%%%%%%%%%%%%%
%%%%%%%%%%%%%%%%%%%%%%%%%%%%%%%%%%%%%%%%%%%%%%%%%%%%%%%%%%%%%%%%%%%%%%%%
%%%%%%%%%%%%%%%%%%%%%%%%%%%%%%%%%%%%%%%%%%%%%%%%%%%%%%%%%%%%%%%%%%%%%%%%

\ifx\inmain\undefined
\bibliographystyle{myamsalpha}
\bibliography{References}
\fi

%% file: LimitsAndColimits.tex
%%%%%%%%%%%%%%%%%%%%%%%%%%%%%%%%%%%%%%%%%%%%%%%%%%%%%%%%%%%%%%%%%%%%%%%%
%%%%%%%%%%%%%%%%%%%%%%%%%%%%%%%%%%%%%%%%%%%%%%%%%%%%%%%%%%%%%%%%%%%%%%%%
%%%%%%%%%%%%%%%%%%%%%%%%%%%%%%%%%%%%%%%%%%%%%%%%%%%%%%%%%%%%%%%%%%%%%%%%
%%%%%%%%%%%%%%%%%%%%%%%%%%%%%%%%%%%%%%%%%%%%%%%%%%%%%%%%%%%%%%%%%%%%%%%%
%%%%%%%%%%%%%%%%%%%%%%%%%%%%%%%%%%%%%%%%%%%%%%%%%%%%%%%%%%%%%%%%%%%%%%%%
%%%%%%%%%%%%%%%%%%%%%%%%%%%%%%%%%%%%%%%%%%%%%%%%%%%%%%%%%%%%%%%%%%%%%%%%
 
\section{Non-existence of limits and colimits in  \texorpdfstring{$\Groth$}{Groth} and  \texorpdfstring{$\Groth_1$}{Groth1}}\label{appendix limits and colimits}

As we discussed in \ref{subsection groth}, the category $\Groth$ of Grothendieck prestable categories and colimit preserving functors has the following properties:
\begin{itemize}
\item $\Groth$ admits limits of diagrams with left exact transitions, and these are preserved by the inclusion into $\Pr^L$.
\item $\Groth$ admits colimits of diagrams with almost compact transitions. Furthermore, the inclusion $\Groth \rightarrow \Pr^L$ preserves colimits of diagrams with strongly compact transitions.
\end{itemize}

Similarly, as we discussed in \ref{subsection Grothabelian} the category $\Groth_1$ of Grothendieck abelian categories and colimit preserving functors has the following features:
\begin{itemize}
\item $\Groth_1$ admits limits of diagrams with left exact transitions, and these are preserved by the inclusion into $\Pr^L$.
\item $\Groth_1$ admits colimits of diagrams with strongly compact transitions, and these are preserved by the inclusion into $\Pr^L$.
\end{itemize}

The goal of this appendix is to present some examples that illustrate the behavior of limits and colimits in $\Groth$ and $\Groth_1$ for diagrams that do not belong to the above classes. In particular, we will see that $\Groth$ and $\Groth_1$ do not admit small limits nor colimits, and that the inclusion $\Groth^\cnorm \rightarrow \Pr^L$ does not preserve small colimits. Along the way, we discuss the non-existence of Hom objects in $\Groth$ and $\Groth_1$, and present examples that exhibit the failure of Zariski descent for additive linear categories.

%%%%%%%%%%%%%%%%%%%%%%%%%%%%%%%%%%%%%%%%%%%%%%%%%%%%%%%%%%%%%%%%%%%%%%%%
%%%%%%%%%%%%%%%%%%%%%%%%%%%%%%%%%%%%%%%%%%%%%%%%%%%%%%%%%%%%%%%%%%%%%%%%
%%%%%%%%%%%%%%%%%%%%%%%%%%%%%%%%%%%%%%%%%%%%%%%%%%%%%%%%%%%%%%%%%%%%%%%%
%%%%%%%%%%%%%%%%%%%%%%%%%%%%%%%%%%%%%%%%%%%%%%%%%%%%%%%%%%%%%%%%%%%%%%%%
%%%%%%%%%%%%%%%%%%%%%%%%%%%%%%%%%%%%%%%%%%%%%%%%%%%%%%%%%%%%%%%%%%%%%%%%
%%%%%%%%%%%%%%%%%%%%%%%%%%%%%%%%%%%%%%%%%%%%%%%%%%%%%%%%%%%%%%%%%%%%%%%%

\subsection{Abelian examples}

We begin by collecting some examples in the abelian setting.

\begin{example}\label{example pushout abelian}
Consider the commutative square of commutative rings
\[
\begin{tikzcd}
\ZZ[x] \arrow{r}{} \arrow{d}{} & \ZZ[x, x^{-1}] \arrow{d}{} \\
\ZZ[x, (x-1)^{-1}] \arrow{r}{} & \ZZ[x, x^{-1}, (x-1)^{-1}].
\end{tikzcd}
\]
We obtain from the above a commutative square of $\ZZ[x]$-linear Grothendieck abelian categories
\[
\begin{tikzcd}
\Mod^\heartsuit_{\ZZ[x]}  & \Mod^\heartsuit_{\ZZ[x, x^{-1}] } \arrow{l}{} \\
\Mod^\heartsuit_{\ZZ[x, (x-1)^{-1}]} \arrow{u}{} &\arrow{l}{} \arrow{u}{} \Mod^\heartsuit_{ \ZZ[x, x^{-1}, (x-1)^{-1}]}
\end{tikzcd}
\]
where the maps are given by restriction of scalars. The above becomes a pushout square upon extension of scalars along the maps $\ZZ[x] \rightarrow \ZZ[x, x^{-1}]$ and $\ZZ[x] \rightarrow \ZZ[x, (x-1)^{-1}]$, and it is therefore a pushout square in $\Groth_{1, \ZZ[x]} $ by \cite{SAG} corollary D.6.8.4. This is however not a pushout in the category of $\Mod^\heartsuit_{\ZZ[x]}$-modules in $\Pr^L$: the square obtained from the above by passing to right adjoints is not a pullback, since $\ZZ[x]$ is sent to zero under the right adjoints to the top horizontal and left vertical arrows. It follows in particular that the assignment $R \mapsto  \Mod_{\Mod^\heartsuit_{R}}(\Pr^L)$ does not satisfy Zariski descent.
\end{example}

\begin{remark}\label{remark example pushout abelian}
If $R$ is a commutative ring then $\Mod_R^\heartsuit$ is a self dual object of $\Groth_1$. It follows that the forgetful functor $\Groth_{1, R} \rightarrow \Groth_1$ admits a right adjoint,  and in particular it  preserves all limits and colimits that exist in its source. It follows that example \ref{example pushout abelian} also supplies an example of a pushout square in $\Groth_1$ which is not preserved by the inclusion into $\Pr^L$.
\end{remark}

\begin{proposition}\label{proposition inclusion preserves limits that exist abelian}
The inclusion $\Groth_1 \rightarrow \Pr^L$  preserves all limits and all Hom objects that exist in $\Groth$.
\end{proposition}
\begin{proof}
We give the proof in the case of limits - the proof for Hom objects is completely analogous. Let $\ccal$ be the limit in $\Groth_1$ of a diagram $\ccal_\alpha$. Since $\Pr^L$ is generated under colimits by presheaf categories, to prove the proposition it will suffice to show that for every small category $\Ical$ the map 
\[
\Hom_{\Pr^L}(\Pcal(\Ical), \ccal) \rightarrow \lim \Hom_{\Pr^L}(\Pcal(\Ical), \ccal_\alpha)
\]
is an isomorphism. The above is equivalent to the map
\[
\Hom_{\Pr^L}(\Pcal(\Ical) \otimes \Ab, \ccal) \rightarrow \lim \Hom_{\Pr^L}(\Pcal(\Ical) \otimes \Ab, \ccal_\alpha).
\]
The proof finishes by observing that $\Pcal(\Ical) \otimes \Ab = \Fun(\Ical^\op, \Ab)$ is a Grothendieck abelian category.
\end{proof}

\begin{example}\label{example lack of limits abelian}
Consider the inverse system of Grothendieck abelian categories
\[
\rightarrow \ldots \Mod_{\ZZ[x]/(x^3)}^\heartsuit \rightarrow \Mod_{\ZZ[x]/(x^2)}^\heartsuit  \rightarrow \Mod_{\ZZ[x]/(x)}^\heartsuit  
\]
where the transitions are the extension of scalars functors. The limit of the above diagram in $\Mod_{\Ab}(\Pr^L)$ is the category of $(x)$-adically complete $\ZZ[x]$-modules, which is known not to be abelian. It follows from proposition \ref{proposition inclusion preserves limits that exist abelian} that the above sequence does not admit a limit in $\Groth_1$.
\end{example}

 \begin{example}
 Let $\ccal$ be the full subcategory of $\Mod^\heartsuit_{\ZZ[x]}$ on the modules for which the action of $x$ is locally nilpotent. Then $\Fun^L(\ccal, \Ab)$ is the category of $(x)$-adically complete $\ZZ[x]$-modules, which is not abelian. It follows from proposition \ref{proposition inclusion preserves limits that exist abelian} that $\Groth$ does not admit a Hom object from $\ccal$ to $\Ab$.
 \end{example}
 
\begin{remark}\label{remark groth1 not colimits}
Since $\Groth_1$ does not have all small limits, it cannot be a localization of $\Mod_{\Ab}(\Pr^L)$. We note that $\Mod_{\Ab}(\Pr^L)$ is generated under small colimits by  objects that belong to $\Groth_1$ (namely, categories of the form $\Pcal(\Ical) \otimes \Ab$ for some small category $\Ical$). It follows from this that $\Groth_1$ does not admit all small colimits: otherwise, the left adjoint to its inclusion into $\Mod_{\Ab}(\Pr^L)$ would be defined on all objects. Since $\Groth_1$ admits small coproducts (by proposition \ref{prop limits and colimits abelian}) we see that $\Groth_1$ does not admit all pushouts, and it also does not admit all geometric realizations.
\end{remark}

%%%%%%%%%%%%%%%%%%%%%%%%%%%%%%%%%%%%%%%%%%%%%%%%%%%%%%%%%%%%%%%%%%%%%%%%
%%%%%%%%%%%%%%%%%%%%%%%%%%%%%%%%%%%%%%%%%%%%%%%%%%%%%%%%%%%%%%%%%%%%%%%%
%%%%%%%%%%%%%%%%%%%%%%%%%%%%%%%%%%%%%%%%%%%%%%%%%%%%%%%%%%%%%%%%%%%%%%%%
%%%%%%%%%%%%%%%%%%%%%%%%%%%%%%%%%%%%%%%%%%%%%%%%%%%%%%%%%%%%%%%%%%%%%%%%
%%%%%%%%%%%%%%%%%%%%%%%%%%%%%%%%%%%%%%%%%%%%%%%%%%%%%%%%%%%%%%%%%%%%%%%%
%%%%%%%%%%%%%%%%%%%%%%%%%%%%%%%%%%%%%%%%%%%%%%%%%%%%%%%%%%%%%%%%%%%%%%%%

\subsection{Prestable examples}

We now discuss the case of $\Groth$.

\begin{example}
As in example \ref{example pushout abelian}, the square of $\ZZ[x]$-linear Grothendieck prestable categories
\[
\begin{tikzcd}
\Mod^\cn_{\ZZ[x]}  & \Mod^\cn_{\ZZ[x, x^{-1}] } \arrow{l}{} \\
\Mod^\cn_{\ZZ[x, (x-1)^{-1}]} \arrow{u}{} &\arrow{l}{} \arrow{u}{} \Mod^\cn_{ \ZZ[x, x^{-1}, (x-1)^{-1}]}
\end{tikzcd}
\]
is a pushout in $\Groth_{\ZZ[x]}$ but not in $\Mod_{\Mod_{\ZZ[x]}^\cn}(\Pr^L)$. This example also shows that the assignment $R \mapsto \Mod_{\Mod_R^\cn}(\Pr^L)$ does not satisfy Zariski descent. An adaptation of the argument from remark \ref{remark example pushout abelian} shows that the above square is in fact also a pushout square in $\Groth$ which is not preserved by the inclusion into $\Pr^L$.
\end{example}

\begin{example}
Let $\ccal$ be a compactly generated presentable additive category which is not Grothendieck prestable (for instance $\ccal = \Ab$). Let $\ccal^\omega$ be the full subcategory of $\ccal$ on the compact objects, which we regard as an object in the category $\Cat_{\text{rex}, \text{id}}$ of small idempotent complete categories with finite colimits and functors which preserve finite colimits. The forgetful functor $\Cat_{\text{rex}, \text{id}} \rightarrow \Cat$ is monadic, and in particular we may write $\ccal^\omega$ as a colimit in $\Cat_{\text{rex}, \text{id}}$ of its Bar resolution $\Bar(\ccal^\omega)_\bullet$. It follows that $\ccal$ is the geometric realization of $\Ind(\Bar(\ccal^\omega)_\bullet) \otimes \Sp^\cn$ in $\Pr^L$. We note that the transitions in $\Ind(\Bar(\ccal^\omega)_\bullet) \otimes \Sp^\cn$ are compact functors. Furthermore, for each $n \geq 0$ we have an equivalence 
\[
\Ind(\Bar(\ccal^\omega)_{n}) \otimes \Sp^\cn =  \Pcal(\Bar(\ccal^\omega)_{n-1}) \otimes \Sp^\cn = \Fun((\Bar(\ccal^\omega)_{n-1}^\op, \Sp^\cn)
\]
where here in the case $n = 0$ we set $\Bar(\ccal^\omega)_{-1} = \ccal^\omega$. It follows that $\Ind(\Bar(\ccal^\omega)_\bullet) \otimes \Sp^\cn$ is a diagram in $\Groth^\cnorm$ whose colimit in $\Pr^L$ is not Grothendieck prestable. In particular, the inclusion $\Groth^\cnorm \rightarrow \Pr^L$ does not preserve small colimits.
\end{example}

\begin{proposition}\label{proposition inclusion preserves limits that exist prestable}
The inclusion $\Groth  \rightarrow \Pr^L$  preserves all limits and all Hom objects that exist in $\Groth$.
\end{proposition}
\begin{proof}
Analogous to the proof of proposition \ref{proposition inclusion preserves limits that exist abelian}.
\end{proof}

\begin{example}\label{example non existence dual prestable}
Let $X$ be a non affine quasi-compact algebraic space with affine diagonal. Our goal in this example is to show that $\Groth$ does not admit a Hom object from $\QCoh(X)^\cn$ to $\Sp^\cn$. By virtue of proposition \ref{proposition inclusion preserves limits that exist prestable}, this is equivalent to the claim that the category $\Fun^L(\QCoh(X)^\cn, \Sp^\cn)$ of colimit preserving functors from $\QCoh(X)^\cn$ to $\Sp^\cn$ is not Grothendieck prestable. Let $G: \Mod_{\QCoh(X)^\cn}(\Groth) \rightarrow \Groth$ be the forgetful functor. Then $G$ embeds inside the forgetful functor $\Mod_{\QCoh(X)^\cn}(\Pr^L) \rightarrow \Pr^L$, which admits a right adjoint sending $\Sp^\cn$ to a $\QCoh(X)^\cn$-module with underlying category $\Fun^L(\QCoh(X)^\cn, \Sp^\cn)$. Our claim will then follow if we show that $G$ does not admit a right adjoint at $\Sp^\cn$. To do so it is enough to construct a diagram $\Dcal_\alpha$ in $ \Mod_{\QCoh(X)^\cn}(\Groth)$ with colimit $\Dcal$ with the property that $\Hom_{\Groth}(G(\dcal), \Sp^\cn) \neq \lim \Hom_{\Groth}(G(\Dcal_\alpha), \Sp^\cn)$.

Let $p: U \rightarrow X$ be a faithfully flat \'etale morphism with $U$ affine and consider the \v{C}ech nerve $U_\bullet$ of $p$. There is an induced cosimplicial object $\ccal^\bullet$ of $\twoQCoh^{\pst}(X)$ such that for every $n \geq 0$ the sheaf $\ccal^n$ is the pushforward along the map $U_n \rightarrow X$ of the structure sheaf $\QCoh^\cn_{U_n}$. The fact that the diagonal of $X$ is affine implies that all maps in $U_\bullet$ are affine, and hence we may construct a simplicial object $\ccal_\bullet$ in  $\twoQCoh^{\pst}(X)$ by passing to right adjoints the maps in $\ccal^\bullet$. Note that $\ccal^\bullet$ comes equipped with a coaugmentation from $\QCoh_X^\cn$, which after passing to right adjoints induces an augmentation $\ccal_\bullet \rightarrow \QCoh_X^\cn$. Since the base change of $U_\bullet$ along $p$ is split with geometric realization $U$ we see that $p^*\ccal^\bullet$ is split with totalization $p^*\QCoh_X^\cn$, and after passing to right adjoints we deduce that $p^*\ccal_\bullet$ is split with geometric realization $p^*\QCoh_X^\cn$. Since $\twoQCoh^{\pst}$ satisfies \'etale descent (\cite{SAG} theorem D.4.1.2) we conclude that $\QCoh_X^\cn$ is the geometric realization of $\ccal_\bullet$.

Applying \cite{SAG} theorem 10.2.0.2 we deduce that $\QCoh(X)^\cn$ is the geometric realization in $\Mod_{\QCoh(X)^\cn}(\Groth)$ of $\Gamma^\enh(X, \ccal_\bullet)$ (where here $\Gamma^\enh(X, -)$ is defined in a way similar to the functor $\Gamma^\enh_\comp(X, -)$ from notation \ref{notation gamma enh}). Unwinding the definitions, we have that $G(\Gamma^\enh(X, \ccal_\bullet))$  is the simplicial Grothendieck prestable category $\QCoh(U_\bullet)^\cn_*$ obtained from $U_\bullet$ by applying the functor $\QCoh^\cn$ in its pushforward functoriality. The functors of global sections $\QCoh(U_\bullet)^\cn \rightarrow \Sp^\cn$ assemble together into an augmentation $\QCoh(U_\bullet)^\cn_* \rightarrow \Sp^\cn$. To prove our claim it suffices to show that this does not extend to a map  $\QCoh(X)^\cn \rightarrow \Sp^\cn$.   Stabilizing, we may reduce to showing that the augmentation $\QCoh(U_\bullet)_* \rightarrow \Sp$ induced from the global section functors $\QCoh(U_n) \rightarrow \Sp$ does not extend to a right t-exact colimit preserving functor $\QCoh(X) \rightarrow \Sp$. Since $X$ is not affine, the functor of global sections $\QCoh(X) \rightarrow \Sp$ is not right t-exact (\cite{SAG} proposition 9.6.5.1). It is now enough to prove that $\QCoh(X)$ is the geometric realization in $\Pr^L$ of $\QCoh(U_\bullet)_*$. 

Consider the category $\twoQCoh^{\St}(X)$ of quasicoherent sheaves of presentable stable categories on $X$ (which we may identify with the full subcategory of $\twoQCoh^{\pst}(X)$ on those sheaves whose value on each affine scheme over $X$ is stable). As before, from $U_\bullet$ we may construct an augmented simplicial object $\Ccal'_\bullet \rightarrow \QCoh_X$ where for each $n \geq 0$ the sheaf $\Ccal'_n$ is the pushforward along the map $U_n \rightarrow X$ of the stable structure sheaf $\QCoh_{U_n}$. This is a colimit diagram, and therefore we have that $\QCoh(X)$ is the geometric realization in $\Mod_{\QCoh(X)}(\Pr^L)$ of  $\Gamma^\enh(X, \ccal'_\bullet)$. Our claim now follows from the fact that the image of this augmented simplicial diagram under the (colimit preserving) forgetful functor $\Mod_{\QCoh(X)}(\Pr^L) \rightarrow \Pr^L$ recovers the augmented simplicial object $\QCoh(U_\bullet)_* \rightarrow \QCoh(X)$.
\end{example}

\begin{example}\label{example non existence limits prestable}
Let $\ccal$ be a Grothendieck prestable category such that $\Fun^L(\ccal, \Sp^\cn)$ is not Grothendieck prestable (for instance, we may take $\ccal = \QCoh(X)^\cn$ for any non affine quasi-compact algebraic space with affine diagonal, see example \ref{example non existence dual prestable}). Let $\kappa$ be a regular cardinal such that $\ccal$ is $\kappa$-compactly generated and let $\ccal^\kappa$ be the full subcategory of $\ccal$ on the $\kappa$-compact objects, which we regard as an object in the category $\Cat_{\kappa}$ of small categories with $\kappa$-small colimits and $\kappa$-small colimit preserving functors. The forgetful functor $\Cat_{\kappa} \rightarrow \Cat$ is monadic, and in particular we may write $\ccal^\kappa$ as the colimit in $\Cat_\kappa$ of its Bar resolution $\Bar(\ccal^\kappa)_\bullet$. It follows that $\ccal$ is the geometric realization of $\Ind_{\kappa}(\Bar(\ccal^\kappa)_\bullet)$ in $\Pr^L$. We now have
\[
\Fun^L(\ccal, \Sp^\cn) = \lim \Fun^L(\Ind_{\kappa}(\Bar(\ccal^\kappa)_\bullet), \Sp^\cn).
\]
We note that $\Ind_{\kappa}(\Bar(\ccal^\kappa)_n)$ is a presheaf category for all $n$, and therefore the right hand side is a limit of Grothendieck prestable categories along colimit preserving functors.  It now follows from  proposition \ref{proposition inclusion preserves limits that exist prestable} that the cosimplicial diagram $\Fun^L(\Ind_{\kappa}(\Bar(\ccal^\kappa)_\bullet), \Sp^\cn)$ does not admit a totalization in $\Groth$.
\end{example}
 
\begin{remark}
As in remark \ref{remark groth1 not colimits}, the fact that $\Groth$ does not admit small limits implies that the inclusion $\Groth \rightarrow \Mod_{\Sp^\cn}(\Pr^L)$ does not admit a left adjoint, and consequently $\Groth$ does not admit small colimits. Combined with proposition \ref{prop colimits y limits} (which guarantees the existence of small coproducts) we deduce that $\Groth$ does not admit geometric realizations nor pushouts.
\end{remark}

%%%%%%%%%%%%%%%%%%%%%%%%%%%%%%%%%%%%%%%%%%%%%%%%%%%%%%%%%%%%%%%%%%%%%%%%
%%%%%%%%%%%%%%%%%%%%%%%%%%%%%%%%%%%%%%%%%%%%%%%%%%%%%%%%%%%%%%%%%%%%%%%%
%%%%%%%%%%%%%%%%%%%%%%%%%%%%%%%%%%%%%%%%%%%%%%%%%%%%%%%%%%%%%%%%%%%%%%%%
%%%%%%%%%%%%%%%%%%%%%%%%%%%%%%%%%%%%%%%%%%%%%%%%%%%%%%%%%%%%%%%%%%%%%%%%
%%%%%%%%%%%%%%%%%%%%%%%%%%%%%%%%%%%%%%%%%%%%%%%%%%%%%%%%%%%%%%%%%%%%%%%%
%%%%%%%%%%%%%%%%%%%%%%%%%%%%%%%%%%%%%%%%%%%%%%%%%%%%%%%%%%%%%%%%%%%%%%%%

\ifx\inmain\undefined
\bibliographystyle{myamsalpha}
\bibliography{References}
\fi